\documentclass{article}

\usepackage[twoside,
paperwidth=210mm,
paperheight=297mm,
textheight=622pt,
textwidth=468pt,
centering,
headheight=50pt,
headsep=12pt,
footskip=18pt,
footnotesep=24pt plus 2pt minus 12pt,
columnsep=2pc]{geometry}

\usepackage{tikz}
\usetikzlibrary{calc}
\usepackage{makecell}

\usepackage[numbers,sort&compress]{natbib}
\usepackage[auth-sc]{authblk}

\usepackage{mathtools}

\allowdisplaybreaks

\usepackage{amssymb}
\usepackage[mathscr]{eucal}

\usepackage{mismath}

\usepackage{hyperref}

\usepackage{bm} 

\usepackage{graphicx}
\graphicspath{{figs/}}

\usepackage{tikz}
\usetikzlibrary{math}
\usetikzlibrary{calc}

\usepackage{subcaption}

\usepackage{multirow}

\usepackage{makecell}
\setcellgapes{2pt}

\usepackage{booktabs}

\usepackage{siunitx}
\usepackage{algorithm}
\usepackage{algpseudocode}

\usepackage[skins]{tcolorbox}
\newtcolorbox{stepbox}[2][]{%
  enhanced,
  attach boxed title to top center={yshift=-3mm,yshifttext=-1mm},
  colframe=blue!75!black,
  colbacktitle=red!80!black,
  fonttitle=\bfseries,
  title=#2,#1
}

\usepackage{amsthm, amsmath, amssymb}

\newcommand{\btau}{\underline{\bm{\tau}}}

\newcommand{\bedis}[2]{\mathbf{e}^n_{#1,#2}}
\newcommand{\ube}[2]{\underline{\mathbf{e}}^n_{#1,#2}}
\newcommand{\inner}[2]{\left(#1,#2\right)}

\usepackage[capitalise]{cleveref}

\newtheorem{theorem}{Theorem}[section]
\newtheorem{lemma}[theorem]{Lemma}
\newtheorem{remark}[theorem]{Remark}

\theoremstyle{definition}

\theoremstyle{remark}

\providecommand{\keywords}[1]{\textbf{\textit{Keywords---}} #1} 

\crefname{equation}{}{}

\usepackage{lineno}
\usepackage[numbers,sort&compress]{natbib}

\numberwithin{figure}{section}
\numberwithin{table}{section}
\numberwithin{algorithm}{section}

\title{Iterative Contact-resolving Hybrid Methods for Multiscale Contact Mechanics}
\author[1]{Eric T. Chung}
\author[2]{Hyea Hyun Kim}
\author[1]{Xiang Zhong\thanks{Corresponding author.
(Email address: \href{mailto:xzhong@math.cuhk.edu.hk}{xzhong@math.cuhk.edu.hk})}}
\affil[1]{Department of Mathematics, The Chinese University of Hong Kong, Shatin, Hong~Kong~SAR, China.}
\affil[2]{Department of Applied Mathematics, and Institute of Natural Sciences, Kyung Hee University, Yongin, Republic of Korea.}
\date{}
\begin{document}
\maketitle

\begin{abstract}
  Modeling contact mechanics with high contrast coefficients presents significant mathematical and computational challenges, especially in achieving strongly symmetric stress approximations for mixed formulations. Due to the inherent nonlinearity of contact problems, conventional methods that treat the entire domain as a monolithic system often lead to high global complexity. To address this, we develop an iterative contact-resolving hybrid method by localizing nonlinear contact constraints within a smaller subdomain, while the larger subdomain is governed by a linear system. Our system employs variational inequality theory, minimization principles, and penalty methods. More importantly, we propose four discretization types within the two-subdomain framework, ranging from applying standard/mixed FEM across the entire domain to combining standard/mixed multiscale methods in the larger subdomain with standard/mixed FEM in the smaller one. 
By employing a multiscale reduction technique, the method avoids excessive degrees of freedom inherent
in conventional methods in the larger domain, while the mixed formulation enables direct stress computation, ensures local momentum conservation, and resists locking in nearly incompressible materials.
Convergence analysis and the corresponding algorithms are provided for all cases. Extensive numerical experiments are presented to validate the effectiveness of the approaches.
\end{abstract}


\section{Introduction}
Contact mechanics with high contrast coefficients arise in numerous engineering and geophysical applications where materials with significantly different mechanical properties interact. Typical examples include rubber seals pressing against metal surfaces, tire-road contact, and geological faults between dissimilar rock strata.
Such problems are characterized by a large disparity in material parameters (e.g., Young's modulus, Lam\'{e} coefficients) across the contacting interfaces, leading to complex, localized deformation patterns and challenging numerical simulation. 
The high contrast coefficients coupled with nonlinear contact conditions poses significant challenges, including solution ill-conditioning and boundary layers, which require robust discretization and solution techniques.
Kikuchi and Oden \cite{kikuchi1988contact} established a foundational framework for elasticity contact problems using variational inequalities. 
Computational methods essential for solving these nonlinear and potentially ill-conditioned systems are discussed in modern works by Wriggers \cite{wriggers2006computational, wriggers2006analysis, wriggers2008formulation}, Han \cite{han2019numerical, han2002quasistatic}, Haslinger \cite{haslinger2022numerical} and Laursen \cite{laursen2003computational}. An overview of recent developments for contact problems is provided by Chouly and Hild et al. \cite{chouly2018overview, chouly2023finite}. Further advancements in this area can be found in more recent studies such as \cite{anaya2025nitsche, burman2020nitsche, burman2023augmented, epalle2025parallel, gustafsson2025finite, gustafsson2020nitsche, wang2024mixed}. Contact constraints can be enforced via various methods, including Lagrange multipliers \cite{heintz2006stabilized, wang2024mixed}, penalty approaches \cite{gustafsson2025finite, hu2024mixed, kikuchi1981penalty, oden1982penalty,yastrebov2013numerical}, augmented Lagrangian techniques \cite{burman2019augmented, burman2023augmented}, and mortar formulations \cite{puso2004mortar}. 
Many numerical analyses have been developed for elasticity contact problems, such as optimized Schwarz methods \cite{xu2010spectral}, mixed finite element method (FEM) \cite{hu2024mixed,coorevits2002mixed,belgacem2005mixed,gallego1989mixed,alart1991mixed}, augmented Lagrangian method \cite{zhang2015boundary,simo1992augmented,burman2019augmented, burman2023augmented}, Nitsche finite element method \cite{burman2017penalty, chouly2015nitsche, chouly2015nitsche2, wriggers2008formulation}, least squares method \cite{attia2009first}, discontinuous Petrov–Galerkin methods (DPG) \cite{fuhrer2018dpg} and so on.   


 The resulting system of equations in contact mechanics is typically a large-scale, nonlinear, and ill-conditioned problem, especially in cases of high material contrast. To tackle these challenges, we propose an efficient iterative contact-resolving hybrid framework for the penalized contact problem. This approach employs a Robin boundary condition as the transmission condition, combined with a derivative-free technique for updating transmission data on the interfaces (see \cite{deng1997timely, deng2003nonoverlapping}). In our setting, the contact boundary is contained entirely within the smaller subdomain and combined with penalty approach, while the larger subdomain allows for flexible discretization choices. For instance, the constraint energy minimizing generalized multiscale finite element method (CEM-GMsFEM) \cite{chung2023multiscale, chung2025locking} can be employed to significantly reduce computational costs associated with high-contrast features in the larger domain, while still preserving coarse-mesh convergence. More specifically, we introduce four types of discretizations within the iterative contact-resolving hybrid framework. These include combinations where both subdomains use standard/mixed finite element methods (FEM), or where the larger subdomain uses standard/mixed CEM-GMsFEM and the smaller one uses standard/mixed FEM. Corresponding algorithms are provided for each case, forming efficient strategies for designing scalable solvers and preconditioners.

Our main contributions are threefold. Firstly, we develop a novel iterative contact-resolving hybrid method for multiscale contact mechanics that isolates nonlinear contact constraints within a smaller subdomain, while the larger subdomain retains a linear system (see Algorithms \ref{alg:B}-\ref{alg:BB}). Conventional methods typically treat the nonlinear contact problem across the entire domain, leading to a globally complex system. In our approach, the nonlinear subproblem, which is restricted to a very small region with a width of one coarse mesh element, is efficiently solved by means of a semismooth Newton method \cite{hintermuller2002primal}. This localized treatment significantly shortens iteration time. Meanwhile, the larger subdomain supports flexible discretization choices without the need to account for nonlinearity.
Secondly, we introduce mixed formulations (see Algorithms \ref{an iterative dd algorithm for continuous setting in mixed form} and \ref{alg:BB}) for solving linear elasticity contact problems involving highly heterogeneous and high-contrast coefficients. By adopting a stress-displacement mixed formulation, our method enables direct computation of stress fields, avoiding post-processing and enforcing local momentum conservation for enhanced physical consistency. Moreover, this formulation inherently prevents Poisson’s ratio locking, enabling robust simulations of nearly incompressible materials.
Lastly, we incorporate the CEM-GMsFEM framework within the larger subdomain (see Algorithms \ref{alg:A}-\ref{alg:BB}). It is well-known that obtaining strongly symmetric stress approximations in mixed methods do not come easy. While established mixed methods \cite{adam2002,arnold2002,arnold2008,DFLS2008,zhong2023spectral} require excessive degrees of freedom to enforce such symmetry, our multiscale model reduction technique substantially lowers computational cost while preserving accuracy.

We emphasize that the primary novelty of our work lies in introducing a hybrid methodology for addressing multiscale contact mechanics—an approach that, to the best of our knowledge, has not been fully explored in the existing literature, such as \cite{anaya2025nitsche, attia2009first, belgacem2005mixed, burman2017penalty, burman2020nitsche, burman2019augmented, burman2023augmented, chouly2018overview, chouly2015nitsche, chouly2015nitsche2, epalle2025parallel, gustafsson2020nitsche, wang2024mixed}. Within this hybrid framework, nonlinear contact behavior is effectively handled. Our methodology differs notably from that described in \cite{deng1997timely, deng2003optimal, deng2003nonoverlapping}. It should be noted that \cite{deng2003nonoverlapping} considers a nonoverlapping domain decomposition method applied to a simple second-order linear elliptic problem, in which the transmission condition can be updated easily due to its linear nature. In contrast, our work focuses on efficiently resolving nonlinear contact constraints, where updating the transmission condition along the interfaces poses a nontrivial challenge. This difficulty arises from the need to compute solutions associated with the contact area through a nonlinear iterative sub-procedure. Consequently, the core data-updating step in our algorithm is closely tied to the inherent nonlinearity of the contact constraints, rendering it considerably more complex.

This paper is structured as follows. In Section \ref{sec: Preliminaries}, we introduce the model problem and some notation. The contact problem in elasticity is reformulated using a minimization theorem and penalty methods. Section \ref{various discretizations} presents an iterative contact-resolving hybrid framework along with various discretization schemes. Specifically,  the methods based on standard FEM and mixed FEM formulations in the whole domain are detailed in Sections \ref{Nonoverlapping domain decomposition method associated with Standard FEM formulation} and \ref{Nonoverlapping domain decomposition method associated with Mixed FEM}, respectively, while those incorporating standard and mixed CEM-GMsFEM in the larger subdomain are discussed in Sections \ref{Nonoverlapping domain decomposition method associated with standard CEM-GMsFEM} and \ref{Nonoverlapping domain decomposition method associated with mixed CEM-GMsFEM}.  Corresponding algorithms are provided for all four cases (i.e. Algorithms \ref{alg:B}-\ref{alg:BB}). In Section \ref{Analysis}, we give some analyses for the convergence of the iterative contact-resolving hybrid methods proposed in Section  \ref{various discretizations}. 
Section \ref{Numerical experiments} reports numerical experiments on two test models to demonstrate the performance of the proposed method. Finally, a summary of the conclusions is presented in Section \ref{conclusions}.

\section{Model problem}\label{sec: Preliminaries}
We consider the elasticity contact problem in the domain $\Omega\subset\mathbb{R}^n$:
\begin{subequations}
\label{model_problem}
    \begin{align}
        \mathcal{A} \underline{\bm{\sigma}} &= \underline{\bm{\epsilon}}(\mathbf{u}) \quad &&\text{in } \Omega, \label{model_problem_d}\\
        -\nabla\cdot(\underline{\bm{\sigma}}(\mathbf{u})) &= \mathbf{f} \quad &&\text{in } \Omega, \label{model_problem_a} \\
        \mathbf{u} &= \mathbf{0} \quad &&\text{on } \Gamma_D, \label{model_problem_b} \\
        \mathbf{u}\cdot\mathbf{n}_c\leq0,\quad \mathbf{n}_c\cdot(\underline{\bm{\sigma}}\cdot\mathbf{n}_c)\leq0,\quad \big(\mathbf{n}_c\cdot(\underline{\bm{\sigma}}\cdot\mathbf{n}_c)\big) \mathbf{u}\cdot\mathbf{n}_c&=0  \quad &&\text{on } \Gamma_C, \label{model_problem_c}
    \end{align}	
\end{subequations}
where
$\underline{\bm{\sigma}}:\Omega\to\mathbb{R}^{n\times n}$ is the symmetric stress tensor, $\mathbf{u}$ denotes the displacement and $\underline{\bm{\epsilon}}(\mathbf{u})=\frac{1}{2}(\nabla\mathbf{u}+(\nabla\mathbf{u})^{\mathrm{T}})$ the linearized strain tensor. $\mathbf{n}_c$ is the unit outward normal vector to $\Gamma_C$. $\Omega\subset \mathbb{R}^n$ 
($n=2,3$) is a bounded and connected Lipschitz polyhedral domain occupied by an isotropic and linearly elastic solid. {$\mathcal{A}$ is the inverse of the elasticity operator, which is given by
$\mathcal{A}\underline{\bm{\tau}} := \frac{1}{2\mu}\,\underline{\bm{\tau}} - \frac{\lambda}{2\mu(3\lambda+2\mu)} (\operatorname{tr}\underline{\bm{\tau}}) \boldsymbol{I}$ in three dimension.
Here $\lambda$ and $\mu$ are the Lam$\rm \acute{e}$ coefficients. The same expression applies to the 2D plane parts of stress and strain, in the case of plane stress. Note that in the plane strain, the in-plane constitutive relation differs (e.g., with the denominator $  2\lambda + 2\mu$). $\bm{\mathit{I}}$ is the identity matrix of $\mathbb{R}^{n\times n}$.
For nearly incompressible materials, $\lambda$ is large in comparison with $\mu$. In this paper, $\lambda$ and $\mu$ are highly heterogeneous in space and possibly high contrast.

The domain $\Omega$ can be decomposed into two subdomains based on the contact boundary $\Gamma_C$: a smaller subdomain that contains $\Gamma_C$, and a larger subdomain that is free of it. $\partial\Omega$ is the boundary of the domain $\Omega$, which admits
a disjoint partition $\partial\Omega=\Gamma_D\cup\Gamma_C$. Note that $\Gamma_D$ has a positive measure. For simplicity, we consider a two dimensional unilateral contact model as depicted in Figure \ref{2D model}. In this case, $\bar{\Omega}=\bar{\Omega}_1 \cup \bar{\Omega}_2=\Omega_1 \cup \Omega_2 \cup \gamma \cup \partial \Omega$,
$
\gamma=\partial \Omega_1 \cap \partial \Omega_2$. $\Omega_2$ is the smaller subdomain that contains $\Gamma_C$ while the larger domain $\Omega_1$ devoid of it.
 Denote $\mathbf{n}_i$ as the unit outward normal vector to the boundary of $\Omega_i$ for $i=1,2$. To localize the nonlinear contact constraints within an extremely small region, we define $\Omega_2$ to be the width of a single coarse mesh.
We further assume frictionless contact, i.e.,
$
\mathbf{t}_c\cdot(\underline{\bm{\sigma}}\cdot\mathbf{n}_c)=0$ on $\Gamma_C$,
where $\mathbf{t}_c$ is the unit tangential vector to $\Gamma_C$.
To simplify the notation, we define the normal displacement $u_c = \mathbf{u} \cdot \mathbf{n}_c$ and the normal stress $\sigma_c = \mathbf{n}_c \cdot (\underline{\bm{\sigma}} \cdot \mathbf{n}_c)$. Then the contact boundary condition (\ref{model_problem_c}) can be reformulated as
\[
u_c\leq0,\quad \sigma_c\leq 0,\quad\sigma_cu_c=0\qquad  \text{on } \Gamma_C.
\]

\begin{figure}[tbph] 
		\centering 
		\includegraphics[height=4cm,width=9cm]{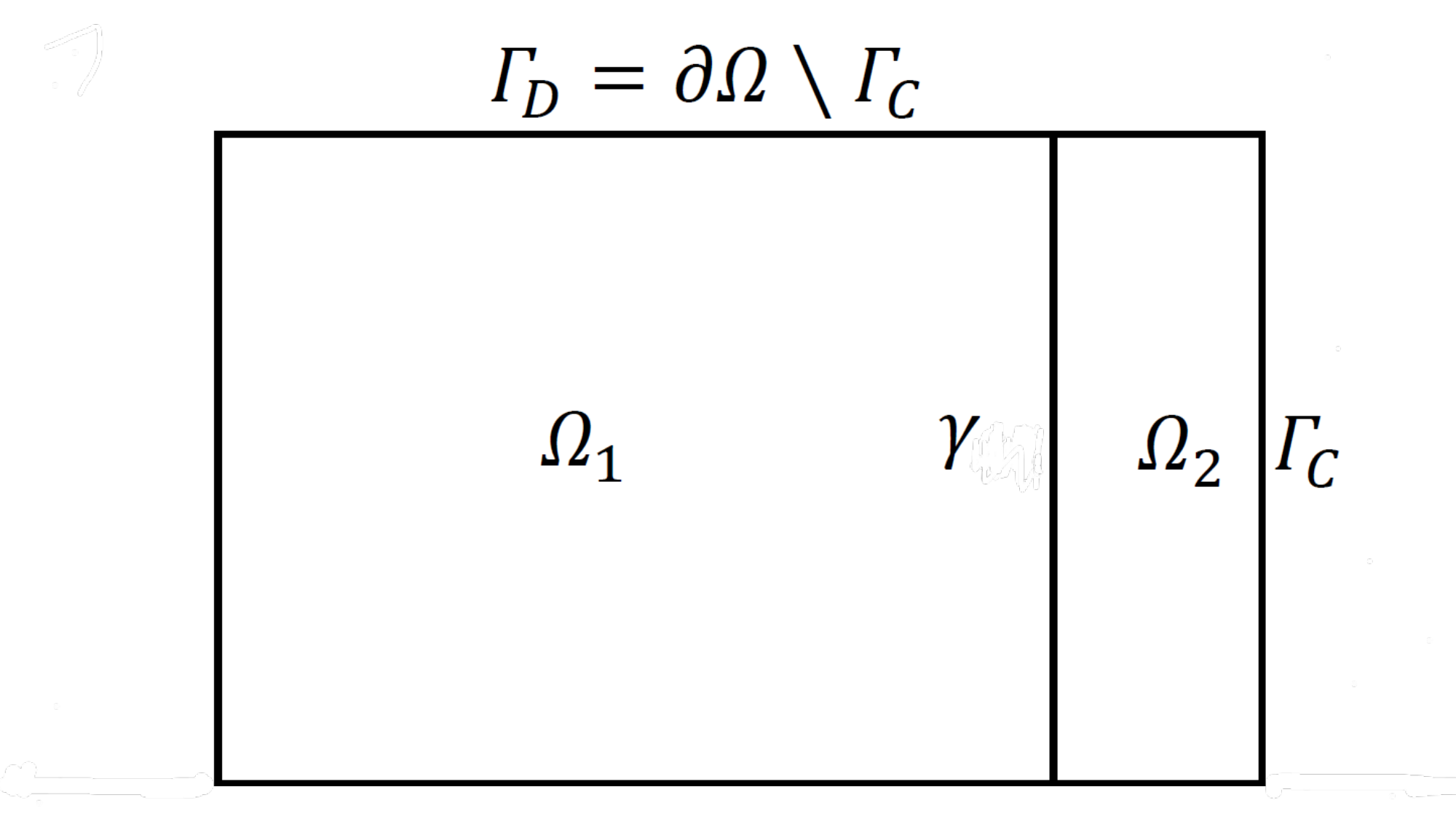} 
		\caption{Simplified 2D model}
		\label{2D model}
\end{figure}
We define the following subsets of the Sobolev space $\bm{\mathit{H}}^1(\Omega)$ of vector valued functions: 
\[
\bm{\mathit{H}}^1_{\Gamma_D}(\Omega):=\{\mathbf{v}\in\bm{\mathit{H}}^1(\Omega)\mid \mathbf{v}=\mathbf{0}\text{ on }\Gamma_D\},
\quad
\mathbf{V}:=\{\mathbf{v}\in\bm{\mathit{H}}^1_{\Gamma_D}(\Omega)\mid v_c \leq 0 \text{ a.e. on }\Gamma_C\}.
\]
Following \cite{PG2013}, the exact solution to problem (\ref{model_problem}) is given by the variational inequality: find $\mathbf{u} \in \mathbf{V}$ such that
\begin{align}
\label{variational inequality}
a(\mathbf{u},\mathbf{v}-\mathbf{u})\geq (\mathbf{f},\mathbf{v}-\mathbf{u})
\end{align}
for all $\mathbf{v}\in\mathbf{V}$, where $a(\mathbf{w},\mathbf{v})=\int_\Omega\mathcal{A}^{-1}\underline{\bm{\epsilon}}(\mathbf{u})\colon\underline{\bm{\epsilon}}. (\mathbf{v})dx$ and $(\mathbf{f},\mathbf{v})=\int_\Omega\mathbf{f}\cdot\mathbf{v}dx$. And we define $(\mathbf{w},\mathbf{v})_D=\int_D\mathbf{w}\cdot\mathbf{v}dx$ for all $D\subset\Omega$, where $D$ will be dropped if $D=\Omega$. According to \cite{PG2013}, the weak problem \eqref{variational inequality} is well-posed and admits a unique solution. 
Furthermore, the solution can be equivalently characterized as the minimizer of a constrained minimization problem. Specifically, given the functional $F: \bm{\mathit{H}}^1_{\Gamma_D}(\Omega)\to\mathbb{R}$ as $F(\mathbf{v})=\frac{1}{2}a(\mathbf{v},\mathbf{v})-(\mathbf{f},\mathbf{v})$, find $\mathbf{u}\in\mathbf{V}$ such that $$F(\mathbf{u})=\inf_{\mathbf{v}\in\mathbf{V}}F(\mathbf{v}).$$
Next, we employ the penalty method to transform the constrained minimization problem into an unconstrained one. This is achieved by introducing a penalty term into the objective function that effectively suppresses constraint violations. We denote this penalty term as
$P: \bm{\mathit{H}}^1_{\Gamma_D}(\Omega)\rightarrow \mathbb{R}$
and a new functional $F_{\delta}: \bm{\mathit{H}}^1_{\Gamma_D}(\Omega) \rightarrow \mathbb{R}$ depending on a real parameter $\delta > 0$ as follows
$$
F_{\delta}(\mathbf{v}) = F(\mathbf{v}) + \frac{1}{\delta} P(\mathbf{v}),
$$
where $
P(\mathbf{v})\coloneqq \frac{1}{2}\int_{\Gamma_{\mathrm{C}}} \abs{v_c^{+}}^2 \, \mathrm{d} s
$, $v_c^+=\max\{0,v_c\}$ and $v_c^-=-\min\{0,v_c\}$. Note that $v_c=v_c^+-v_c^-$, $-v_c^-\cdot v_c=\abs{v_c^-}^2$, $v_c^+\cdot v_c=\abs{v_c^+}^2$, $v_c^+v_c^-=0$. Then for all $\mathbf{v},\mathbf{w}\in\bm{\mathit{H}}^1_{\Gamma_D}(\Omega)$, we have
\begin{equation}
\label{an_important_inequality}
\begin{aligned}
(v_c^+ - w_c^+)(v_c - w_c) 
&= (v_c^+ - w_c^+)^2 + (v_c^+ - w_c^+)(w_c^- - v_c^-) \\
&= (v_c^+ - w_c^+)^2 + v_c^- w_c^+ +w_c^- v_c^+ \geq (v_c^+ - w_c^+)^2 \geq 0.
\end{aligned}
\end{equation}
The penalty term $P$ is designed such that a candidate minimizer $\mathbf{v} \in \bm{\mathit{H}}^1_{\Gamma_D}(\Omega)$ incurs a larger penalty with increasing violation of the constraint $v_c \leq 0$ on $\Gamma_C$. Consequently, for every $\delta > 0$, the functional $F_{\delta}$ admits a minimizer $u_{\delta} \in V$. This leads to the constrained minimization problem: find $\mathbf{u} \in \bm{\mathit{H}}^1_{\Gamma_D}(\Omega)$ such that
$F_{\delta}(\mathbf{u}) = \inf_{\mathbf{v}\in\bm{\mathit{H}}^1_{\Gamma_D}(\Omega)} F_{\delta}(\mathbf{v}).$

\section{Iterative contact-resolving hybrid methods associated with various discretizations} \label{various discretizations}
In this section, we present various discretization schemes for the iterative contact-resolving hybrid framework. These include applying standard/mixed FEM across the entire domain (in Sections \ref{Nonoverlapping domain decomposition method associated with Standard FEM formulation} and \ref{Nonoverlapping domain decomposition method associated with Mixed FEM}), as well as a combined approach that couples standard/mixed CEM-GMsFEM in the larger domain with standard/mixed FEM in the smaller domain (in Sections \ref{Nonoverlapping domain decomposition method associated with standard CEM-GMsFEM} and \ref{Nonoverlapping domain decomposition method associated with mixed CEM-GMsFEM}).
\subsection{Iterative contact-resolving hybrid method associated with Standard FEM}\label{Nonoverlapping domain decomposition method associated with Standard FEM formulation}
This section firstly introduces the standard FEM formulation for the penalized contact problem, and then presents the iterative contact-resolving hybrid algorithm (Algorithm \ref{alg:B}). Convergence result for the algorithm is given in Theorem \ref{convergence in discrete setting}.

Provided that the functionals $F$ and $P$ (introduced in Section \ref{sec: Preliminaries}) are Gateaux-differentiable, by \cite[Chapter 6]{kikuchi1988contact}, the unconstrained minimization problem admits a characterization: find $\mathbf{u} \in \bm{\mathit{H}}^1_{\Gamma_D}(\Omega)$ such that
\begin{equation}
\label{weak form}
a(\mathbf{u},\mathbf{v})+\frac{1}{\delta}\int_{\Gamma_C}u_c^+\cdot v_cds=(\mathbf{f},\mathbf{v})\quad\forall \mathbf{v}\in \bm{\mathit{H}}^1_{\Gamma_D}(\Omega).
\end{equation}

Let $\mathbf{V}_h \subset \bm{\mathit{H}}^1(\Omega)$ be the the standard vector-valued piecewise linear finite element space ($\bm{\mathit{P}}_1$).
Define $\mathbf{V}_{h, \Gamma_D}$ to be the space consisting of the functions in $\mathbf{V}_h$ that takes the value $\mathbf{0}$ on $\Gamma_D$. Then the standard finite element method for (\ref{weak form}) is to find $\mathbf{u}_h \in \mathbf{V}_{h, \Gamma_D}$ such that (denote $v_{hc}=\mathbf{v}_h\cdot\mathbf{n}_c$ for all $\mathbf{v}_h \in \mathbf{V}_{h, \Gamma_D}$)
\begin{align}
\label{discrete form}
a\left(\mathbf{u}_h, \mathbf{v}_h\right)+ \frac{1}{\delta} \int_{\Gamma_C} \left(u_{hc}\right)^{+} \cdot v_{hc} \, ds = \left(\mathbf{f}, \mathbf{v}_h\right) \quad \forall \mathbf{v}_h \in \mathbf{V}_{h, \Gamma_D}.
\end{align}
More precisely, $\mathbf{V}_h = \operatorname{span}\left\{\bm{\varphi}_j^1,\bm{\varphi}_j^2\right\}_{j=1}^{N_o}$, where $N_o$ denotes the finite element nodal point set. In two dimensions, the first component of $\bm{\varphi}_j^1$ (for each $1\leq j\leq N_o$) is a nodal basis and the second component is 0; $\bm{\varphi}_j^2$ has the reverse structure.

Next we develop an iterative contact-resolving hybrid method for solving (\ref{discrete form}). 
Let $\mathbf{V}_{h,i} = \mathbf{V}_{h,\Gamma_D} |_{\Omega_i}$ for $i=1,2$. Let $N_\gamma = \{ p \in N_o \mid p \in \gamma \}$ denote the set of finite element nodal points lying on the interface $\gamma$, and
$
\mathbf{V}_h\left(\gamma\right) = \left\{ \mathbf{v} \in \mathbf{V}_h \mid \mathbf{v} = \sum_{p \in N_\gamma} (y^1_p \bm{\varphi}^1_p+y^2_p \bm{\varphi}^2_p), \, y^1_p, y^2_p \in \mathbb{R} \right\},
$
where $\left\{\bm{\varphi}^1_p,\bm{\varphi}^2_p\right\}_{p \in N_o}$ is the nodal basis at the point $p$ of the finite element space $\mathbf{V}_h$. We denote

\begin{equation}
\label{Newton-Cotes-integral}
\int_{\gamma}^* \bm{u}\cdot\bm{v} \, ds = \sum_{p \in N_\gamma} u_1(p) v_1(p) w^1_p+u_2(p)v_2(p) w^2_p,
\end{equation}
where $\mathbf{u}=(u_1,u_2)^{\mathrm{T}}$, $\mathbf{v}=(v_1,v_2)^{\mathrm{T}}$. $w^1_p,w^2_p$ are the weights at $p$ obtained from Composite Newton-Cotes Quadrature. The nodes of the Lagrange basis functions on $\gamma$ coincide with the quadrature nodes of the Composite Newton–Cotes rule (see e.g., \cite{hamming2012numerical,hildebrand1987introduction,trefethen2022exactness}). Therefore, the integrals on the interfaces are evaluated numerically via (\ref{Newton-Cotes-integral}). In the theoretical analysis we use the composite Newton–Cotes quadrature for its simple, node‑aligned structure; in actual computations, Gaussian quadrature can be employed for higher accuracy and efficiency.

To restrict the nonlinear contact constraints within a smaller subdomain while keeping the larger subdomain as a linear system, a Robin boundary condition like the form ``$\big(\mathcal{A}^{-1}\underline{\bm{\epsilon}}(\mathbf{u})\big)\cdot \mathbf{n} + \alpha \mathbf{u} = \mathbf{g}$" is employed as its transmission condition on the interface $\gamma$ ($\alpha>0$ is a transmission coeﬃcient, which will be used below). Its basic idea is to combine Dirichlet and Neumann transmission
condition as a Robin-type transmission condition, thereby transforming overdetermined subproblems into well-posed Robin boundary subproblems. This idea, originally proposed by Lions \cite{lions1990schwarz} and Tang \cite{tang1992generalized}, has been further developed in a number of later works \cite{deng1997timely, deng2003optimal, deng2003nonoverlapping, xu2010spectral, qin2006parallel}. These studies show that the convergence speed critically depends on the transmission coefficient $\alpha$, which has been assigned values like $1$, $h^{-1/2}$, or $h^{1/2}$ to examine convergence behavior.

In order to establish the convergence of Algorithm \ref{alg:B}, We first reformulate (\ref{discrete form}) into an equivalent split subproblem form as follows. We denote $\norm{\mathbf{v}}_{a_i}^2=a_i(\mathbf{v},\mathbf{v})\coloneqq\int_{\Omega_i}\mathcal{A}^{-1}\underline{\bm{\epsilon}}(\mathbf{v})\colon\underline{\bm{\epsilon}}(\mathbf{v})dx$ for $i=1,2$.
\begin{theorem}
\label{equivalent discrete form}
Let $\mathbf{u}_h \in \mathbf{V}_{h,\Gamma_D}$ be the solution of problem (\ref{discrete form}) and $\mathbf{u}_{h,i} = \left.\mathbf{u}_h\right|_{\Omega_i}$ for $i=1,2$. Then the problem (\ref{discrete form}) can be split into an equivalent splitting subproblem form. That is, there exist $\mathbf{g}_{12}^*,\mathbf{g}_{21}^* \in \mathbf{V}_h\left(\gamma\right)$ such that $\mathbf{u}_{h,i} \in \mathbf{V}_{h,i} \, (i=1,2)$ satisfies
$$
a_1\left(\mathbf{u}_{h,1}, \mathbf{v}_{h,1}\right) + \alpha \int_{\gamma}^* \mathbf{u}_{h,1} \cdot \mathbf{v}_{h,1} \, ds = \left(\mathbf{f}, \mathbf{v}_{h,1}\right)_{\Omega_1} +  \int_{\gamma}^* \mathbf{g}_{12}^* \cdot \mathbf{v}_{h,1} \, ds\quad \forall \mathbf{v}_{h,1} \in \mathbf{V}_{h,1},
$$
and
$$
\begin{aligned}
& \quad a_2\left(\mathbf{u}_{h,2}, \mathbf{v}_{h,2}\right) + \frac{1}{\delta} \int_{\Gamma_C} \left(u_{hc,2}\right)^{+} \cdot v_{hc,2} \, ds + \alpha \int_{\gamma}^* \mathbf{u}_{h,2} \cdot \mathbf{v}_{h,2} \, ds = \left(\mathbf{f}, \mathbf{v}_{h,2}\right)_{\Omega_2} +  \int_{\gamma}^* \mathbf{g}_{21}^* \cdot \mathbf{v}_{h,2} \, ds
\end{aligned}
$$
for all $\mathbf{v}_{h,2} \in \mathbf{V}_{h,2}$, where $u_{hc,2}=\mathbf{u}_{h,2}\cdot\mathbf{n}_c,v_{hc,2}=\mathbf{v}_{h,2}\cdot\mathbf{n}_c$.
\end{theorem}
Then we give the following convergence result for Algorithm \ref{alg:B}.
\begin{theorem}
\label{convergence in discrete setting}
Let $\mathbf{u}_h \in \mathbf{V}_{h,\Gamma_D}$ be the solution of problem (\ref{discrete form}) and $\mathbf{u}_{h,i} = \left.\mathbf{u}_h\right|_{\Omega_i}$ for $i=1,2$. Let $\mathbf{u}_i^n \in \mathbf{V}_{h,i} \, (i=1,2)$ be the solutions of subproblems (\ref{discrete_algo_1}) and (\ref{discrete_algo_22}) at iterative step $n$. Then we have
$$
\left(\sum_{i=1}^2 \left\|\mathbf{u}_i^n - \mathbf{u}_{h,i}\right\|_{a_i}^2 + \frac{1}{\delta}\int_{\Gamma_C}  \left[\left(u_{2,c}^n\right)^{+} - \left(u_{hc,2}\right)^{+}\right]^2 \, ds \right)^{\frac{1}{2}} \rightarrow 0 \text{ as } n \rightarrow \infty.
$$
\end{theorem}
		
\begin{algorithm}
		\caption{an iterative contact-resolving hybrid algorithm for solving (\ref{discrete form})}  
		\label{alg:B}
		\begin{itemize}
			\item[1.] Given $\mathbf{g}_{12}^0, \mathbf{g}_{21}^0\in \mathbf{V}_h\left(\gamma\right)$ arbitrarily.
			\item[2.] Recursively find $(\mathbf{u}_1^n, \mathbf{u}_2^n) \in \mathbf{V}_{h,1}\times\mathbf{V}_{h,2}$ by solving the subproblems in parallel:
\begin{align}
\label{discrete_algo_1}
a_1\left(\mathbf{u}_1^n, \mathbf{v}_{h,1}\right) + \alpha \int_{\gamma}^* \mathbf{u}_1^n \cdot \mathbf{v}_{h,1} \, ds = \left(\mathbf{f}, \mathbf{v}_{h,1}\right)_{\Omega_1} +  \int_{\gamma}^* \mathbf{g}_{12}^n \cdot \mathbf{v}_{h,1} \, ds
\end{align} 
for all $\mathbf{v}_{h,1} \in \mathbf{V}_{h,1}$; and
\begin{align}
\label{discrete_algo_22}
a_2\left(\mathbf{u}_2^n, \mathbf{v}_{h,2}\right) + \frac{1}{\delta} \int_{\Gamma_C} \left(u_{2,c}^n\right)^{+} \cdot v_{hc,2} \, ds + \alpha \int_{\gamma}^* \mathbf{u}_2^n \cdot \mathbf{v}_{h,2} \, ds = \left(\mathbf{f}, \mathbf{v}_{h,2}\right)_{\Omega_2} +  \int_{\gamma}^* \mathbf{g}_{21}^n \cdot \mathbf{v}_{h,2} \, ds
\end{align}
for all $\mathbf{v}_{h,2} \in \mathbf{V}_{h,2}$. Here $u_{2,c}^n=\mathbf{u}_2^n\cdot\mathbf{n}_c,v_{hc,2}=\mathbf{v}_{h,2}\cdot\mathbf{n}_c$. (\ref{discrete_algo_22}) is solved via a semismooth Newton subroutine.				
			\item[3.] Update the data of the transmission condition on the interfaces:
$$
\begin{aligned}
 \mathbf{g}_{12}^{n+1}(p) = 2 \alpha \mathbf{u}_2^n(p) - \mathbf{g}_{21}^n(p), \quad \mathbf{g}_{21}^{n+1}(p) = 2 \alpha \mathbf{u}_1^n(p) - \mathbf{g}_{12}^n(p),\quad \forall p \in N_\gamma.
\end{aligned}
$$
\end{itemize} 
\end{algorithm} 
\begin{remark}
The main novelty of Algorithm \ref{alg:B} is that we employ a hybrid methodology
to address nonlinear multiscale contact mechanics. We point out that Algorithm \ref{alg:B} differs notably from that described in \cite{deng2003nonoverlapping}. \cite{deng2003nonoverlapping} considers a nonoverlapping domain decomposition method for a basic second-order linear elliptic problem (i.e. $-\Delta u + \alpha(x) u = f$ in $\Omega$ with $u=0$ on $\partial\Omega$). In \cite{deng2003nonoverlapping}, it is not difficult to update the transmission condition due to the linear nature of the problem.  In contrast, our focus is on efficiently handling nonlinear contact constraints (i.e. (\ref{model_problem_c})). Note that updating the transmission condition on the interfaces in Algorithm \ref{alg:B} is a nontrivial task in our setting. This is because we use \(\mathbf{u}^n_2\) associated with the contact area, which must be obtained through a nonlinear subroutine (i.e. using semismooth Newton
subroutine to solve equation (\ref{discrete_algo_22})).  Consequently, the core step of updating data during iteration is closely related to the inherent nonlinearity of the contact constraints, making it substantially more challenging.
\end{remark}
\subsection{Iterative Contact-resolving hybrid method associated with Mixed FEM}\label{Nonoverlapping domain decomposition method associated with Mixed FEM}
In this section, we begin by presenting the mixed FEM formulation for the penalized contact problem in Section \ref{Mixed formulation}. Building on this formulation, Section \ref{iterative contact-resolving hybrid method for the mixed discretization 2} then introduces the corresponding iterative contact-resolving hybrid algorithm (i.e. Algorithm \ref{an iterative dd algorithm for continuous setting in mixed form}). Convergence result for Algorithm \ref{an iterative dd algorithm for continuous setting in mixed form} is provided in Theorem \ref{convergence for mixed}.
\subsubsection{Mixed formulation for the penalized contact problem} \label{Mixed formulation}
For the contact problem, designing iteration schemes becomes more flexible when considering
the mixed formulation. We will introduce the mixed penalty problems in this section. 

Denote
$
\mathbf{H}^{1/2}_{\Gamma_D}(\Gamma_C)\coloneqq \{\bm{\mu}|_{\Gamma_C}\mid \bm{\mu}\in\mathbf{H}^{1/2}(\partial \Omega), \bm{\mu}=\mathbf{0} \text{ on $\Gamma_D$}\}.
$
The space \(\underline{\bm{H}}(\textbf{div};\Omega,\mathbb{S})\) is defined as follows. 
Let \(\mathbb{S} = \{\underline{\bm{\tau}} \in \mathbb{R}^{n \times n} \mid \underline{\bm{\tau}} = \underline{\bm{\tau}}^\top\}\) denote the set of all \(n \times n\) real symmetric tensors, 
\(\underline{\bm{L}}^2(\Omega;\mathbb{S})\) the space of square-integrable symmetric tensor functions over \(\Omega\), 
and \(\nabla \cdot \underline{\bm{\tau}}\) the row-wise divergence of the tensor field \(\underline{\bm{\tau}}\). Without loss of generality, let $\bm{L}^2(\Omega)$ denote the space of square-integrable vector-valued functions on $\Omega$.
We then define
\[
\underline{\bm{H}}(\textbf{div};\Omega,\mathbb{S}) \coloneqq \bigl\{
\underline{\bm{\tau}} \in \underline{\bm{L}}^2(\Omega;\mathbb{S}) \mid \nabla \cdot \underline{\bm{\tau}} \in \bm{L}^2(\Omega)
\bigr\}.
\]

Let $\underline{\bm{Y}} \coloneqq \{\underline{\bm{\tau}} \in \underline{\bm{H}}(\textbf{div};\Omega,\mathbb{S}) \mid \tau_c = \mathbf{n}_c \cdot (\underline{\bm{\tau}} \cdot \mathbf{n}_c) \in L^2(\Gamma_C),\ \mathbf{t}_c \cdot (\underline{\bm{\tau}} \cdot \mathbf{n}_c) = 0 \text{  on  } \Gamma_C\},$ which is a subspace of $\underline{\bm{\mathit{H}}}(\textbf{div};\Omega,\mathbb{S})$. 
Denote $\underline{\bm{\mathit{K}}}$ as a convex subset of $\underline{\bm{\mathit{Y}}}$:
\[
\underline{\bm{\mathit{K}}}\coloneqq\{\underline{\bm{\tau}}\in \underline{\bm{\mathit{Y}}}\mid \int_{\Gamma_C}(\underline{\bm{\tau}}\cdot\mathbf{n}_c)\cdot\bm{\mu}ds\leq 0,\hspace{0.5em} \forall \bm{\mu}\in \mathbf{H}^{1/2}_{\Gamma_D}(\Gamma_C),\hspace{0.3em} \mu_c=\bm{\mu}\cdot\mathbf{n}_c\geq 0\text{ on }\Gamma_C \}.
\]
By utilizing the standard decomposition of stress and displacement vectors on $\partial\Omega$ in a tangential and a
normal component, we have
$
\underline{\bm{\sigma}}\cdot\mathbf{n}_c=\bm{\sigma}_t+\sigma_c\mathbf{n}_c,$ $\mathbf{u}=\mathbf{u}_t+u_c\mathbf{n}_c.
$
In the set $\underline{\bm{\mathit{K}}}$, the conditions $\mathbf{t}_c \cdot (\underline{\bm{\tau}} \cdot \mathbf{n}_c) = 0$ and $\int_{\Gamma_C} (\underline{\bm{\tau}} \cdot \mathbf{n}_c) \cdot \bm{\mu}  \, ds \leq 0$, together with the standard decompositions, imply that $\bm{\tau}_t=\mathbf{0}$ and $\int_{\Gamma_C}(\underline{\bm{\tau}}\cdot\mathbf{n}_c)\cdot\bm{\mu}ds=\int_{\Gamma_C}(\bm{\tau}_t+\tau_c\mathbf{n}_c)\cdot(\bm{\mu}_t+\mu_c\mathbf{n}_c)=\int_{\Gamma_C}\tau_c\mathbf{n}_c\cdot(\bm{\mu}_t+\mu_c\mathbf{n}_c)=\int_{\Gamma_C}\tau_c\mu_c\leq0$. Thus, $\underline{\bm{\mathit{K}}}$ can be rewritten as
\begin{equation}
\label{an_important_set_K}
\underline{\bm{\mathit{K}}}=\{\underline{\bm{\tau}}\in \underline{\bm{\mathit{Y}}}\mid \int_{\Gamma_C}\tau_c\mu_cds\leq 0,\hspace{0.5em} \forall \bm{\mu}\in \mathbf{H}^{1/2}_{\Gamma_D}(\Gamma_C),\hspace{0.3em} \mu_c\geq 0\text{ on }\Gamma_C \}.
\end{equation}
We consider the following dual mixed formulation of problem (\ref{model_problem}): find $(\underline{\bm{\sigma}},\mathbf{u})\in \underline{\bm{\mathit{K}}}\times \bm{\mathit{L}}^2(\Omega)$ such that
\begin{subequations}
	\label{mixed weak form for our pde}
	\begin{align}		(\mathcal{A}\underline{\bm{\sigma}},\underline{\bm{\tau}}-\underline{\bm{\sigma}})+(\textbf{div}(\underline{\bm{\tau}}-\underline{\bm{\sigma}}),\mathbf{u})&\geq 0  \quad \forall \underline{\bm{\tau}}\in \underline{\bm{\mathit{K}}}, \label{mixed weak form for our pde_a}\\
		(\textbf{div}\underline{\bm{\sigma}},\mathbf{v})&= -(\mathbf{f},\mathbf{v})\quad \forall \mathbf{v}\in \bm{\mathit{L}}^2(\Omega).\label{mixed weak form for our pde_d}
	\end{align}	
\end{subequations}
Having established the well-posedness of problem (\ref{mixed weak form for our pde}) (see, e.g., \cite{slimane2004mixed}), we now introduce the mixed penalty problem: find $(\underline{\bm{\sigma}},\mathbf{u})\in \underline{\bm{\mathit{Y}}}\times \bm{\mathit{L}}^2(\Omega)$ such that (here, $\delta$ is the penalty parameter)
\begin{subequations}
	\label{mixed weak form for our pde + penalty}
	\begin{align}		(\mathcal{A}\underline{\bm{\sigma}},\underline{\bm{\tau}})+(\textbf{div}\underline{\bm{\tau}},\mathbf{u})+\frac{1}{\delta}\int_{\Gamma_C}\sigma_c^+\tau_cds&= 0  \quad \forall \underline{\bm{\tau}}\in \underline{\bm{\mathit{Y}}}, \label{mixed weak form for our pde + penalty_a}\\
		(\textbf{div}\underline{\bm{\sigma}},\mathbf{v})&= -(\mathbf{f},\mathbf{v})\quad \forall \mathbf{v}\in \bm{\mathit{L}}^2(\Omega).\label{mixed weak form for our pde + penalty_d}
	\end{align}	
\end{subequations}
It is clear that (\ref{mixed weak form for our pde + penalty}) can be considered as the weak form of the following boundary value problem
\begin{subequations}
\label{model_problem_update_mixed_form}
    \begin{align}
        \mathcal{A} \underline{\bm{\sigma}} &= \underline{\bm{\epsilon}}(\mathbf{u}) \quad && \text{in } \Omega \label{model_problem_update_mixed_form_d}\\
       -\nabla\cdot(\underline{\bm{\sigma}}(\mathbf{u})) &= \mathbf{f} \quad &&\text{in } \Omega, \label{model_problem_update_mixed_form_a} \\
        \mathbf{u} &= \mathbf{0} \quad &&\text{on } \Gamma_D, \label{model_problem_update_mixed_form_b} \\
        u_c + \frac{1}{\delta}\sigma_c^+ &= 0 \quad &&\text{on } \Gamma_C. \label{model_problem_update_mixed_form_c}
    \end{align}	
\end{subequations}
Let $\underline{\bm{\mathit{\Sigma}}}_h \subset \underline{\bm{\mathit{Y}}}$ and $\bm{\mathit{U}}_h \subset \bm{\mathit{L}}^2(\Omega)$ be conforming discrete spaces (which can be chosen in various ways, see e.g., \cite{arnold2008,johnson1978some,watwood1968equilibrium,chung2025locking,DFLS2008}). Following the discretization approach in \cite{DFLS2008,chung2025locking}, we consider the following discrete problem: find $(\underline{\bm{\sigma}}_h,\mathbf{u}_h) \in \underline{\bm{\mathit{\Sigma}}}_h \times \bm{\mathit{U}}_h$ such that
\begin{subequations}
	\label{penalty_mixed_discrete_formula}
	\begin{align}		
		(\mathcal{A}\underline{\bm{\sigma}}_h,\underline{\bm{\tau}}_h)+(\textbf{div}\underline{\bm{\tau}}_h,\mathbf{u}_h)+\frac{1}{\delta}\int_{\Gamma_C}\sigma_{hc}^+\tau_{hc}ds&=0  \quad \forall \underline{\bm{\tau}}_h\in \underline{\bm{\mathit{\Sigma}}}_h, \label{penalty_mixed_discrete_formula_a}\\
		-(\textbf{div}\underline{\bm{\sigma}}_h,\mathbf{v}_h)&= s(\pi(\tilde{k}^{-1}\mathbf{f}),\mathbf{v}_h)\quad \forall \mathbf{v}_h\in \bm{\mathit{U}}_h,\label{penalty_mixed_discrete_formula_d}
	\end{align}	
\end{subequations}
where $s(\cdot,\cdot)$ is a weighted $L^2$ inner product (see \cite{chung2025locking} or Section \ref{standard cem} for more details). Note that in the discrete scheme (\ref{penalty_mixed_discrete_formula}), we denote the right-hand side as the bilinear form $s(\cdot,\cdot)$ rather than a standard $L^2$ inner product. This choice is made for convenience in certain error analyses and does not fundamentally affect the overall stability or convergence (see \cite{chung2025locking} for details).
The well-posedness of Eqs. (\ref{penalty_mixed_discrete_formula}) can be confirmed by combining the references \cite{chung2025locking,hu2024mixed,slimane2004mixed,kikuchi1988contact}. For all $\underline{\bm{\tau}}_h\in\underline{\bm{\mathit{\Sigma}}}_h$, the normal component $\underline{\bm{\tau}}_h \cdot \mathbf{n}$ is continuous across $\gamma$ because $\underline{\bm{\mathit{\Sigma}}}_h \subset \underline{\bm{\mathit{Y}}}\subset \underline{\bm{H}}(\textbf{div};\Omega,\mathbb{S})$. In particular, we have $\underline{\bm{\tau}}_h \cdot \mathbf{n} \in \mathbf{V}_h$. Building on the splitting technique established in Theorem \ref{equivalent discrete form}, we can then decompose (\ref{penalty_mixed_discrete_formula}) into the equivalent subproblems as (\ref{penalty_mixed_discrete_formula_seperate_1})-(\ref{penalty_mixed_discrete_formula_seperate_2}) below.

For the mixed stress-displacement formulation, we employ a Robin-type transmission condition like the form
``\(
\beta \underline{\bm{\sigma}} \cdot \mathbf{n} + \mathbf{u} = \mathbf{g}
\)"
on the interface $\gamma$ (where \(\beta > 0\) is the transmission coefficient, which will be used below). This condition is also motivated by Lions~\cite{lions1990schwarz}, Tang~\cite{tang1992generalized} and later works such as \cite{deng1997timely, deng2003optimal, deng2003nonoverlapping}.  Let $(\underline{\bm{\sigma}}_h,\mathbf{u}_h) \in \underline{\bm{\mathit{\Sigma}}}_h \times \bm{\mathit{U}}_h$ be the solution of problem (\ref{penalty_mixed_discrete_formula}) and $\underline{\bm{\sigma}}_{h,i}=\underline{\bm{\sigma}}_h|_{\Omega_i},\mathbf{u}_{h,i}=\mathbf{u}_h|_{\Omega_i}$ for $i=1,2$. Then we can split (\ref{penalty_mixed_discrete_formula}) into equivalent subproblems: there exist $\mathbf{g}_{12}^{mix},\mathbf{g}_{21}^{mix} \in \mathbf{V}_h\left(\gamma\right)$ such that $(\underline{\bm{\sigma}}_{h,i},\mathbf{u}_{h,i})\in \underline{\bm{\mathit{\Sigma}}}_{h,i}\times\bm{\mathit{U}}_{h,i}\coloneqq \underline{\bm{\mathit{\Sigma}}}_h|_{\Omega_i}\times\bm{\mathit{U}}_h|_{\Omega_i} \, (i=1,2)$ satisfies
\begin{subequations}
\label{penalty_mixed_discrete_formula_seperate_1}
	\begin{align}		
&(\mathcal{A}\underline{\bm{\sigma}}_{h,1},\underline{\bm{\tau}}_{h,1})_{\Omega_1}+(\textbf{div}\underline{\bm{\tau}}_{h,1},\mathbf{u}_{h,1})_{\Omega_1}+\beta\int_{\gamma}^*\sigma_{hn,1}\tau_{hn,1}ds=\int_{\gamma}^*\mathbf{g}_{12}^{mix}\cdot(\underline{\bm{\tau}}_{h,1}\cdot\mathbf{n}_1)ds  \quad \forall \underline{\bm{\tau}}_{h,1}\in \underline{\bm{\mathit{\Sigma}}}_{h,1}, \label{penalty_mixed_discrete_formula_seperate_1_a}  \\
	&	(\textbf{div}\underline{\bm{\sigma}}_{h,1},\mathbf{v}_{h,1})_{\Omega_1}= s(\pi(\tilde{k}^{-1}\mathbf{f}),\mathbf{v}_{h,1})\quad \forall \mathbf{v}_{h,1}\in \bm{\mathit{U}}_{h,1},\label{penalty_mixed_discrete_formula_seperate_1_d}
	\end{align}	
\end{subequations}
\begin{subequations}
\label{penalty_mixed_discrete_formula_seperate_2}
	\begin{align}		
		&\quad (\mathcal{A}\underline{\bm{\sigma}}_{h,2},\underline{\bm{\tau}}_{h,2})_{\Omega_2}+(\textbf{div}\underline{\bm{\tau}}_{h,2},\mathbf{u}_{h,2})_{\Omega_2}+\beta\int_{\gamma}^*\sigma_{hn,2}\tau_{hn,2}ds+\frac{1}{\delta}\int_{\Gamma_C}(\sigma_{hc,2})^+\tau_{hc,2}ds \nonumber\\
&=\int_{\gamma}^*\mathbf{g}_{21}^{mix}\cdot(\underline{\bm{\tau}}_{h,2}\cdot\mathbf{n}_2)ds  \quad \forall \underline{\bm{\tau}}_{h,2}\in \underline{\bm{\mathit{\Sigma}}}_{h,2}, \label{penalty_mixed_discrete_formula_seperate_2_a} \\
	&\quad	(\textbf{div}\underline{\bm{\sigma}}_{h,2},\mathbf{v}_{h,2})_{\Omega_2}= s(\pi(\tilde{k}^{-1}\mathbf{f}),\mathbf{v}_{h,2})\quad \forall \mathbf{v}_{h,2}\in \bm{\mathit{U}}_{h,2},\label{penalty_mixed_discrete_formula_seperate_2_d}
	\end{align}	
\end{subequations}
where $\tau_{hn,i}=\mathbf{n}_i\cdot(\underline{\bm{\tau}}_h\cdot\mathbf{n}_i)$ ($i=1,2$), $\tau_{hc,2}=\mathbf{n}_c\cdot(\underline{\bm{\tau}}_{h,2}\cdot\mathbf{n}_c)$. Note that the transmission coefficient $\beta$ in the mixed setting (i.e. (\ref{penalty_mixed_discrete_formula_seperate_1})-(\ref{penalty_mixed_discrete_formula_seperate_2})) differ from $\alpha$ in the standard setting (i.e. (\ref{discrete_algo_1})-(\ref{discrete_algo_22})).

\subsubsection{An iterative contact-resolving hybrid method for the mixed discretization (\ref{penalty_mixed_discrete_formula})} \label{iterative contact-resolving hybrid method for the mixed discretization 2}
This section presents an iterative contact-resolving hybrid algorithm (Algorithm \ref{an iterative dd algorithm for continuous setting in mixed form}) for solving (\ref{penalty_mixed_discrete_formula}). The convergence of Algorithm \ref{an iterative dd algorithm for continuous setting in mixed form} is given in Theorem \ref{convergence for mixed}.

\begin{algorithm}
		\caption{an iterative contact-resolving hybrid algorithm for solving (\ref{penalty_mixed_discrete_formula})}  
		\label{an iterative dd algorithm for continuous setting in mixed form}
		\begin{itemize}
			\item[1.] Given $\mathbf{g}_{12}^0, \mathbf{g}_{21}^0\in \mathbf{V}_h\left(\gamma\right)$ arbitrarily.
			\item[2.] Recursively find $(\underline{\bm{\sigma}}^n_1,\mathbf{u}^n_1)\in \underline{\bm{\Sigma}}_{h,1}\times\bm{U}_{h,1}$ by solving
\begin{subequations}
\label{eq:system}
	\begin{align}
(\mathcal{A}\underline{\bm{\sigma}}_1^n,\underline{\bm{\tau}}_{h,1})_{\Omega_1}
		+ (\textbf{div}\,\underline{\bm{\tau}}_{h,1},\mathbf{u}_1^n)_{\Omega_1}
		+ \beta\int_{\gamma}^*(\sigma_1^n)_n\tau_{hn,1}\,ds 
		&= \int_{\gamma}^*\mathbf{g}_{12}^{n}\cdot(\underline{\bm{\tau}}_{h,1}\cdot\mathbf{n}_1)\,ds,  \label{eq:sub1} \\
		(\textbf{div}\,\underline{\bm{\sigma}}_1^n,\mathbf{v}_{h,1})_{\Omega_1} 
		&= s\bigl(\pi(\tilde{k}^{-1}\mathbf{f}),\mathbf{v}_{h,1}\bigr), \label{eq:sub2}
	\end{align}
\end{subequations}
for all $(\underline{\bm{\tau}}_{h,1},\mathbf{v}_{h,1})\in\underline{\bm{\mathit{\Sigma}}}_{h,1}\times \bm{\mathit{U}}_{h,1}$ and find $(\underline{\bm{\sigma}}^n_2,\mathbf{u}_2^n)\in \underline{\bm{\Sigma}}_{h,2}\times\bm{U}_{h,2}$ by solving
\begin{subequations}
\label{eq:system2}
\begin{align}
&(\mathcal{A}\underline{\bm{\sigma}}^n_2,\underline{\bm{\tau}}_{h,2})_{\Omega_2} 
    + (\textbf{div}\,\underline{\bm{\tau}}_{h,2},\mathbf{u}^n_2)_{\Omega_2}
    + \beta\int_{\gamma}^*(\sigma^n_2)_n\tau_{hn,2}\,ds
    + \frac{1}{\delta}\int_{\Gamma_C}(\sigma^n_{2,c})^+\tau_{hc,2}\,ds = \int_{\gamma}^*\mathbf{g}_{21}^{n}\cdot(\underline{\bm{\tau}}_{h,2}\cdot\mathbf{n}_2)\,ds, \label{eq:2a} \\
&(\textbf{div}\,\underline{\bm{\sigma}}^n_2,\mathbf{v}_{h,2})_{\Omega_2} 
    = (\mathbf{f},\mathbf{v}_{h,2})_{\Omega_2}, \label{eq:2b}
\end{align}
\end{subequations}
for all $(\underline{\bm{\tau}}_{h,2},\mathbf{v}_{h,2})\in\underline{\bm{\mathit{\Sigma}}}_{h,2}\times \bm{\mathit{U}}_{h,2}$.
Here $(\sigma^n_i)_n=\mathbf{n}_i\cdot(\underline{\bm{\sigma}}^n_i\cdot\mathbf{n}_i),\sigma^n_{2,c}=\mathbf{n}_c\cdot(\underline{\bm{\sigma}}^n_2\cdot\mathbf{n}_c),\tau_{hn,i}=\mathbf{n}_i\cdot(\underline{\bm{\tau}}_{h,i}\cdot\mathbf{n}_i),\tau_{hc,2}=\mathbf{n}_c\cdot(\underline{\bm{\tau}}_{h,2}\cdot\mathbf{n}_c)$ for $i=1,2$. (\ref{eq:system2}) is solved via a semismooth Newton subroutine.		

			\item[3.] Update the data of the transmission condition on the interfaces:
         $$
            \begin{array}{ll}
            \mathbf{g}_{12}^{n+1}(p)=-2 \beta\underline{\bm{\sigma}}^n_2(p)\cdot\mathbf{n}_2+\mathbf{g}_{21}^n(p),\quad \mathbf{g}_{21}^{n+1}(p)=-2 \beta \underline{\bm{\sigma}}^n_1(p)\cdot\mathbf{n}_1 +\mathbf{g}_{12}^n(p),\quad \forall p \in N_\gamma. 
            \end{array}
            $$
		\end{itemize} 
\end{algorithm}
\begin{remark}
The primary novelty of Algorithm \ref{an iterative dd algorithm for continuous setting in mixed form} is to present a mixed formulation within the hybrid framework for the nonlinear contact problem. This formulation enables the simulation of (nearly) incompressible materials, which is not a simple extension of Algorithm~\ref{alg:B}. Crucially, in Algorithm \ref{an iterative dd algorithm for continuous setting in mixed form}, the stress-related nonlinear term $(\sigma^n_{2,c})^+$ must be solved iteratively in \(\Omega_2\), and the transmission condition is associated with the stress tensor, which is different from those in Algorithm \ref{alg:B}. From a theoretical standpoint, we can estimate the errors for both stress and displacement without postprocessing (see Theorem \ref{convergence for mixed} for more details). Numerically, the approach presents greater complexity than the standard case due to the enlarged system size and the need to properly handle updates to the transmission data, which now involve the normal component of $\underline{\bm{\sigma}}^n_2$.
\end{remark}
Next we give the convergence result for Algorithm \ref{an iterative dd algorithm for continuous setting in mixed form}. For the ease of notation expressions, we denote
$$
\begin{aligned}
 \underline{\mathbf{e}}^n_\sigma &= \left(\underline{\mathbf{e}}^n_{\sigma,1},\underline{\mathbf{e}}^n_{\sigma,2}\right) := \left(\underline{\bm{\sigma}}^n_1-\underline{\bm{\sigma}}_{h,1}, \underline{\bm{\sigma}}^n_2-\underline{\bm{\sigma}}_{h,2}\right) \in \underline{\bm{\mathit{\Sigma}}}_{h,1}\times\underline{\bm{\mathit{\Sigma}}}_{h,2},\\
 \mathbf{e}^n_u &= \left(\mathbf{e}^n_{u,1},\mathbf{e}^n_{u,2}\right) := \left(\mathbf{u}^n_1-\mathbf{u}_{h,1}, \mathbf{u}^n_2-\mathbf{u}_{h,2}\right) \in \bm{\mathit{U}}_{h,1}\times\bm{\mathit{U}}_{h,2}
\end{aligned}
$$
and the norms as follows:
\begin{equation} \label{important_norms}
\left\|\underline{\mathbf{e}}^n_\sigma\right\|_{\mathcal{A}}^2=\sum_{i=1}^2(\mathcal{A}\underline{\mathbf{e}}^n_{\sigma,i},\underline{\mathbf{e}}^n_{\sigma,i})_{\Omega_i},\quad
\left\|\underline{\mathbf{e}}^n_\sigma\right\|_e^2 = \left\|\underline{\mathbf{e}}^n_\sigma\right\|_{\mathcal{A}}^2 + \frac{1}{\delta}\int_{\Gamma_C} \left|(\sigma^n_{2,c})^+ - (\sigma_{hc,2})^+\right|^2 \, ds,\quad \left\|\mathbf{e}^n_u\right\|_{L^2}^2=\sum_{i=1}^2(\mathbf{e}^n_{u,i},\mathbf{e}^n_{u,i})_{\Omega_i},
\end{equation}
where $\sigma^n_{2,c}=\mathbf{n}_c\cdot(\underline{\bm{\sigma}}^n_2\cdot\mathbf{n}_c)$, $\sigma_{hc,2}=\mathbf{n}_c\cdot(\underline{\bm{\sigma}}_{h,2}\cdot\mathbf{n}_c)$.
\begin{theorem}\label{convergence for mixed}
Let $(\underline{\bm{\sigma}}_h,\mathbf{u}_h) \in \underline{\bm{\mathit{\Sigma}}}_h \times \bm{\mathit{U}}_h$ be the solution of problem (\ref{penalty_mixed_discrete_formula}) and $\underline{\bm{\sigma}}_{h,i}=\underline{\bm{\sigma}}_h|_{\Omega_i},\mathbf{u}_{h,i}=\mathbf{u}_h|_{\Omega_i}$ for $i=1,2$. Let $(\underline{\bm{\sigma}}_i^n, \mathbf{u}_i^n) \in \underline{\bm{\mathit{\Sigma}}}_{h,i}\times\bm{\mathit{U}}_{h,i} \, (i=1,2)$ be the solutions of subproblems (\ref{eq:system}) and (\ref{eq:system2}) at iterative step $n$. Then we have
$$
\left\|\underline{\mathbf{e}}^n_\sigma\right\|_e^2 + \left\|\mathbf{e}^n_u\right\|_{L^2}^2 \rightarrow 0 \text{ as } n \rightarrow \infty,
$$
where the norms $\left\|\underline{\mathbf{e}}^n_\sigma\right\|_e, \left\|\mathbf{e}^n_u\right\|_{L^2} $ are defined as in (\ref{important_norms}).
\end{theorem}

\subsection{Iterative contact-resolving hybrid method associated with standard CEM-GMsFEM}\label{Nonoverlapping domain decomposition method associated with standard CEM-GMsFEM}
In this section, we present an iterative contact-resolving hybrid method associated with the standard CEM-GMsFEM \cite{chung2023multiscale, ye2023constraint, chung2018, chung2025locking}. The main novelty lies in introducing multiscale model reduction into the hybrid framework to improve computational efficiency for nonlinear contact problems with high-contrast coefficients.
Here, the nonlinear contact region $\Omega_2$ occupies merely one coarse mesh wide and therefore does not dominate the primary computational cost. However, the larger domain $\Omega_1$ covers almost the entire domain, and resolving its fine-scale heterogeneities can be computationally expensive, especially with high-contrast coefficients. To address this, we integrate multiscale reduction techniques to lower the computational burden.

The splitting formulation for this setting is described in Section \ref{Discrete formulation with standard CEM-GMsFEM}, while the construction of multiscale basis functions within the larger subdomain is presented in Section \ref{standard cem}. An iterative domain decomposition procedure (Algorithm \ref{alg:A}) and the corresponding convergence (Theorem \ref{convergence in standard cem}) are then developed in Section \ref{iterative domain decomposition method for standard problem}.
A key aspect of the present framework is the careful treatment of boundary conditions (see different stages in Algorithm \ref{alg:A}). The solution $\mathbf{u}_h$ obtained from (\ref{discrete form}) is regarded as the reference solution.
\subsubsection{Splitting system for the method associated with standard CEM-GMsFEM}\label{Discrete formulation with standard CEM-GMsFEM}
Analogous to Theorem \ref{equivalent discrete form}, we develop the splitting subproblem formulation for the iterative contact-resolving hybrid method associated with standard CEM-GMsFEM.
Using the multiscale space $\mathbf{V}_{\text{ms}}$ (constructed in Section \ref{standard cem}) for the displacement field in $\Omega_1$ and the standard FEM space $\mathbf{V}_{h,2}$ for that in $\Omega_2$, we can obtain the numerical solution $(\mathbf{u}_{\text{ms},1}, \mathbf{u}_{h,2}) \in \mathbf{V}_{\text{ms}} \times \mathbf{V}_{h,2}$ by solving the analogous subproblems described in Theorem \ref{equivalent discrete form}, with the only modification being the replacement of the space $\mathbf{V}_{h,1}$ by $\mathbf{V}_{\text{ms}}$. We omit the analogous formulations here and focus instead on the multiscale reduction techniques below.
\subsubsection{The construction of the multiscale basis functions for standard CEM-GMsFEM}
\label{standard cem}
\begin{figure}[tbph] 
		\centering 
		\includegraphics[height=5.5cm,width=9.5cm]{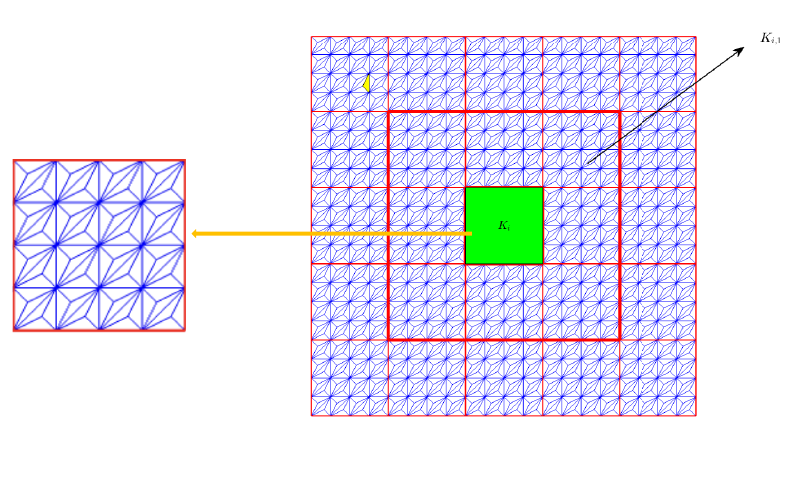} 
		\caption{An illustration of a coarse element $K_i$ (green, and left for clearer), a finite element (a triangle with blue edges), and an oversampling region $K_{i,1}$ by extending $K_i$ outward by one coarse mesh layer (thick red line)}
		\label{coarse_bigger}
\end{figure}
In this section, we will present the construction of the multiscale basis functions for the standard CEM-GMsFEM in $\Omega_1$ (see \cite{chung2023multiscale,chung2025locking}).
Let $\mathcal{T}_H := \bigcup_{i=1}^N \{K_i\}$ be a conforming quasi-uniform coarse mesh of the domain $\Omega$. Here $H$ is the coarse mesh size, and $N$ is the number of coarse elements $K_i$. Each coarse element is further subdivided into a connected union of fine elements. The corresponding fine mesh $\mathcal{T}_h$ is obtained by refining $\mathcal{T}_H$. Define $K_{i,m}$ as the oversampling region by extending $K_i$ by $m$ coarse mesh layers.
For an illustrative example, see Figure \ref{coarse_bigger} (same as \cite[Figure 3.1]{chung2025locking}), which depicts
the coarse mesh, fine mesh, and an oversampling region.

The construction of the basis functions are developed on Figure \ref{coarse_bigger}. 
Let $K_i$ be the $i$-th coarse element and let $\mathbf{V}_{h,i}\left(K_i\right)$ be the restriction of $\mathbf{V}_{h,i}$ on $K_i$. We solve a local spectral problem: for each $K_i$, find a real number $\lambda_j^i \in \mathbb{R}$ and a function $\bm{\phi}_j^i \in \mathbf{V}_{h,i}\left(K_i\right)$ such that
\begin{equation}
\label{local spectral problem}
a_{1,i}\left(\bm{\phi}^i_j, \mathbf{v}\right)+\alpha \int_{\gamma\cap\partial K_i}^* \bm{\phi}^i_j \cdot \mathbf{v} d s=\lambda^i_j s_i(\bm{\phi}^i_j,\mathbf{v}),\quad \forall \mathbf{v}\in\mathbf{V}_{h,i}\left(K_i\right)
\end{equation}
where 
$$a_{1,i}(\mathbf{w},\mathbf{v})=\int_{K_i}\mathcal{A}^{-1}\underline{\bm{\epsilon}}(\mathbf{w})\colon\underline{\bm{\epsilon}}(\mathbf{v})dx,\quad s_i(\mathbf{w},\mathbf{v})=\int_{K_i} \widetilde{k} \mathbf{w} \mathbf{v} dx,\quad \widetilde{k}=k H^{-2}, \quad k=\lambda+2 \mu.$$
Assume that $s_i(\bm{\phi}^i_j,\bm{\phi}^i_j)=1$ and we arrange the eigenvalues of (\ref{local spectral problem}) in non-decreasing order $0\leq\lambda_1\leq\lambda_2\leq\cdot\cdot\cdot\leq\lambda_{L_i}$, where $L_i$ is the dimension of the space $\mathbf{V}_{h,i}\left(K_i\right)$. For each $i\in\{1,2,\cdot\cdot\cdot,N\}$, choose the first $l_i$ $(1\leq l_i\leq L_i)$ eigenfunctions $\{\bm{\phi}^i_j\}_{j=1}^{l_i}$ corresponding to the first $l_i$ smallest eigenvalues. We define the local auxiliary multiscale space $\mathbf{V}_\text{aux}(K_i)$ for the displacement field as 
$\mathbf{V}_\text{aux}(K_i)\coloneqq\text{span}\{\bm{\phi}^i_j\mid 1\leq j\leq l_i\}.$
The global auxiliary multiscale finite element space $\mathbf{V}_\text{aux}$ is defined by $\mathbf{V}_\text{aux}=\oplus_i\mathbf{V}_\text{aux}(K_i)$.
Then we define the inner product $s(\cdot,\cdot)$ on the global auxiliary multiscale space $\mathbf{V}_\text{aux}$ through
\(
s(\mathbf{p},\mathbf{q}) = \sum_{i=1}^N s_i(\mathbf{p},\mathbf{q})\). And \(\norm{\mathbf{p}}_s = s(\mathbf{p},\mathbf{p})^{1/2},\) for all $\mathbf{p} \in \mathbf{V}_\text{aux}$.
We remark that $s(\cdot,\cdot)$ and $\norm{\cdot}_s$ also constitute a valid inner product and norm on the space $\bm{\mathit{L}}^2(\Omega)$.
Next we define the projection operator $\pi_i \colon \bm{\mathit{L}}^2(K_i) \to \mathbf{V}_{\text{aux}}(K_i)$ 
with respect to the inner product $s_i(\cdot,\cdot)$. Specifically, for any $\mathbf{q} \in \bm{\mathit{L}}^2(K_i)$, 
the operator $\pi_i$ is given by
\[
\pi_i(\mathbf{q}) = \sum_{j=1}^{l_i} s_i(\mathbf{q}, \bm{\phi}_j^i)\bm{\phi}_j^i.
\]
Similarly, we define the global projection operator $\pi \colon \bm{\mathit{L}}^2(\Omega) \to \mathbf{V}_{\mathrm{aux}}$ 
with respect to the inner product $s(\cdot,\cdot)$. For any $\mathbf{q} \in \bm{\mathit{L}}^2(\Omega)$,
this operator takes the form
$
\pi(\mathbf{q}) = \sum_{i=1}^N \sum_{j=1}^{l_i} s_i(\mathbf{q}, \mathbf{p}_j^i) \mathbf{p}_j^i.
$
Clearly $\pi = \sum_{i=1}^N \pi_i$.

Next we present the construction of the multiscale basis functions. We directly consider the relaxed constraint energy minimizing generalized multiscale finite element method \cite{chung2018}. We solve the following unconstrained minimization problem: find $\bm{\psi}_{j, \text{ms}}^i \in \mathbf{V}_{h,1}\left(K_{i,m}\right)$ such that 

$$
\bm{\psi}_{j, \text{ms}}^i=\operatorname{argmin}\left\{a_1(\bm{\psi}, \bm{\psi})+\alpha \int_{\gamma}^* \bm{\psi} \cdot \bm{\psi}ds+s\left(\pi \bm{\psi}-\bm{\phi}_j^i, \pi \bm{\psi}-\bm{\phi}_j^i\right) \colon \bm{\psi}\in \mathbf{V}^m_{i,h}\right\},
$$
where $\mathbf{V}_{i, h}^m=\left\{\mathbf{v} \in \mathbf{V}_{h,1}\left(K_{i,m}\right)\mid \mathbf{v}=\mathbf{0}\text{ on $\Gamma_D\cap \partial K_{i,m}$ or $\Omega_1\cap\partial K_{i,m}$}\right\}$. Clearly $\mathbf{V}_{i, h}^m \subset \mathbf{V}_{h, 1}$. $a_1(\cdot,\cdot)$ has been defined in Section \ref{Nonoverlapping domain decomposition method associated with Standard FEM formulation}.
It can be shown that $\bm{\psi}_{j, \text{ms}}^i$ satisfies the following variational form
\begin{equation}
\label{to obtain multiscale basis}
a_1\left(\bm{\psi}_{j,\text{ms}}^i, \mathbf{v}\right)+\alpha \int_{\gamma}^* \bm{\psi}_{j,\text{ms}}^i \cdot \mathbf{v} d s+s\left(\pi \bm{\psi}_{j,\text{ms}}^i, \pi \mathbf{v}\right)=s\left(\bm{\phi}_j^i, \pi \mathbf{v}\right)
\end{equation}
for all $\mathbf{v} \in \mathbf{V}_{i, h}^m$.
Denote
$
\mathbf{V}_{\text{ms}}\coloneqq\text{span}\{\bm{\psi}^i_{j,\text{ms}}\mid 1\leq j\leq l_i,1\leq i\leq N\}$.
Following the methodology in \cite{ye2023constraint}, we compute $\mathbf{u}_{\text{ms},1}$ using a computational procedure comprising the following steps:

\textbf{Step 1:} Find $\mathcal{N}_i^m \mathbf{g}_{12}^*$ such that for all $\mathbf{v}\in \mathbf{V}_{i, h}^m$ ($\mathcal{N}_i^m$ is an operator mapping from $\mathbf{V}_h(\gamma)$ to $\mathbf{V}_{i, h}^m$),
$$\quad a_1\left(\mathcal{N}_i^m \mathbf{g}_{12}^*, \mathbf{v}\right)+\alpha \int_{\gamma}^* \mathcal{N}_i^m \mathbf{g}_{12}^*\cdot \mathbf{v} ds+s\left(\pi\left(\mathcal{N}_i^m \mathbf{g}_{12}^*\right), \pi \mathbf{v}\right)=\int_{\partial K_i \cap \gamma}^* \mathbf{g}_{12}^*\cdot \mathbf{v}ds.$$
Then we obtain $\mathcal{N}^m \mathbf{g}_{12}^*=\sum_{i=1}^N \mathcal{N}_i^m \mathbf{g}_{12}^*$.

\textbf{Step 2:} Prepare the multiscale space $\mathbf{V}_{\text{ms}}$ via (\ref{local spectral problem}) and (\ref{to obtain multiscale basis}).

\textbf{Step 3:} Solve $\mathbf{w}^m$ such that for all $\mathbf{v} \in \mathbf{V}_{\text{ms}}$,
$$
a_1\left(\mathbf{w}^m, \mathbf{v}\right)+\alpha \int_{\gamma}^* \mathbf{w}^m\cdot \mathbf{v} ds=(\mathbf{f},\mathbf{v})_{\Omega_1} +\int_{\gamma}^* \mathbf{g}_{12}^* \cdot \mathbf{v}ds-\left[a_1\left(\mathcal{N}^m \mathbf{g}_{12}^*, \mathbf{v}\right)+\alpha \int_{\gamma}^* \mathcal{N}^m \mathbf{g}_{12}^*\cdot\mathbf{v}ds\right].
$$

\textbf{Step 4:} Construct the numerical solution $\mathbf{u}_{\text{ms},1}$ as
$
\mathbf{u}_{\text{ms},1} \approx \mathbf{w}^m + \mathcal{N}^m \mathbf{g}_{12}^*.
$

The four steps form a direct solver to obtain $\mathbf{u}_{\text{ms},1}$, which may have an $O(H)$ error with respect to the reference solution $\mathbf{u}_h$ obtained from (\ref{discrete form}). We include these detailed steps to clarify the multiscale reduction procedure; furthermore, the multiscale space $\mathbf{V}_{\text{ms}}$ will be reused extensively in Algorithm~\ref{alg:A} below.
\subsubsection{An iterative contact-resolving hybrid method associated with standard CEM-GMsFEM}
\label{iterative domain decomposition method for standard problem}
In this section, we define the iterative procedure associated with standard CEM-GMsFEM as Algorithm \ref{alg:A} and confirm its convergence in Theorem \ref{convergence in standard cem}. Recall the norm notation that $\norm{\mathbf{v}}_{a_i}^2=a_i(\mathbf{v},\mathbf{v})=\int_{\Omega_i}\mathcal{A}^{-1}\underline{\bm{\epsilon}}(\mathbf{v})\colon\underline{\bm{\epsilon}}(\mathbf{v})dx$ for $i=1,2$.
\begin{algorithm}
		\caption{an iterative contact-resolving hybrid algorithm by using standard multiscale techniques} 
		\label{alg:A}
		\begin{itemize}
			\item[1.] Given $\mathbf{g}_{12}^0, \mathbf{g}_{21}^0\in \mathbf{V}_h\left(\gamma\right)$ arbitrarily.
			\item[2.] Recursively find $(\mathbf{u}_1^n,\mathbf{u}_2^n)\in\mathbf{V}_{\text{ms}}\times\mathbf{V}_{h,2}$ by solving the subproblems in parallel:

            (i) To obtain $\mathbf{u}_1^n\in\mathbf{V}_{\text{ms}}$, we consider the following three stages:

            \textbf{Stage a:} Find $\mathcal{N}_i^m \mathbf{g}_{12}^n$ such that for all $\mathbf{v}\in \mathbf{V}_{i, h}^m$,
$$\quad a_1\left(\mathcal{N}_i^m \mathbf{g}_{12}^n, \mathbf{v}\right)+\alpha \int_{\gamma}^* \mathcal{N}_i^m \mathbf{g}_{12}^n\cdot \mathbf{v} ds+s\left(\pi\left(\mathcal{N}_i^m \mathbf{g}_{12}^n\right), \pi \mathbf{v}\right)=\int_{\partial K_i \cap \gamma}^* \mathbf{g}_{12}^n\cdot \mathbf{v}ds.$$
Then we obtain $\mathcal{N}^m \mathbf{g}_{12}^n=\sum_{i=1}^N \mathcal{N}_i^m \mathbf{g}_{12}^n$.

\textbf{Stage b:} Solve $\mathbf{w}^m_n$ such that of for all $\mathbf{v}\in\mathbf{V}_{\text{ms}}$,
\begin{align*}
 a_1\left(\mathbf{w}^m_n, \mathbf{v}\right)+\alpha \int_{\gamma}^* \mathbf{w}^m_n\cdot \mathbf{v} ds=(\mathbf{f},\mathbf{v})_{\Omega_1} +\int_{\gamma}^* \mathbf{g}_{12}^n \cdot \mathbf{v}ds -\left[a_1\left(\mathcal{N}^m \mathbf{g}_{12}^n, \mathbf{v}\right)+\alpha \int_{\gamma}^* \mathcal{N}^m \mathbf{g}_{12}^n\cdot\mathbf{v}ds\right].
\end{align*}

\textbf{Stage c:} Construct the numerical solution $\mathbf{u}^n_{1}$ as
$
\mathbf{u}^n_{1} \approx \mathbf{w}^m_n+\mathcal{N}^m \mathbf{g}_{12}^n.
$

(ii) Obtain $\mathbf{u}_2^n\in\mathbf{V}_{h,2}$ by solving the same equation (\ref{discrete_algo_22}) in Algorithm \ref{alg:B}.	
			\item[3.] Update the data by applying the same transmission condition as that specified in Algorithm \ref{alg:B}.
		\end{itemize} 
\end{algorithm}
\begin{theorem}
\label{convergence in standard cem}
Let $\mathbf{u}_h \in \mathbf{V}_{h,\Gamma_D}$ be the solution of problem (\ref{discrete form}) and $\mathbf{u}_{h,i} = \left.\mathbf{u}_h\right|_{\Omega_i}$ for $i=1,2$. For each iteration step $n \geq 0$, let $(\mathbf{u}_1^n, \mathbf{u}_2^n) \in \mathbf{V}_{\mathrm{ms}} \times \mathbf{V}_{h,2}$ be the numerical approximations generated by Algorithm \ref{alg:A}. 
Then there exists a positive sequence $\{ \epsilon_n\}$ with $ \epsilon_n \to 0$ as $n \to \infty$ such that:
\[
\left( \sum_{i=1}^2 \left\|\mathbf{u}_i^n - \mathbf{u}_{h,i}\right\|_{a_i}^2 + \frac{1}{\delta} \int_{\Gamma_C} \bigl[(u_{2,c}^n)^+ - (u_{hc,2})^+\bigr]^2 \, ds \right)^{\!1/2} \leq  \epsilon_n \to 0 \text{ as } n \rightarrow \infty.
\]
\end{theorem}

\subsection{Iterative contact-resolving hybrid method associated with mixed CEM-GMsFEM} \label{Nonoverlapping domain decomposition method associated with mixed CEM-GMsFEM} 
In this section, we develop an iterative contact-resolving hybrid method based on a combination of mixed CEM-GMsFEM and mixed FEM. The method applies the mixed CEM-GMsFEM to the larger domain $\Omega_1$ and the standard mixed FEM to the smaller domain $\Omega_2$ containing the contact boundary. The splitting system is described in Section \ref{Discrete formulation with mixed CEM-GMsFEM}, and the construction of multiscale basis functions is stated in Section \ref{mixed cem}. Section \ref{iterative domain decomposition method for mixed problem} presents the iterative procedure (Algorithm \ref{alg:BB}) and corresponding convergence (Theorem \ref{convergence in mixed cem}).
$(\underline{\bm{\sigma}}_h,\mathbf{u}_h)$ computed from (\ref{penalty_mixed_discrete_formula}) is taken as the reference solution for the mixed multiscale formulation.
\subsubsection{Splitting system for the method associated with mixed CEM-GMsFEM}\label{Discrete formulation with mixed CEM-GMsFEM}
Referring to (\ref{penalty_mixed_discrete_formula_seperate_1})--(\ref{penalty_mixed_discrete_formula_seperate_2}), we can develop the splitting subproblem formulation for the iterative contact-resolving hybrid method associated with mixed CEM-GMsFEM.
Utilizing the multiscale space $\Sigma_\text{ms}$ and $U_\text{aux}$ (to be constructed in Section \ref{mixed cem}) to approximate the stress and displacement in $\Omega_1$ and the mixed FEM space in $\Omega_2$, we can obtain the numerical solution $(\underline{\bm{\sigma}}_{\text{ms},1},\mathbf{u}_{\text{ms},1})$ in $\Omega_1$ and $(\underline{\bm{\sigma}}_{h,2},\mathbf{u}_{h, 2})$ in $\Omega_2$ by solving subproblems similar to (\ref{penalty_mixed_discrete_formula_seperate_1})--(\ref{penalty_mixed_discrete_formula_seperate_2}), with the only modification being the replacement of the spaces. We omit the analogous formulations here and focus instead on the multiscale reduction techniques below.

\subsubsection{The construction of the multiscale basis functions for mixed CEM-GMsFEM}
\label{mixed cem}
This section presents the construction of multiscale basis functions for the mixed CEM-GMsFEM in $\Omega_1$ (see \cite{chung2025locking}), adopting all notation from Section \ref{standard cem}. The basis functions are constructed over the mesh illustrated in Figure \ref{coarse_bigger}.
The computation of the multiscale basis functions is divided into two stages. 
The first stage consists of constructing the multiscale space for the displacement $\mathbf{u}$. 
In the second stage, we will use the multiscale space for displacement to construct a multiscale space for the stress $\bm{\underline{\sigma}}$. 
We point out that the supports of displacement basis are the coarse elements. 
For stress basis functions, the support is an oversampling region containing the support of displacement basis functions.

\textbf{Stage I:}

 We will construct a set of auxiliary multiscale basis functions for displacement on each coarse element $K_i$ by solving a local spectral problem. First, we define some notation. For a general set $R$, we define $\bm{\mathit{U}}_{h,1}(R)$ as the restriction of $\bm{\mathit{U}}_{h,1}$ on $R\subset\Omega_1$ and $\underline{\bm{\mathit{\Sigma}}}_{h,1}(R)\coloneqq\{\underline{\bm{\tau}}_h\in\underline{\bm{\mathit{\Sigma}}}_{h,1}\colon\underline{\bm{\tau}}_h\mathbf{n}=\mathbf{0}\text{ on } \Omega_1\cap \partial R\}$. Note that $\underline{\bm{\mathit{\Sigma}}}_{h,1}(R)$ is with homogeneous traction boundary condition on $R$, which confirms the conforming property of the multiscale bases in the construction process.

Next, we define the local spectral problem. For each coarse element $K_i\subset\Omega_1$, we solve the eigenvalue problem: find $(\bm{\underline{\phi}}^i_j,\mathbf{p}^i_j)\in \underline{\bm{\mathit{\Sigma}}}_{h,1}(K_i)\times \bm{\mathit{U}}_{h,1}(K_i)$ and $\lambda^i_j\in\mathbb{R}$ such that
\begin{subequations}
	\label{local spectral problem mix}
	\begin{align}		
		(\mathcal{A}\bm{\underline{\phi}}^i_j,\underline{\bm{\tau}}_h)_{K_i}+(\textbf{div}\underline{\bm{\tau}}_h,\mathbf{p}^i_j)_{K_i}+\beta\int_{\gamma\cap\partial K_i}\phi^i_{j,n}\tau_{hn}\,ds&=0  \quad \forall \underline{\bm{\tau}}_h\in \underline{\bm{\mathit{\Sigma}}}_{h,1}(K_i), \label{local spectral problem_a}\\
		-(\textbf{div}\bm{\underline{\phi}}^i_j,\mathbf{v}_h)_{K_i}&= \lambda^i_js_i(\mathbf{p}^i_j,\mathbf{v}_h)\quad \forall \mathbf{v}_h\in \bm{\mathit{U}}_{h,1}(K_i),\label{local spectral problem_d}
	\end{align}	
\end{subequations}
where $\phi^i_{j,n}=\mathbf{n}_1\cdot(\bm{\underline{\phi}}^i_j\cdot\mathbf{n}_1)$, $\tau_{hn}=\mathbf{n}_1\cdot(\underline{\bm{\tau}}_h\cdot\mathbf{n}_1)$.
We arrange the eigenvalues of (\ref{local spectral problem mix}) in non-decreasing order $0\leq\lambda_1\leq\lambda_2\leq\cdot\cdot\cdot\leq\lambda_{L'_i}$, where $L'_i$ is the dimension of the space $\bm{\mathit{U}}_{h,1}(K_i)$. For each $i\in\{1,2,\cdot\cdot\cdot,N\}$, choose the first $l_i$ $(1\leq l_i\leq L'_i)$ eigenfunctions $\{\mathbf{p}^i_j\}_{j=1}^{l_i}$ corresponding to the first $l_i$ smallest eigenvalues. 
For the numerical tests, we used $l_i = 3$ on each $K_i$ ($1 \leq i \leq N$), which is sufficient to ensure the required accuracy (though larger $l_i$ can further enhance it, as shown in \cite[Table 5.2]{chung2025locking}).
Then we define the local auxiliary multiscale space $U_\text{aux}(K_i)$ for displacement as $U_\text{aux}(K_i)\coloneqq\text{span}\{\mathbf{p}^i_j\mid 1\leq j\leq l_i\}.$ The global auxiliary multiscale space $U_\text{aux}$ is defined by $U_\text{aux}=\oplus_iU_\text{aux}(K_i)$. $U_\text{aux}$ is the approximation space for displacement.

\textbf{Stage II: }

Then we present the construction of the stress basis functions. We directly present the relaxed version of stress multiscale basis functions (see \cite{chung2025locking}). Let $\mathbf{p}^i_j\in U_\text{aux}$ be a given displacement basis function supported in $K_i$. We will define a stress basis function $\bm{\underline{\psi}}^i_{j,\text{ms}}\in\underline{\bm{\mathit{\Sigma}}}_h(K_{i,m})$ by solving (\ref{multiscale_basis_mix}). The multiscale space is defined as $\Sigma_\text{ms}\coloneqq\text{span}\{\bm{\underline{\psi}}^i_{j,\text{ms}}\}$. Note that the basis function is supported in $K_{i,m}$, which is a union of connected coarse elements and contains $K_i$. We define $J_i$ as the set of indices such that if $k\in J_i$, then $K_k\in K_{i,m}$. We also define $U_\text{aux}(K_{i,m})=\text{span}\{\mathbf{p}^k_j\mid 1\leq j\leq l_k,k\in J_i\}$.

We find $\bm{\underline{\psi}}^i_{j,\text{ms}}\in\underline{\bm{\mathit{\Sigma}}}_{h,1}(K_{i,m})$ and $\mathbf{q}^i_{j,\text{ms}}\in \bm{\mathit{U}}_{h,1}(K_{i,m})$ such that
\begin{subequations}
	\label{multiscale_basis_mix}
	\begin{align}		
	(\mathcal{A}\bm{\underline{\psi}}^i_{j,\text{ms}},\underline{\bm{\tau}}_h)_{\Omega_1}+(\textbf{div}\underline{\bm{\tau}}_h,\mathbf{q}^i_{j,\text{ms}})_{\Omega_1}+\beta\int_{\gamma}\psi^i_{j,\text{msn}}\tau_{hn}\,ds&=0  \quad \forall \underline{\bm{\tau}}_h\in \underline{\bm{\mathit{\Sigma}}}_{h,1}(K_{i,m}), \label{multiscale_basis_mix_a}\\
s(\pi\mathbf{q}^i_{j,\text{ms}},\pi\mathbf{v}_h)-(\textbf{div}\bm{\underline{\psi}}^i_{j,\text{ms}},\mathbf{v}_h)_{\Omega_1}&= s(\mathbf{p}^i_j,\mathbf{v}_h)\quad \forall \mathbf{v}_h\in \bm{\mathit{U}}_{h,1}(K_{i,m}),\label{multiscale_basis_mix_d}
	\end{align}	
\end{subequations}
where $\psi^i_{j,\text{msn}}=\mathbf{n}_1\cdot(\bm{\underline{\psi}}^i_{j,\text{ms}}\cdot\mathbf{n}_1)$.

Similar to Section \ref{standard cem}, to estimate $(\underline{\bm{\sigma}}_{\text{ms},1},\mathbf{u}_{\text{ms},1})$ described in Section \ref{Discrete formulation with mixed CEM-GMsFEM}, we detailed the following four steps (to clarify the mixed multiscale reduction procedure):

\textbf{Step i:} Find $(\mathcal{Q}_i^m \mathbf{g}_{12}^{mix},\mathcal{N}_i^m \mathbf{g}_{12}^{mix})$ such that for all $(\underline{\bm{\mathit{\tau}}}_h,\mathbf{v}_h)\in \underline{\bm{\mathit{\Sigma}}}_{h,1}(K_{i,m})\times \bm{\mathit{U}}_{h,1}(K_{i,m})$ (Here, $\mathcal{Q}_i^m$ is an operator mapping from $\mathbf{V}_h(\gamma)$ to $\underline{\bm{\mathit{\Sigma}}}_{h,1}(K_{i,m})$ and $\mathcal{N}_i^m$ is defined analogously),
\begin{subequations}
\label{mixed_bdc_1}
\begin{align}		
    &(\mathcal{A}(\mathcal{Q}_i^m \mathbf{g}_{12}^{mix}),\underline{\bm{\tau}}_{h})_{\Omega_1} 
    + (\textbf{div}\,\underline{\bm{\tau}}_{h},\mathcal{N}_i^m \mathbf{g}_{12}^{mix})_{\Omega_1}
    + \beta\int_{\gamma}^*(\mathcal{Q}_i^m \mathbf{g}_{12}^{mix})_n\tau_{hn}\,ds= \int_{\gamma\cap\partial K_i}^*\mathbf{g}_{12}^{mix}\cdot(\underline{\bm{\tau}}_{h}\cdot\mathbf{n}_1)\,ds,
    \label{mixed_bdc_1_a} \\
    &s(\pi(\mathcal{Q}_i^m \mathbf{g}_{12}^{mix}),\pi\mathbf{v}_h)-(\textbf{div}\,(\mathcal{Q}_i^m \mathbf{g}_{12}^{mix}),\mathbf{v}_{h})_{\Omega_1} 
    = 0.
    \label{mixed_bdc_1_b}
\end{align}	
\end{subequations}
Here $\tau_{hn}=\mathbf{n}_1\cdot(\underline{\bm{\tau}}_{h}\cdot\mathbf{n}_1)$. Then we obtain $\mathcal{Q}^m \mathbf{g}_{12}^{mix}=\sum_{i=1}^N \mathcal{Q}_i^m \mathbf{g}_{12}^{mix}$, $\mathcal{N}^m \mathbf{g}_{12}^{mix}=\sum_{i=1}^N \mathcal{N}_i^m \mathbf{g}_{12}^{mix}$.

\textbf{Step ii:} Prepare the multiscale space $U_\text{aux}$, $\Sigma_{\text{ms}}$ by solving (\ref{local spectral problem mix}) and (\ref{multiscale_basis_mix}).

\textbf{Step iii:} 
Solve $(\underline{\bm{r}}^m,\mathbf{w}^m)$ such that for all $(\underline{\bm{\tau}},\mathbf{v})\in \Sigma_\text{ms}\times U_{\text{aux}}$,
\begin{subequations}
\label{mixed_correct_bdc_1}
\begin{align}		
    &(\mathcal{A}\underline{\bm{r}}^m,\underline{\bm{\tau}})_{\Omega_1} 
    + (\textbf{div}\,\underline{\bm{\tau}},\mathbf{w}^m)_{\Omega_1} 
+\beta\int_{\gamma}^*(\underline{r}^m)_n\tau_n\,ds=\int_{\gamma}^*\mathbf{g}_{12}^{mix}\cdot(\underline{\bm{\tau}}\cdot\mathbf{n}_1)ds \nonumber \\
&\quad -\big((\mathcal{A}(\mathcal{Q}^m \mathbf{g}_{12}^{mix}),\underline{\bm{\tau}})_{\Omega_1} 
    + (\textbf{div}\,\underline{\bm{\tau}},\mathcal{N}^m \mathbf{g}_{12}^{mix})_{\Omega_1} 
+\beta\int_{\gamma}^*(\mathcal{Q}^m \mathbf{g}_{12}^{mix})_n\tau_n\,ds\big),
    \label{mixed_correct_bdc_1_a} \\
    &(\textbf{div}\,\underline{\bm{r}}^m,\mathbf{v})_{\Omega_1}  
    = s(\pi(\tilde{k}^{-1}\mathbf{f}),\mathbf{v})-(\textbf{div}\,(\mathcal{Q}^m \mathbf{g}_{12}^{mix}),\mathbf{v})_{\Omega_1}.
    \label{mixed_correct_bdc_1_b}
\end{align}	
\end{subequations}
Here $(\underline{r}^m)_n=\mathbf{n}_1\cdot(\underline{\bm{r}}^m\cdot\mathbf{n}_1),\tau_n=\mathbf{n}_1\cdot(\underline{\bm{\tau}}\cdot\mathbf{n}_1)$.

\textbf{Step iv:} Construct the numerical solution as 
\(
\underline{\bm{\sigma}}_{\text{ms},1} \approx \underline{\bm{r}}^m+\mathcal{Q}^m \mathbf{g}_{12}^{mix},\) \(\mathbf{u}_{\text{ms},1} \approx \mathbf{w}^m+\mathcal{N}^m \mathbf{g}_{12}^{mix}.
\)

\subsubsection{An iterative contact-resolving hybrid method associated with mixed CEM-GMsFEM}
\label{iterative domain decomposition method for mixed problem}
In this section, we formulate the iterative procedure associated with mixed CEM-GMsFEM as Algorithm \ref{alg:BB} and establish its convergence in Theorem \ref{convergence in mixed cem}.

\begin{algorithm}
		\caption{an iterative contact-resolving hybrid algorithm by using mixed multiscale techniques}  
		\label{alg:BB}
		\begin{itemize}
			\item[1.] Given $\mathbf{g}_{12}^0, \mathbf{g}_{21}^0\in \mathbf{V}_h\left(\gamma\right)$ arbitrarily.
			\item[2.] Recursively find $(\underline{\bm{\sigma}}^n_1,\mathbf{u}_1^n)\in\Sigma_{\text{ms}}\times U_\text{aux},(\underline{\bm{\sigma}}^n_2,\mathbf{u}_2^n)\in\underline{\bm{\Sigma}}_{h,2}\times\bm{U}_{h,2}$ by solving  the subproblems in parallel:

            (i) To obtain $(\underline{\bm{\sigma}}^n_1,\mathbf{u}_1^n)\in\Sigma_{\text{ms}}\times U_\text{aux}$, we consider the following three stages:

            \textbf{Stage a:} Find $(\mathcal{Q}_i^m \mathbf{g}_{12}^{n},\mathcal{N}_i^m \mathbf{g}_{12}^{n})$ such that for all $(\underline{\bm{\mathit{\tau}}}_h,\mathbf{v}_h)\in \underline{\bm{\mathit{\Sigma}}}_{h,1}(K_{i,m})\times \bm{\mathit{U}}_{h,1}(K_{i,m})$,
\[\begin{aligned}		
    &(\mathcal{A}(\mathcal{Q}_i^m \mathbf{g}_{12}^{n}),\underline{\bm{\tau}}_{h})_{\Omega_1} 
    + (\textbf{div}\,\underline{\bm{\tau}}_{h},\mathcal{N}_i^m \mathbf{g}_{12}^{n})_{\Omega_1}
    + \beta\int_{\gamma}^*(\mathcal{Q}_i^m \mathbf{g}_{12}^{n})_n\tau_{hn}\,ds = \int_{\gamma\cap\partial K_i}^*\mathbf{g}_{12}^{n}\cdot(\underline{\bm{\tau}}_{h}\cdot\mathbf{n}_1)\,ds,
    \label{mixed_bdc_iteration_1_a}  \\
    &s(\pi(\mathcal{Q}_i^m \mathbf{g}_{12}^{n}),\pi\mathbf{v}_h)-(\textbf{div}\,(\mathcal{Q}_i^m \mathbf{g}_{12}^{n}),\mathbf{v}_{h})_{\Omega_1} 
    = 0.
\end{aligned}	
\]
Here $\tau_{hn}=\mathbf{n}_1\cdot(\underline{\bm{\tau}}_{h}\cdot\mathbf{n}_1)$, $(\mathcal{Q}_i^m \mathbf{g}_{12}^{n})_n= \mathbf{n}_1\cdot(\mathcal{Q}_i^m \mathbf{g}_{12}^{n}\cdot\mathbf{n}_1)$. Then we obtain $\mathcal{Q}^m \mathbf{g}_{12}^{n}=\sum_{i=1}^N \mathcal{Q}_i^m \mathbf{g}_{12}^{n}$, $\mathcal{N}^m \mathbf{g}_{12}^{n}=\sum_{i=1}^N \mathcal{N}_i^m \mathbf{g}_{12}^{n}$.

\textbf{Stage b:} Solve $(\underline{\bm{r}}^m_n,\mathbf{w}^m_n)$ such that for all $(\underline{\bm{\tau}},\mathbf{v})\in \Sigma_\text{ms}\times U_{\text{aux}}$,
\[
\begin{aligned}		
    &(\mathcal{A}\underline{\bm{r}}^m_n,\underline{\bm{\tau}})_{\Omega_1} 
    + (\textbf{div}\,\underline{\bm{\tau}},\mathbf{w}^m_n)_{\Omega_1} 
+\beta\int_{\gamma}^*(\underline{r}^m_n)_n\tau_n\,ds=\int_{\gamma}^*\mathbf{g}_{12}^{n}\cdot(\underline{\bm{\tau}}\cdot\mathbf{n}_1)\,ds  \\
&\quad -\big((\mathcal{A}(\mathcal{Q}^m \mathbf{g}_{12}^{n}),\underline{\bm{\tau}})_{\Omega_1} 
    + (\textbf{div}\,\underline{\bm{\tau}},\mathcal{N}^m \mathbf{g}_{12}^{n})_{\Omega_1} 
+\beta\int_{\gamma}^*(\mathcal{Q}^m \mathbf{g}_{12}^{n})_n\tau_n\,ds\big) \\
    &(\textbf{div}\,\underline{\bm{r}}^m_n,\mathbf{v})_{\Omega_1}  
    = s(\pi(\tilde{k}^{-1}\mathbf{f}),\mathbf{v})-(\textbf{div}\,(\mathcal{Q}^m \mathbf{g}_{12}^{n}),\mathbf{v})_{\Omega_1}.
\end{aligned}	
\]
Here $(\underline{r}^m_n)_n=\mathbf{n}_1\cdot(\underline{\bm{r}}^m\cdot\mathbf{n}_1)$, $\tau_n=\mathbf{n}_1\cdot(\underline{\bm{\tau}}\cdot\mathbf{n}_1)$. 

\textbf{Stage c:} Construct the numerical solution $(\underline{\bm{\sigma}}^n_1,\mathbf{u}^n_1)$ as 
\(
\underline{\bm{\sigma}}^n_1 \approx \underline{\bm{r}}^m_n+\mathcal{Q}^m \mathbf{g}_{12}^{n},\) 
\(
 \mathbf{u}^n_1 \approx \mathbf{w}^m_n+\mathcal{N}^m \mathbf{g}_{12}^{n}.
\)

(ii) Obtain $(\underline{\bm{\sigma}}^n_2,\mathbf{u}_2^n)\in\underline{\bm{\Sigma}}_{h,2}\times\bm{U}_{h,2}$ by solving the same equation (\ref{eq:system2}) in Algorithm \ref{an iterative dd algorithm for continuous setting in mixed form}.	
			\item[3.] Update the data by applying the same transmission condition as that specified in Algorithm \ref{an iterative dd algorithm for continuous setting in mixed form}.
		\end{itemize} 
\end{algorithm} 
\begin{theorem}\label{convergence in mixed cem}
Let $(\underline{\bm{\sigma}}_h,\mathbf{u}_h) \in \underline{\bm{\mathit{\Sigma}}}_h \times \bm{\mathit{U}}_h$ be the solution of problem (\ref{penalty_mixed_discrete_formula}) and $\underline{\bm{\sigma}}_{h,i}=\underline{\bm{\sigma}}_h|_{\Omega_i},\mathbf{u}_{h,i}=\mathbf{u}_h|_{\Omega_i}$ for $i=1,2$. 
For each iteration step $n \geq 0$, let $(\underline{\bm{\sigma}}_1^n, \mathbf{u}_1^n) \in\Sigma_{\mathup{ms}}\times U_\mathup{aux}$, $(\underline{\bm{\sigma}}_2^n, \mathbf{u}_2^n) \in\underline{\bm{\mathit{\Sigma}}}_{h,2}\times\bm{\mathit{U}}_{h,2}$ be the numerical solutions generated by Algorithm \ref{alg:BB}. Then there exists a positive sequence $\{ \xi_n\}$ with $ \xi_n \to 0$ as $n \to \infty$ such that
\[
\left\|\underline{\mathbf{e}}^n_\sigma\right\|_e^2 + \left\|\mathbf{e}^n_u\right\|_{L^2}^2 \leq  \xi_n \to 0 \text{ as } n \rightarrow \infty,
\]
where the norms $\left\|\underline{\mathbf{e}}^n_\sigma\right\|_e, \left\|\mathbf{e}^n_u\right\|_{L^2} $ are defined as in Section \ref{iterative contact-resolving hybrid method for the mixed discretization 2}, with the only modification being the replacement of the space $\underline{\bm{\mathit{\Sigma}}}_{h,1}\times\bm{\mathit{U}}_{h,1}$ by $\Sigma_{\mathup{ms}}\times U_\mathup{aux}$.
\end{theorem}
\section{Analysis of the iterative contact-resolving hybrid methods}\label{Analysis}
In this section, we present convergence analyses for the algorithms proposed in Section \ref{various discretizations}. We first prove the equivalence between the split discrete subproblems and the global problem (Theorem \ref{equivalent discrete form}) in Section \ref{proof_1}. The convergence analyses of Algorithms \ref{alg:B}-\ref{alg:BB} are then presented in Sections \ref{standard analysis}-\ref{proof for mixed cem}, respectively.

\subsection{Proof of Theorem \ref{equivalent discrete form}} 
\label{proof_1}
\begin{proof}[proof of Theorem \ref{equivalent discrete form}]
It is clear that (\ref{discrete form}) is equivalent to (note that $\left\{\bm{\varphi}_p^1,\bm{\varphi}_p^2\right\}_{p \in N_o}$ is a nodal basis of $\mathbf{V}_h$):
$
a\left(\mathbf{u}_h, \bm{\varphi}_p^l\right) + \frac{1}{\delta} \int_{\Gamma_C} \left(u_{hc}\right)^{+}\varphi^l_{p,c} \, ds  = \left(\mathbf{f}, \bm{\varphi}^l_p\right) $ for all $ p\in N_o,l=1,2,$
where $\varphi^l_{p,c}=\bm{\varphi}^l_p\cdot\mathbf{n}_c$. Since $\bm{\varphi}^1_p(p) = (1,0),\bm{\varphi}^2_p(p) = (0,1)$ and $\bm{\varphi}^i_p(p') = \mathbf{0}$ for all $p'\in N_o$, $p'\neq p$ and $l=1,2,$ then $
a_1\left(\mathbf{u}_h, \bm{\varphi}^l_p\right) + a_2\left(\mathbf{u}_h, \bm{\varphi}^l_p\right) +\frac{1}{\delta} \int_{\Gamma_C} \left(u_{hc,2}\right)^{+}\varphi^l_{p,c} \, ds \, ds = \left(\mathbf{f}, \bm{\varphi}^l_p\right)_{\Omega_1} + \left(\mathbf{f}, \bm{\varphi}^l_p\right)_{\Omega_2}$ for all $p \in N_\gamma,
$ $l=1,2$. That is, we have
\begin{align}
\label{discrete form for basis test}
a_1\left(\mathbf{u}_h, \bm{\varphi}^l_p\right) - \left(\mathbf{f}, \bm{\varphi}^l_p\right)_{\Omega_1} = -\big(a_2\left(\mathbf{u}_h, \bm{\varphi}^l_p\right)+\frac{1}{\delta} \int_{\Gamma_C} \left(u_{hc,2}\right)^{+}\varphi^l_{p,c} \, ds - \left(\mathbf{f}, \bm{\varphi}^l_p\right)_{\Omega_2}\big)
\end{align}
for all $p \in N_\gamma$, and $l=1,2$.
For $\gamma$, we define $G_{12}^p$ and $G_{21}^p$ as follows:
$$
\begin{aligned}
G_{12}^{p,l} = -\frac{1}{w_p}\big(a_2\left(\mathbf{u}_h, \bm{\varphi}_p^l\right) - \left(\mathbf{f}, \bm{\varphi}_p^l\right)_{\Omega_2} + \frac{1}{\delta} \int_{\Gamma_C} \left(u_{hc,2}\right)^{+}\varphi^l_{p,c} \, ds\big), \quad
G_{21}^{p,l} = -\frac{1}{w_p}\big(a_1\left(\mathbf{u}_h, \bm{\varphi}_p^l\right) - \left(\mathbf{f}, \bm{\varphi}_p^l\right)_{\Omega_1}\big).
\end{aligned}
$$
Then we construct $\mathbf{g}_{12}^*, \mathbf{g}_{21}^* \in \mathbf{V}_h\left(\gamma\right)$, $\operatorname{meas}\left(\gamma\right) > 0$, such that
\begin{equation}
\label{choosing g star}
\begin{aligned}
\mathbf{g}_{12}^*(x) = \sum_{p \in N_\gamma,l=1,2} \left(\alpha u_{h,1}^l(p) +  G_{12}^{p,l}\right) \bm{\varphi}^l_p(x), \quad
\mathbf{g}_{21}^*(x) = \sum_{p \in N_\gamma,l=1,2} \left(\alpha u_{h,2}^l(p) + G_{21}^{p,l}\right) \bm{\varphi}^l_p(x),
\end{aligned}
\end{equation}
where we note that
$u^l_{h,1}(p)$ and $u^l_{h,2}(p)$ ($l=1,2$) are scalar values and they correspond to nodal values of the $l$-th component of the solutions $\mathbf{u}_{h,1}(x)$ and $\mathbf{u}_{h,2}(x)$, i.e.,
$\mathbf{u}_{h,k}=\sum_{p \in N_\gamma,l=1,2} u^l_{h,k}(p) \bm{\varphi}^l_p(x),$ $k=1,2.$

Denote the $m$-th component of a vector function $\mathbf{v}$ as $(\mathbf{v})_m$. Clearly $\big(\bm{\varphi}^l_p(p)\big)_m=\delta_{lm}$ for $l,m=1,2$. By (\ref{choosing g star}), we know
$(\mathbf{g}_{12}^*)_m, (\mathbf{g}_{21}^*)_m$ (for $m=1,2$) can be written as follows
\begin{subequations}
\label{choosing g star component}
\begin{align}
\big(\mathbf{g}_{12}^*(x)\big)_1 &= \sum_{p \in N_\gamma} \left(\alpha u_{h,1}^1(p) +  G_{12}^{p,1}\right) (\bm{\varphi}^1_p(x))_1, \label{choosing g star component_a} \\ 
\big(\mathbf{g}_{12}^*(x)\big)_2 &= \sum_{p \in N_\gamma} \left(\alpha u_{h,1}^2(p) +  G_{12}^{p,2}\right) (\bm{\varphi}^2_p(x))_2, \label{choosing g star component_b} \\
\big(\mathbf{g}_{21}^*(x)\big)_1 &= \sum_{p \in N_\gamma} \left(\alpha u_{h,2}^1(p) +  G_{21}^{p,1}\right) (\bm{\varphi}^1_p(x))_1,\label{choosing g star component_c} \\
\big(\mathbf{g}_{21}^*(x)\big)_2 &= \sum_{p \in N_\gamma} \left(\alpha u_{h,2}^2(p) +  G_{21}^{p,2}\right) (\bm{\varphi}^2_p(x))_2.\label{choosing g star component_d}
\end{align}
\end{subequations}
Therefore, it follows from (\ref{discrete form for basis test}) and (\ref{choosing g star component_a}) that, for all $p \in \overline{\Omega}_1 \cap N_o$, $\mathbf{u}_{h, 1} = \mathbf{u}_h |_{\Omega_1}$ satisfies
$$
\begin{aligned}
&\quad a_1\left(\mathbf{u}_{h, 1}, \bm{\varphi}^1_p\right) - \left(\mathbf{f}, \bm{\varphi}^1_p\right)_{\Omega_1} = w_pG_{12}^{p,1} \\
&= w_p\left((\mathbf{g}_{12}^*(p))_1 - \alpha \mathbf{u}^1_{h, 1}(p)\right) = w_p\left(\mathbf{g}_{12}^*(p) - \alpha \mathbf{u}_{h, 1}(p)\right)\cdot\bm{\varphi}^1_p(p)=\int_{\gamma}^*\left(\mathbf{g}_{12}^* - \alpha \mathbf{u}_{h, 1}\right) \cdot \bm{\varphi}^1_p \, ds,
\end{aligned}
$$
where $\bm{\varphi}^1_p(p)=(1,0)$ has been used in the second line.
Thus, 
$
a_1\left(\mathbf{u}_{h, 1}, \bm{\varphi}^1_p\right) + \alpha \int_{\gamma}^* \mathbf{u}_{h, 1} \cdot \bm{\varphi}^1_p \, ds = \left(\mathbf{f}, \bm{\varphi}^1_p\right)_{\Omega_1} +  \int_{\gamma}^* \mathbf{g}_{12}^* \cdot \bm{\varphi}^1_p \, ds.
$
Similarly, by (\ref{choosing g star component_b})-(\ref{choosing g star component_d}), we obtain that
$$
\begin{aligned}
&\quad a_1\left(\mathbf{u}_{h, 1}, \bm{\varphi}^2_p\right) + \alpha \int_{\gamma}^* \mathbf{u}_{h, 1} \cdot \bm{\varphi}^2_p \, ds = \left(\mathbf{f}, \bm{\varphi}^2_p\right)_{\Omega_1} +  \int_{\gamma}^* \mathbf{g}_{12}^* \cdot \bm{\varphi}^2_p \, ds,\\
&\quad a_2\left(\mathbf{u}_{h, 2}, \bm{\varphi}^l_p\right) - \left(\mathbf{f}, \bm{\varphi}^l_p\right)_{\Omega_2} + \frac{1}{\delta} \int_{\Gamma_C} \left(u_{hc, 2}\right)^{+} \cdot {\varphi}_{p,c}^l \, ds= w_p G_{21}^{p,l} \\
&= w_p\big(\mathbf{g}_{21}^*(p) - \alpha \mathbf{u}_{h, 2}(p)\big)\cdot \bm{\varphi}^l_p(p)  =  \int_{\gamma}\left(\mathbf{g}_{21}^* - \alpha \mathbf{u}_{h, 2}\right) \cdot \bm{\varphi}^l_p \, ds
\end{aligned}
$$
for $l=1,2$.
That is, we have,
\begin{equation*}
\begin{aligned}
a_1\left(\mathbf{u}_{h, 1}, \bm{\varphi}^l_p\right) + \alpha \int_{\gamma}^* \mathbf{u}_{h, 1} \cdot \bm{\varphi}^l_p \, ds &= \left(\mathbf{f}, \bm{\varphi}^l_p\right)_{\Omega_1} +  \int_{\gamma}^* \mathbf{g}_{12}^* \cdot \bm{\varphi}^l_p \, ds,\\
a_2\left(\mathbf{u}_{h, 2}, \bm{\varphi}^l_p\right) + \alpha \int_{\gamma}^* \mathbf{u}_{h, 2} \cdot \bm{\varphi}^l_p \, ds + \frac{1}{\delta} \int_{\Gamma_C} \left(u_{hc, 2}\right)^{+} \cdot {\varphi}_{p,c}^l \, ds &= \left(\mathbf{f}, \bm{\varphi}^l_p\right)_{\Omega_2} + 
 \int_{\gamma}^* \mathbf{g}_{21}^* \cdot \bm{\varphi}^l_p \, ds,
 \end{aligned}
\end{equation*}
for all $p \in N_\gamma$, and $l=1,2$.
The proof is completed.
\end{proof}

\subsection{Convergence for Algorithm \ref{alg:B}}\label{standard analysis}
This section is devoted to the proof of Theorem \ref{convergence in discrete setting}, which establishes the convergence of Algorithm \ref{alg:B}. To this end, we first present two auxiliary lemmas and then proceed to the final proof.

For the ease of notation, we denote (based on the notation in Section \ref{Nonoverlapping domain decomposition method associated with Standard FEM formulation})
$$
\begin{aligned}
 \mathbf{u}^n = \left(\mathbf{u}_i^n\right)_{i=1,2} \in \prod_{i=1}^2 \mathbf{V}_{h,i}, \quad \mathbf{e}^n_h = \left(\mathbf{e}_{h,i}^n\right)_{i=1,2} := \left(\mathbf{u}_i^n - \mathbf{u}_{h,i}\right)_{i=1,2} \in \prod_{i=1}^2 \mathbf{V}_{h,i},
\end{aligned}
$$
and the norms as follows: for $i=1,2,$
$$
\left\|\mathbf{e}_{h,i}^n\right\|_{a_i}^2=a_i\left(\mathbf{e}_{h,i}^n, \mathbf{e}_{h,i}^n\right),\quad\left\|\mathbf{e}^n_h\right\|_e^2 = \sum_{i=1}^2 \left\|\mathbf{e}_{h,i}^n\right\|_{a_i}^2 + \frac{1}{\delta}\int_{\Gamma_C} \left|\left(u_{2,c}^n\right)^{+} - \left(u_{hc,2}\right)^{+}\right|^2 \, ds,
$$
where $u_{2,c}^n=\mathbf{u}_2^n\cdot\mathbf{n}_c,u_{hc,2}=\mathbf{u}_{h,2}\cdot\mathbf{n}_c$.		
Then we clearly have that
\begin{align}
\label{energy_estimate_1}
a_1\left(\mathbf{e}_{h,1}^n, \mathbf{v}_{h,1}\right) + \alpha \int_{\gamma}^* \mathbf{e}_{h,1}^n \cdot \mathbf{v}_{h,1} \, ds =  \int_{\gamma}^* \mathbf{g}_{12}^n \cdot \mathbf{v}_{h,1} \, ds \quad \forall \mathbf{v}_{h,1} \in \mathbf{V}_{h,1}.
\end{align}
\begin{equation}
\label{energy_estimate_2}
\begin{aligned}
&\quad a_2\left(\mathbf{e}_{h,2}^n, \mathbf{v}_{h,2}\right) + \frac{1}{\delta} \int_{\Gamma_C} \left( \left(u_{2,c}^n\right)^{+} - \left(u_{hc,2}\right)^{+}\right) v_{hc,2} \, ds + \alpha \int_{\gamma}^* \mathbf{e}_{h,2}^n \cdot \mathbf{v}_{h,2} \, ds \\
& =  \int_{\gamma}^* \mathbf{g}_{21}^n \cdot \mathbf{v}_{h,2} \, ds \quad \forall \mathbf{v}_{h,2} \in \mathbf{V}_{h,2}.
\end{aligned}
\end{equation}
\begin{align}
\label{g value at p}
\mathbf{g}_{12}^{n+1}(p) = 2 \alpha \mathbf{e}_{h,2}^n(p) - \mathbf{g}_{21}^n(p), \quad \mathbf{g}_{21}^{n+1}(p) = 2 \alpha \mathbf{e}_{h,1}^n(p) - \mathbf{g}_{12}^n(p)
\end{align}
on $p \in N_\gamma$. Here we have used $\mathbf{g}_{ij}^n$ to replace $\mathbf{g}_{ij}^n-\mathbf{g}_{ij}^{*}$( $i,j=1,2$, $i\neq j$) just for notation simplicity.
Then we have Lemma \ref{ah1_ah2} below.
\begin{lemma}
\label{ah1_ah2}
We have the following identities:
$$
\begin{aligned}
\label{equality for a}
a_1\left(\mathbf{e}_{h,1}^n, \mathbf{e}_{h,1}^n\right) &=  \int_{\gamma}^*\left(\mathbf{g}_{12}^n - \alpha \mathbf{e}_{h,1}^n\right) \cdot \mathbf{e}_{h,1}^n \, ds, \\
a_2\left(\mathbf{e}_{h,2}^n, \mathbf{e}_{h,2}^n\right) &= -\frac{1}{\delta} \int_{\Gamma_C} \left[\left(u_{2,c}^n\right)^{+} - \left(u_{hc,2}\right)^{+}\right] (u_{2,c}^n - u_{hc,2})\, ds +  \int_{\gamma}^*\left(\mathbf{g}_{21}^n - \alpha \mathbf{e}_{h,2}^n\right) \cdot \mathbf{e}_{h,2}^n \, ds.
\end{aligned}
$$
\end{lemma}
\begin{lemma}
\label{equality for g star and a}
There holds the following identity:
$$
\left\|\underline{\mathbf{g}}^{n+1}\right\|_*^2 = \left\|\underline{\mathbf{g}}^n\right\|_*^2 - 4\left[\sum_{i=1}^2 a_i\left(\mathbf{e}_{h,i}^n, \mathbf{e}_{h,i}^n\right) + \frac{1}{\delta} \int_{\Gamma_C} \left[\left(u_{2,c}^n\right)^{+} - \left(u_{hc,2}\right)^{+} \right](u_{2,c}^n - u_{hc,2}) \, ds\right],
$$
where
$
\left\|\underline{\mathbf{g}}^k\right\|_*^2 = \frac{1}{\alpha} \int_{\gamma}^* \left|\mathbf{g}_{12}^k\right|^2+\left|\mathbf{g}_{21}^k\right|^2 \, ds, \quad \underline{\mathbf{g}}^k = \left(\mathbf{g}_{12}^k, \mathbf{g}_{21}^k\right)$, $k=n,n+1.$
\end{lemma}
\begin{proof}
By combining the definition of the norm $\left\|\cdot\right\|_*$ and Lemma \ref{ah1_ah2}, we have
\begin{align*}
\left\|\underline{\mathbf{g}}^{n+1}\right\|_*^2 &= \frac{1}{\alpha} \int_{\gamma}^* \left|\mathbf{g}_{12}^{n+1}\right|^2 \, ds + \frac{1}{\alpha} \int_{\gamma}^* \left|\mathbf{g}_{21}^{n+1}\right|^2 \, ds = \frac{1}{\alpha} \int_{\gamma}^* \left|2 \alpha \mathbf{e}_{h,2}^n - \mathbf{g}_{21}^n\right|^2 \, ds + \frac{1}{\alpha} \int_{\gamma}^* \left|2 \alpha \mathbf{e}_{h,1}^n - \mathbf{g}_{12}^n\right|^2 \, ds\\
&= \frac{1}{\alpha} \int_{\gamma}^* \left|\mathbf{g}_{21}^n\right|^2 \, ds - 4 \int_{\gamma}^* \left(\mathbf{g}_{21}^n - \alpha \mathbf{e}_{h,2}^n\right) \cdot \mathbf{e}_{h,2}^n \, ds + \frac{1}{\alpha} \int_{\gamma}^* \left|\mathbf{g}_{12}^n\right|^2 \, ds - 4 \int_{\gamma}^* \left(\mathbf{g}_{12}^n - \alpha \mathbf{e}_{h,1}^n\right) \cdot \mathbf{e}_{h,1}^n \, ds  \\
&= \left\|\underline{\mathbf{g}}^n\right\|_*^2 - 4 a_1\left(\mathbf{e}_{h,1}^n, \mathbf{e}_{h,1}^n\right) - 4 \left[a_2\left(\mathbf{e}_{h,2}^n, \mathbf{e}_{h,2}^n\right) + \frac{1}{\delta} \int_{\Gamma_C} \left[\left(u_{2,c}^n\right)^{+} - \left(u_{hc,2}\right)^{+} \right](u_{2,c}^n - u_{hc,2}) \, ds\right] \\
&= \left\|\underline{\mathbf{g}}^n\right\|_*^2 - 4 \left[\sum_{i=1}^2 a_i\left(\mathbf{e}_{h,i}^n, \mathbf{e}_{h,i}^n\right) + \frac{1}{\delta} \int_{\Gamma_C} \left[\left(u_{2,c}^n\right)^{+} - \left(u_{hc,2}\right)^{+} \right](u_{2,c}^n - u_{hc,2}) \, ds\right].
\end{align*}
Then the desired equality is proved.
\end{proof}
Note that by (\ref{an_important_inequality}), we have
\begin{equation}
\label{convergence inequality}
\frac{1}{\delta} \int_{\Gamma_C} \left[\left(u_{2,c}^n\right)^{+} - \left(u_{hc,2}\right)^{+}\right] (u_{2,c}^n - u_{hc,2}) \, ds \geq  \frac{1}{\delta} \int_{\Gamma_C} \left[\left(u_{2,c}^n\right)^{+} - \left(u_{hc,2}\right)^{+}\right]^2 \, ds \geq 0.
\end{equation}
Next we give the proof of the final convergence result (i.e. Theorem \ref{convergence in discrete setting}).
\begin{proof}[proof of Theorem \ref{convergence in discrete setting}]
By using Lemma \ref{equality for g star and a}, we have that for any positive integer $M$:
$$
\begin{aligned}
\sum_{n=0}^M \sum_{i=1}^2 a_i\left(\mathbf{e}_{h,i}^n, \mathbf{e}_{h,i}^n\right) = & \sum_{n=0}^M \left[\frac{1}{4}\left(\left\|\underline{\mathbf{g}}^n\right\|_*^2 - \left\|\underline{\mathbf{g}}^{n+1}\right\|_*^2\right) - \frac{1}{\delta} \int_{\Gamma_C} \left[\left(u_{2,c}^n\right)^{+} - \left(u_{hc,2}\right)^{+} \right](u_{2,c}^n - u_{hc,2})  \, ds\right] \\
= & \frac{1}{4}\left(\|\underline{\mathbf{g}}^0\|_*^2 - \left\|\underline{\mathbf{g}}^{M+1}\right\|_*^2\right) - \sum_{n=0}^M \frac{1}{\delta} \int_{\Gamma_C} \left[\left(u_{2,c}^n\right)^{+} - \left(u_{hc,2}\right)^{+} \right](u_{2,c}^n - u_{hc,2})  \, ds.
\end{aligned}
$$
That is,
\begin{equation}
\begin{aligned}
\label{aaaaa}
&\quad \sum_{n=0}^M \left[\sum_{i=1}^2 a_i\left(\mathbf{e}_{h,i}^n, \mathbf{e}_{h,i}^n\right) + \frac{1}{\delta} \int_{\Gamma_C} \left[\left(u_{2,c}^n\right)^{+} - \left(u_{hc,2}\right)^{+} \right](u_{2,c}^n - u_{hc,2}) \, ds\right] = \frac{1}{4}\left(\|\underline{\mathbf{g}}^0\|_*^2 - \left\|\underline{\mathbf{g}}^{M+1}\right\|_*^2\right)\leq \frac{1}{4}\|\underline{\mathbf{g}}^0\|_*^2.
\end{aligned}
\end{equation}
In terms of (\ref{aaaaa}) and (\ref{convergence inequality}), we have
\begin{equation}
\label{limit}
\sum_{i=1}^2 a_i\left(\mathbf{e}_{h,i}^n, \mathbf{e}_{h,i}^n\right) + \frac{1}{\delta} \int_{\Gamma_C} \left[\left(u_{2,c}^n\right)^{+} - \left(u_{hc,2}\right)^{+} \right]^2 \, ds \rightarrow 0 \text{ as } n \rightarrow \infty. 
\end{equation}
This completes the proof.
\end{proof}
\subsection{Convergence for Algorithm \ref{an iterative dd algorithm for continuous setting in mixed form}}
\label{proof for mixed}
In this section, we prove Theorem \ref{convergence for mixed}, which confirms the convergence of Algorithm \ref{an iterative dd algorithm for continuous setting in mixed form}.
\begin{proof}[proof of Theorem \ref{convergence for mixed}]
Using the error notation $\underline{\mathbf{e}}^n_\sigma$ and $\mathbf{e}^n_u$ from Section \ref{iterative contact-resolving hybrid method for the mixed discretization 2}, we subtract equations (\ref{penalty_mixed_discrete_formula_seperate_1})-(\ref{penalty_mixed_discrete_formula_seperate_2}) from (\ref{eq:system})-(\ref{eq:system2}) to obtain
\begin{subequations}
\label{eq:main_system}
\begin{align}
    &\inner{\mathcal{A} \ube{\sigma}{1}}{\btau_{h,1}}_{\Omega_1} + \inner{\textbf{div} \btau_{h,1}}{\bedis{u}{1}} 
      + \beta \int_{\gamma}^* \left(e^n_{\sigma,1}\right)_{n} \tau_{h n,1} \, ds = \int_{\Gamma_C}\mathbf{g}_{12}^n \cdot 
      \left(\btau_{h,1} \cdot \mathbf{n}_1\right) ds, 
      && \forall \btau_{h,1} \in \underline{\bm{\mathit{\Sigma}}}_{h,1}, \label{eq:3.2a} \\
    &\inner{\textbf{div} \ube{\sigma}{1}}{\mathbf{v}_{h,1}}_{\Omega_1} = 0, 
      && \forall \mathbf{v}_{h,1}\in \bm{\mathit{U}}_{h,1}, \label{eq:3.2b} \\[6pt]
    &\inner{\mathcal{A} \ube{\sigma}{2}}{\btau_{h,2}}_{\Omega_2} + \inner{\textbf{div} \btau_{h,2}}{\bedis{u}{2}} 
      + \beta \int_{\gamma}^* \left(e^n_{\sigma,2}\right)_{n} \tau_{h n,2} \, ds \notag \\
    &\quad + \frac{1}{\delta} \int_{\Gamma_C} \left[ \left(\sigma_{2,c}^n\right)^+ 
      - \left(\sigma_{h c,2}\right)^+ \right] \cdot \tau_{h c,2} \, ds = \int_{\gamma}^* \mathbf{g}_{21}^n \cdot 
      \left(\btau_{h,2} \cdot \mathbf{n}_2\right) ds, 
      && \forall \btau_{h,2} \in \underline{\bm{\mathit{\Sigma}}}_{h,2}, \label{eq:3.2c} \\
    &\inner{\textbf{div} \ube{\sigma}{2}}{\mathbf{v}_{h,2}}_{\Omega_2} = 0, 
      && \forall \mathbf{v}_{h,2} \in \bm{\mathit{U}}_{h,2}, \label{eq:3.2d} \\[6pt]
    &\mathbf{g}_{12}^{n+1}(p) = -2\beta \, \ube{\sigma}{2}(p) \cdot \mathbf{n}_2 + \mathbf{g}_{21}^n(p), \quad \mathbf{g}_{21}^{n+1}(p) = -2\beta \, \ube{\sigma}{1}(p) \cdot \mathbf{n}_1 + \mathbf{g}_{12}^n(p). \label{eq:3.2f}
\end{align}
\end{subequations}
for $p\in N_\gamma$, where $(e^n_{\sigma,i})_n=\mathbf{n}_i\cdot(\ube{\sigma}{i}\cdot\mathbf{n}_i)$ ($i=1,2$).
For notational simplicity, we have used \(\mathbf{g}_{12}^n\) and \(\mathbf{g}_{21}^n\) to represent 
\(\mathbf{g}_{12}^n - \mathbf{g}_{12}^{mix}\) and \(\mathbf{g}_{21}^n - \mathbf{g}_{21}^{mix}\), respectively.
Let the test functions \(\btau_{h,1} = \ube{\sigma}{1}\), \(\mathbf{v}_{h,1} = \bedis{u}{1}\), 
\(\btau_{h,2} = \ube{\sigma}{2}\), and \(\mathbf{v}_{h,2} = \bedis{u}{2}\) in 
equations \eqref{eq:3.2a}--\eqref{eq:3.2d}. We then obtain:
\begin{subequations}
\label{eq:energy_relations}
\begin{align}
    \inner{\mathcal{A} \ube{\sigma}{1}}{\ube{\sigma}{1}}_{\Omega_1} 
    &= \int_{\gamma}^* \left[\mathbf{g}_{12}^n - \beta \left(e^n_{\sigma,1}\right)_n\mathbf{n}_1 \right] 
       \cdot \left(\ube{\sigma}{1} \cdot \mathbf{n}_1\right) ds, \label{eq:energy1} \\
    \inner{\mathcal{A} \ube{\sigma}{2}}{\ube{\sigma}{2}}_{\Omega_2} 
    &= \int_{\gamma}^* \left[\mathbf{g}_{21}^n - \beta \left(e^n_{\sigma,2}\right)_{n}\mathbf{n}_2 \right] 
       \cdot \left(\ube{\sigma}{2} \cdot \mathbf{n}_2\right) ds - \frac{1}{\delta} \int_{\Gamma_C} \left[ \left(\sigma_{2,c}^n\right)^+ 
      - \left(\sigma_{h c,2}\right)^+ \right] \cdot 
      \left(\sigma_{2,c}^n - \sigma_{h c,2}\right) ds. \label{eq:energy2}
\end{align}
\end{subequations}
Then, by \eqref{eq:energy_relations}, \eqref{eq:3.2f} and the definition of $\left\|\underline{\mathbf{e}}^n_\sigma\right\|_{\mathcal{A}}$ from Section \ref{iterative contact-resolving hybrid method for the mixed discretization 2}, we can obtain:
\begin{align}
    &\quad\frac{1}{\beta} \int_{\gamma}^* \left( \left|\mathbf{g}_{12}^{n+1} \right|^2 
      + \left| \mathbf{g}_{21}^{n+1} \right|^2 \right) ds = \frac{1}{\beta} \int_{\gamma}^* \left| -2\beta \, \ube{\sigma}{2} \cdot \mathbf{n}_2 
      +\mathbf{g}_{21}^n \right|^2 ds 
      + \frac{1}{\beta} \int_{\gamma}^* \left| -2\beta \, \ube{\sigma}{1} \cdot \mathbf{n}_1 
      +\mathbf{g}_{12}^n \right|^2 ds \notag \\
    &= \frac{1}{\beta} \left[ \int_{\gamma}^* \left|\mathbf{g}_{21}^n \right|^2 ds 
      + \int_{\gamma}^* \left|\mathbf{g}_{12}^n \right|^2 ds \right] - 4 \left[ \left\|\underline{\mathbf{e}}^n_\sigma\right\|_{\mathcal{A}}^2
      + \frac{1}{\delta} \int_{\Gamma_C} \left[ \left( \sigma_{2,c}^n \right)^{+} 
      - \left( \sigma_{h c,2} \right)^{+} \right] \cdot 
      \left( \sigma_{2,c}^n - \sigma_{h c,2} \right) ds \right]. \label{eq:key_identity}
\end{align}
Note that by equation (\ref{an_important_inequality}), we clearly have
\begin{equation}
\label{eq:positivity}
\begin{aligned}
\frac{1}{\delta} \int_{\Gamma_C} \left[ \left( \sigma_{2,c}^n \right)^{+} - \left( \sigma_{h c,2} \right)^{+} \right] \cdot \left( \sigma_{2,c}^n - \sigma_{h c,2} \right) \, ds  \geqslant \frac{1}{\delta} \int_{\Gamma_C} \left| \sigma_{2,c}^n - \sigma_{h c,2} \right|^2 \, ds \geqslant 0.
\end{aligned}
\end{equation}
Next, by combining (\ref{eq:key_identity}) and applying the similar steps to the proof of Theorem \ref{convergence in discrete setting}, we obtain for any positive integer $M$ that
\begin{align}
\label{eq:sum_identity}
\begin{aligned}
\sum_{n=0}^M \left\|\underline{\mathbf{e}}^n_\sigma\right\|_{\mathcal{A}}^2
&= \frac{1}{4} \Bigg( \frac{1}{\beta} \int_{\gamma}^* \left( \left| \mathbf{g}_{21}^0 \right|^2 + \left| \mathbf{g}_{12}^0 \right|^2 \right) \, ds - \frac{1}{\beta} \int_{\gamma}^* \left( \left| \mathbf{g}_{12}^{M+1} \right|^2 + \left| \mathbf{g}_{21}^{M+1} \right|^2 \right) \, ds \Bigg) \\
&\qquad - \sum_{n=0}^M \frac{1}{\delta} \int_{\Gamma_C} \left[ \left( \sigma_{2,c}^n \right)^{+} - \left( \sigma_{h c,2} \right)^{+} \right] \cdot \left( \sigma_{2,c}^n - \sigma_{h c,2} \right) \, ds.
\end{aligned}
\end{align}
Then we obtain
$\sum_{n=0}^M \left( \left\|\underline{\mathbf{e}}^n_\sigma\right\|_{\mathcal{A}}^2 + \frac{1}{\delta} \int_{\Gamma_C} \left[ \left( \sigma_{2,c}^n \right)^{+} - \left( \sigma_{h c,2} \right)^{+} \right] \cdot \left( \sigma_{2,c}^n - \sigma_{h c,2} \right) \, ds \right) \leqslant \frac{1}{4\beta} \int_{\gamma}^* \left( \left| \mathbf{g}_{21}^0 \right|^2 + \left| \mathbf{g}_{12}^0 \right|^2 \right) \, ds < \infty.
$
Thus, combining equation (\ref{eq:positivity}) and the definitions of $\left\|\underline{\mathbf{e}}^n_\sigma\right\|_e, \left\|\underline{\mathbf{e}}^n_\sigma\right\|_{\mathcal{A}}$ in (\ref{important_norms}), we have
\begin{equation}
\label{eq:stress_convergence}
\begin{aligned}
\left\|\underline{\mathbf{e}}^n_\sigma\right\|_e^2 &\leqslant \left\|\underline{\mathbf{e}}^n_\sigma\right\|_{\mathcal{A}}^2 + \frac{1}{\delta} \int_{\Gamma_C} \left[ \left( \sigma_{2,c}^n \right)^{+} - \left( \sigma_{h c,2} \right)^{+} \right] \cdot \left( \sigma_{2,c}^n - \sigma_{h c,2} \right) \, ds \rightarrow 0 \quad \text{as } n \rightarrow \infty.
\end{aligned}
\end{equation}

Next, we estimate the displacement field errors using techniques that differ significantly from the analysis of Theorem \ref{convergence in discrete setting}. Following the approach in \cite[Theorem 4.4]{zhong2023spectral}, which employs discrete $H^1$ stability arguments, we select appropriate test functions $\btau_{h,1}$ and $\btau_{h,2}$ in (\ref{eq:3.2a}) and (\ref{eq:3.2c}) to derive a discrete $H^1$-norm estimate for $\bedis{u}{1}$ and $\bedis{u}{2}$. 
Then combining the discrete Sobolev embedding inequalities from \cite[Remark 4.5]{zhong2023spectral} and the definitions of $\left\|\mathbf{e}^n_u\right\|_{L^2},\left\|\underline{\mathbf{e}}^n_\sigma\right\|_{\mathcal{A}}$, we obtain
\begin{equation*}
\label{eq:displacement_bound}
\left\|\mathbf{e}^n_u\right\|_{L^2}^2 \leqslant C \left\|\underline{\mathbf{e}}^n_\sigma\right\|_{\mathcal{A}}^2\leq C \left\|\underline{\mathbf{e}}^n_\sigma\right\|_e^2 \rightarrow 0 \quad \text{as } n \rightarrow \infty,
\end{equation*}
where $C>0$ is independent of material parameters and mesh size. Therefore, we finally obtain the convergence result:
$
\left\|\mathbf{e}^n_u\right\|_{L^2}^2 + \left\|\underline{\mathbf{e}}^n_\sigma\right\|_e^2 \rightarrow 0
$ as $n \rightarrow \infty$.
\end{proof}

\subsection{Convergence for Algorithm \ref{alg:A}}
\label{proof for standard cem}
In this section, we prove Theorem \ref{convergence in standard cem}, which confirms the convergence of Algorithm \ref{alg:A}.
\begin{proof}[proof of Theorem \ref{convergence in standard cem}]
Combining Stages b–c in Algorithm \ref{alg:A} with the subproblems in Theorem \ref{equivalent discrete form},
\begin{align}\label{subtracting_for_standard_cem}
 a_1\bigl(\mathbf{u}_1^n - \mathbf{u}_{h,1}, \mathbf{v}\bigr) 
   + \alpha \int_\gamma^* \bigl(\mathbf{u}_1^n - \mathbf{u}_{h,1}\bigr) \cdot \mathbf{v} \, ds 
   = \int_\gamma^* (\mathbf{g}_{12}^n - \mathbf{g}_{12}^*) \cdot \mathbf{v} \, ds,\quad \forall \mathbf{v}\in\mathbf{V}_{\text{ms}}.
\end{align}
We observe that (\ref{subtracting_for_standard_cem}) has a form similar to (\ref{energy_estimate_1}), with only replacement of the spaces. Moreover, since the formulation in $\Omega_2$ is standard, we have the same one as (\ref{energy_estimate_2}). Denote the sequence $\{\epsilon_n\}$ with
\[
 \epsilon_n = \left(\sum_{i=1}^2 \left\|\mathbf{u}_i^n - \mathbf{u}_{h,i}\right\|_{a_i}^2 
 + \frac{1}{\delta} \int_{\Gamma_C} \bigl[(u_{2,c}^n)^+ - (u_{hc,2})^+\bigr] \bigl(u_{2,c}^n - u_{h,2}\bigr) \, ds\right)^{\frac{1}{2}}.
\] 
Following the similar steps to the proof of Theorem \ref{convergence in discrete setting}, we conclude that \(( \sum_{i=1}^2 \left\|\mathbf{u}_i^n - \mathbf{u}_{h,i}\right\|_{a_i}^2 \) \(+ \frac{1}{\delta} \int_{\Gamma_C} \bigl[(u_{2,c}^n)^+ - (u_{hc,2})^+\bigr]^2)^{\frac{1}{2}} \leq \epsilon_n \to 0\) as \(n \to \infty\), thereby establishing the desired result.
\end{proof}
\subsection{Convergence for Algorithm \ref{alg:BB}}
\label{proof for mixed cem}
In this section, we prove Theorem \ref{convergence in mixed cem}, which presents the convergence of Algorithm \ref{alg:BB}.
\begin{proof}[proof of Theorem \ref{convergence in mixed cem}]
Combining the procedures in Stages b–c in Algorithm \ref{alg:BB} with the reference problem (\ref{penalty_mixed_discrete_formula_seperate_1}), we can obtain formulations of a form analogous to (\ref{eq:3.2a})–(\ref{eq:3.2b}), with the spaces $\underline{\bm{\mathit{\Sigma}}}_{h,1},\bm{\mathit{U}}_{h,1}$ replaced by $\Sigma_\text{ms}, U_{\text{aux}}$ for the stress and displacement. Next, following the similar techniques in Section \ref{proof for mixed}, we have (using the notation defined in Section \ref{iterative contact-resolving hybrid method for the mixed discretization 2}):
\(
\left\|\underline{\mathbf{e}}^n_\sigma\right\|_e^2 \leq \left\|\underline{\mathbf{e}}^n_\sigma\right\|_{\mathcal{A}}^2 + \frac{1}{\delta}\int_{\Gamma_C} ((\sigma^n_{2,c})^+ - (\sigma_{hc,2})^+)(\sigma^n_{2,c} - \sigma_{hc,2}) \, ds
\).
Besides, there exists $C>0$ independent of material parameters and mesh size such that $\left\|\mathbf{e}^n_u\right\|_{L^2}^2\leq C \left\|\underline{\mathbf{e}}^n_\sigma\right\|_e^2$. Choosing a sequence $\{\xi_n\}$ with
$$\xi_n=\left(1+C\right)\left(\left\|\underline{\mathbf{e}}^n_\sigma\right\|_{\mathcal{A}}^2 + \frac{1}{\delta}\int_{\Gamma_C} \left((\sigma^n_{2,c})^+ - (\sigma_{hc,2})^+\right)\left(\sigma^n_{2,c} - \sigma_{hc,2}\right) \, ds\right).$$
 Then similar to the proof of Theorem \ref{convergence for mixed}, we have \( \xi_n \to 0\) as \(n \to \infty\). This completes the proof.
\end{proof}

\section{Numerical experiments}\label{Numerical experiments}
In this section, we conduct numerical experiments to demonstrate the performance of the proposed iterative contact-resolving hybrid methods. The demonstration spans from the use of mixed FEM across both subdomains to a combination of mixed multiscale methods in the larger domain and mixed FEM in the smaller domain. All computations were performed using MATLAB 2021a on a Lenovo ThinkCentre M80q Gen 4 desktop computer equipped with an Intel Core~i9-13900T processor and 32~GB of RAM. 

The computational domain is set as $\Omega = [0,1]^2$, with the subdomains defined by $\Omega_1 = [0,1-H] \times [0,1]$ and $\Omega_2 = [1-H,1] \times [0,1]$. We give two test models in $\Omega_1$ as shown in Figure \ref{fig:heterogeneity_pattern}, where the Young's modulus $E$ is taken as a piecewise constant with heterogeneity pattern. For clearer, we take the yellow part as $E_1$ and the blue part as $E_2$. The Poisson ratio $\nu_1$ (yellow region) and $\nu_2$ (blue region) are both chosen as 0.35 in Sections \ref{test 1}-\ref{numerical results for DDM-CEM}. To simulate nearly incompressible materials, we will assign 0.49/0.499/0.4999 to either $\nu_1$ or $\nu_2$ in Section \ref{numerical results for nearly incompressible materials}. For simplicity, the Poisson's ratio and Young's modulus in $\Omega_2$ are fixed at $E_2$ and $\nu_2$, respectively, throughout all numerical experiments.
The Lam$\rm \acute{e}$ coefficients $\lambda,\mu$ are denoted as \begin{equation*}
	\label{Lame constants}
	\lambda\coloneqq\frac{E\nu}{(1+\nu)(1-2\nu)},\quad  \mu\coloneqq\frac{E}{2(1+\nu)}.
\end{equation*} 
Let \((\underline{\bm{\sigma}}_h, \mathbf{u}_h) \in \underline{\bm{\Sigma}}_h \times \mathbf{U}_h\) be the solution of (\ref{penalty_mixed_discrete_formula}) on the entire domain, and set \(\underline{\bm{\sigma}}_{h,i} = \underline{\bm{\sigma}}_h|_{\Omega_i}\) and \(\mathbf{u}_{h,i} = \mathbf{u}_h|_{\Omega_i}\) for \(i = 1,2\). We write \(\underline{\bm{\sigma}}_h = (\underline{\bm{\sigma}}_{h,1}, \underline{\bm{\sigma}}_{h,2})\) and \(\mathbf{u}_h = (\mathbf{u}_{h,1}, \mathbf{u}_{h,2})\), and treat \((\underline{\bm{\sigma}}_h, \mathbf{u}_h)\) as the reference solution in our numerical simulations.
Define \(\underline{\bm{\sigma}}^n = (\underline{\bm{\sigma}}^n_1, \underline{\bm{\sigma}}^n_2)\) and \(\mathbf{u}^n = (\mathbf{u}^n_1, \mathbf{u}^n_2)\) as the iterative solutions at the final iteration \(n\) obtained by the proposed iterative contact-resolving hybrid algorithms (Algorithms \ref{an iterative dd algorithm for continuous setting in mixed form}/\ref{alg:BB}).
In order to evaluate the accuracy of the iteration solution (at the last iteration time $n$), we use the relative errors:
$
e_{\sigma}^r=\frac{\norm{\underline{\bm{\sigma}}_h-\underline{\bm{\sigma}}^n}_\mathcal{A}}{\norm{\underline{\bm{\sigma}}_h}_{\mathcal{A}}},$ $ e_u^r = \frac{\norm{\mathbf{u}_h-\mathbf{u}^n}_{L^2}}{\norm{\mathbf{u}_h}_{L^2}}
$ (where the norm notation defined in Section~\ref{iterative contact-resolving hybrid method for the mixed discretization 2} are used).
The stopping criterion is: 
\[\max(\frac{\norm{\underline{\bm{\sigma}}^n-\underline{\bm{\sigma}}^{n-1}}_\mathcal{A}}{\norm{\underline{\bm{\sigma}}^n}_{\mathcal{A}}}, \frac{\norm{\mathbf{u}^n-\mathbf{u}^{n-1}}_{L^2}}{\norm{\mathbf{u}^n}_{L^2}})\leq 10^{-6}.\] Here \(\max(\frac{\norm{\underline{\bm{\sigma}}^n-\underline{\bm{\sigma}}^{n-1}}_\mathcal{A}}{\norm{\underline{\bm{\sigma}}^n}_{\mathcal{A}}}, \frac{\norm{\mathbf{u}^n-\mathbf{u}^{n-1}}_{L^2}}{\norm{\mathbf{u}^n}_{L^2}})\) is called relative tolerance. 
The initial values $\mathbf{g}^0_{12}$ and $\mathbf{g}^0_{21}$ can be chosen arbitrarily; they do not affect the final convergence but only the number of iterations. Without loss of generality, we fix them as $(0,0)^{\mathrm{T}}$ for all iterations. This facilitates a consistent comparison of convergence behavior under different transmission coefficients. The transmission coefficient $\beta$ is set to 1 or $\sqrt{h}$, inspired by some existing literature \cite{lions1990schwarz,tang1992generalized,deng1997timely, deng2003optimal, deng2003nonoverlapping, xu2010spectral, qin2006parallel,qin2008optimized}.

In Sections~\ref{test 1} and~\ref{test 2}, we present the iteration behaviors, relative tolerance, and relative errors for both the stress and displacement fields using piecewise constant functions $\mathbf{f}$ in Test Models 1 and 2, respectively. The components of the stress and displacement fields, along with their values on the contact boundary $\Gamma_C$, are displayed. In Section~\ref{numerical results for DDM-CEM}, we report the relative errors and iteration numbers for stress and displacement as the layer extension varies within the iterative framework associated with CEM-GMsFEM. Section~\ref{numerical results for nearly incompressible materials} investigates the robustness of the iteration behavior for nearly incompressible materials. Finally, we report the computational costs in Section \ref{times}. All numerical experiments in Sections~\ref{test 1}--\ref{times} approximate the penalty problem (\ref{model_problem_update_mixed_form}), with the boundary conditions (\ref{model_problem_update_mixed_form_b})--(\ref{model_problem_update_mixed_form_c}). The Dirichlet boundary is defined as
\(
\Gamma_D \coloneqq \{(x,y): x=0,\; 0\le y\le 1\} \cup \{(x,y): y=0,\; 0\le x\le 1\} \cup \{(x,y): y=1,\; 0\le x\le 1\},
\)
and the contact boundary is
\(
\Gamma_C \coloneqq \{(x,y): x=1,\; 0\le y\le 1\}.
\)
Solutions from the test models are further verified to satisfy (\ref{model_problem_b})-(\ref{model_problem_c}), as illustrated in Figures \ref{discrete_displacement_model_1}-\ref{contact values_model_1} and \ref{discrete_displacement_2}-\ref{contact values_2}. We performed post-processing on the computed relative errors to remove Dirichlet–Robin boundary intersection effects.

\begin{figure}[tbph] 
		\centering 
		\includegraphics[height=6cm,width=14cm]{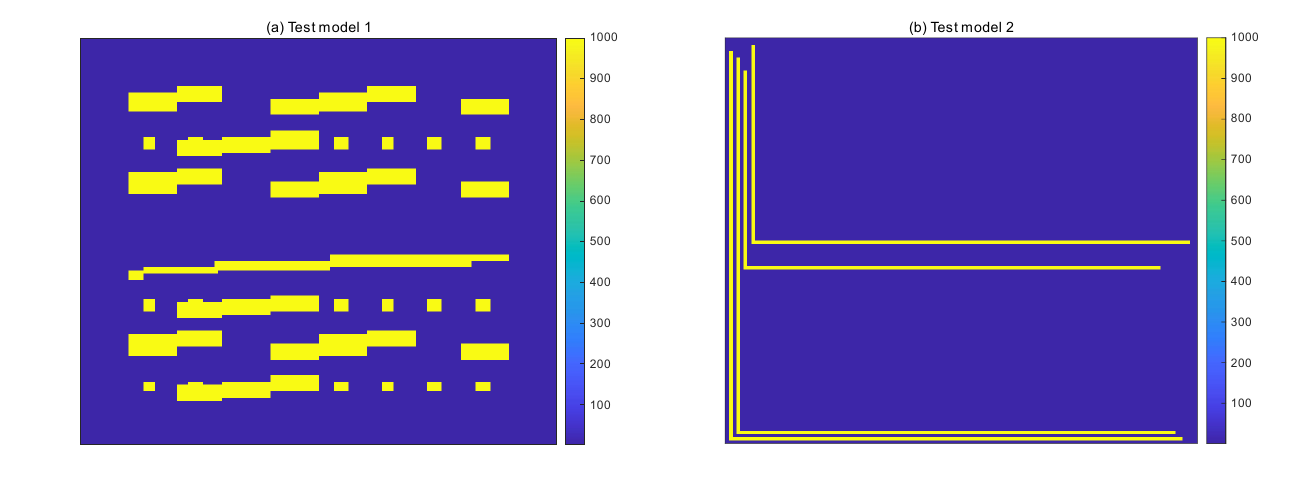} 
		\caption{Young’s Modulus of the test models in $\Omega_1$}
		\label{fig:heterogeneity_pattern}
\end{figure}
\subsection{Test model 1}\label{test 1}
In the first numerical experiment, we test the iterative contact-resolving hybrid framework associated with mixed method (i.e. Algorithm \ref{an iterative dd algorithm for continuous setting in mixed form}). We consider Test Model~1 illustrated in Figure~\ref{fig:heterogeneity_pattern}. Poisson's ratios are set equal with $\nu_1 = \nu_2 = 0.35$, while Young's moduli follow a contrast ratio with $E_2 = 1$ fixed and $E_1=10^3$ in $\Omega_1$.
We take $\mathbf{f}=(f_1,f_2)$ with $f_2(x,y)=0$ and $f_1(x,y)$ expressed as

\[ f(x,y) = 
\begin{cases}
-\frac{1}{2}, & \text{if } \frac{7}{8} \leq x \leq 1 \text{ and } \frac{1}{8} \leq y \leq \frac{1}{2} \\
1, & \text{if } \frac{7}{8} \leq x \leq 1 \text{ and } \frac{5}{8} \leq y \leq \frac{7}{8} \\
0, & \text{othewise}.
\end{cases}
\]

We present the number of iterations, relative tolerance, and relative errors for the stress and displacement fields with different $\beta$. The results in Table \ref{Relative errors of Test model 1} demonstrate that the relative errors for both the stress and displacement fields reach approximately the \(10^{-4}\) level across different values of \(\beta\) under fixed mesh sizes $h=1/64$ or $h=1/128$. For instance, at $h=1/128$, the errors are 4.32e-04 (stress) and 1.94e-04 (displacement) when $\beta=\sqrt{h}$, and 4.33e-04 (stress) and 1.43e-04 (displacement) when $\beta=1$. Starting from the same initial guess, i.e., \(\mathbf{g}^0_{12} = \mathbf{g}^0_{21} = (0,0)^{\mathrm{T}}\), the number of iterations required for \(\beta = \sqrt{h}\) is notably lower than for \(\beta = 1\) under fixed $h$. This is visualized in Figure~\ref{iterations for model 1}:
when \(h = 1/64\), varying $\beta$ from 1 (blue curve) to \(\sqrt{h}\) (green curve) accelerates convergence, requiring only about 20--30 iterations to satisfy the stopping criterion. Similarly, fixing \(h = 1/128\): for \(\beta=\sqrt{h}\) (purple curve), only about 40--50 iterations are needed to achieve the same accuracy, while for \(\beta=1\) (red curve), convergence slows down, taking more than 90 iterations. As clearly shown in Figure~\ref{iterations for model 1}, adopting \(\beta = \sqrt{h}\) leads to faster convergence compared to \(\beta = 1\) under a fixed $h$.

A comparison of the stress field components is presented in Figure~\ref{discrete_stress_model_1}. The reference solution (panels (a)-(c)) and our iterative solution with $h=1/64$, $\beta=1$ (panels (d)-(f)) appear to be consistent, validating the method's accuracy. Visual agreement can also be observed for the displacement field in Figure~\ref{discrete_displacement_model_1}. Furthermore, the behavior of the normal stress and displacement components on the contact boundary $\Gamma_C$ is detailed in Figure \ref{contact values_model_1}. The results consistently satisfy the contact condition (\ref{model_problem_c}): $\bm{u}^n \cdot \mathbf{n}_c = 0$ where $\mathbf{n}_c \cdot (\underline{\bm{\sigma}}^n \cdot \mathbf{n}_c) < 0$, and $\mathbf{n}_c \cdot (\underline{\bm{\sigma}}^n \cdot \mathbf{n}_c) = 0$ where $\bm{u}^n \cdot \mathbf{n}_c < 0$. The active contact set, where $\bm{u}^n \cdot \mathbf{n}_c = 0$ and $\mathbf{n}_c \cdot (\underline{\bm{\sigma}}^n \cdot \mathbf{n}_c) < 0$, spans approximately $y\in[0,0.5]$ on $\Gamma_C$.

\begin{table}
	\footnotesize
	\centering
	\caption{Iterations and errors with different $\beta$ for Algorithm \ref{an iterative dd algorithm for continuous setting in mixed form} in Test model 1}
	\label{Relative errors of Test model 1}
	\begin{tabular}{ccccccc}
		\hline
		$h$ & $\beta$ & Iterations & Relative tolerance & $e_{\sigma}^r$ & $e_u^r$\\
		\hline
\multirow{2}*{1/64} &1 & 69	& 8.46e-06 &1.49e-04 &1.18e-04\\
 &$\sqrt{h}$ & 26 & 7.01e-06 &1.51e-04 &1.19e-04\\ \hline
\multirow{2}*{1/128} &1 & 94  & 9.61e-06	&4.33e-04	&1.43e-04\\
 &$\sqrt{h}$ & 43 & 8.37e-06 &4.32e-04 &1.94e-04
		\\
		\hline
	\end{tabular}
\end{table}


\begin{figure}[tbph] 
		\centering 
		\includegraphics[height=8cm,width=15cm]{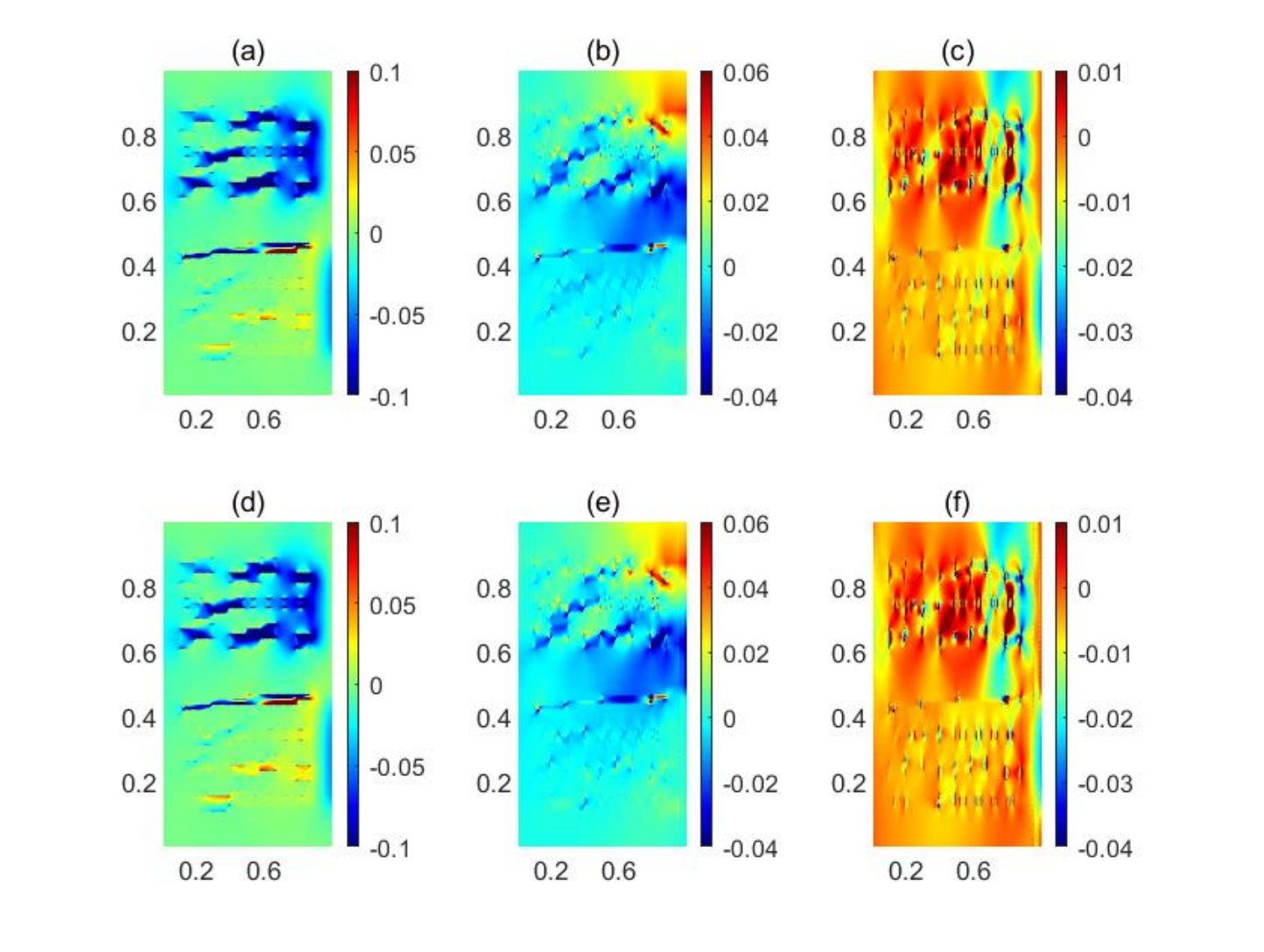} 
		\caption{(a)-(c): components of stress reference solution in Test model 1 (i.e. $(\underline{\bm{\sigma}}_h)_{11},(\underline{\bm{\sigma}}_h)_{12},(\underline{\bm{\sigma}}_h)_{22}$); (d)-(f): components of final stress iteration solution for Algorithm \ref{an iterative dd algorithm for continuous setting in mixed form} in Test model 1 (i.e. $(\underline{\bm{\sigma}}^n)_{11},(\underline{\bm{\sigma}}^n)_{12},(\underline{\bm{\sigma}}^n)_{22}$) with $h=1/64$, $\beta=1$.}
		\label{discrete_stress_model_1}
\end{figure}
\begin{figure}[tbph] 
		\centering 
		\includegraphics[height=8cm,width=10cm]{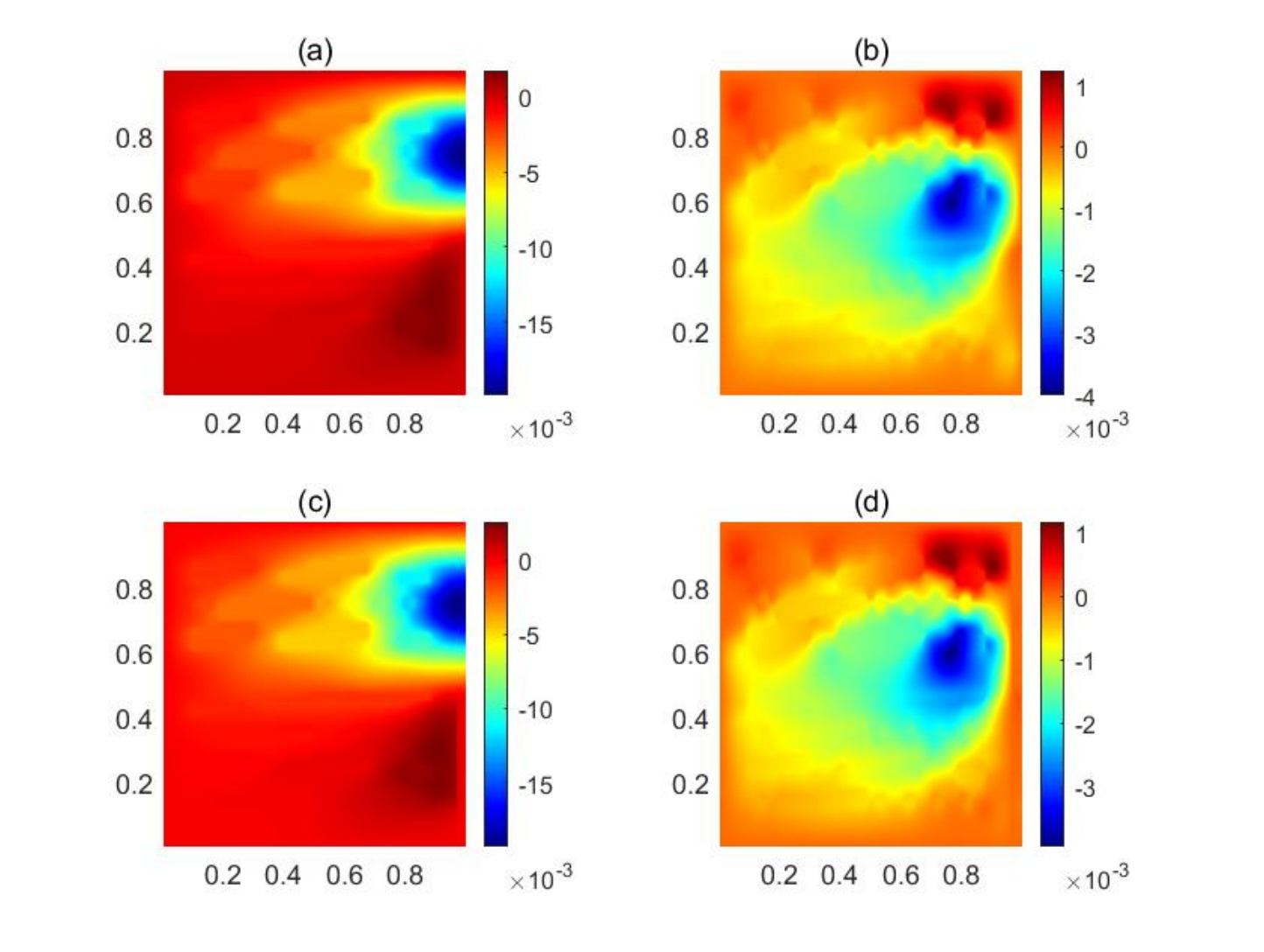} 
		\caption{(a)-(b): components of displacement reference solution in Test model 1 (i.e. $(\mathbf{u}_h)_1,(\mathbf{u}_h)_2$); (c)-(d): components of final displacement iteration solution for  Algorithm \ref{an iterative dd algorithm for continuous setting in mixed form} in Test model 1 (i.e. $(\mathbf{u}^n)_1,(\mathbf{u}^n)_2$) with $h=1/64$, $\beta=1$}
		\label{discrete_displacement_model_1}
\end{figure}

\begin{figure}[tbph] 
		\centering 
		\includegraphics[height=6cm,width=9cm]{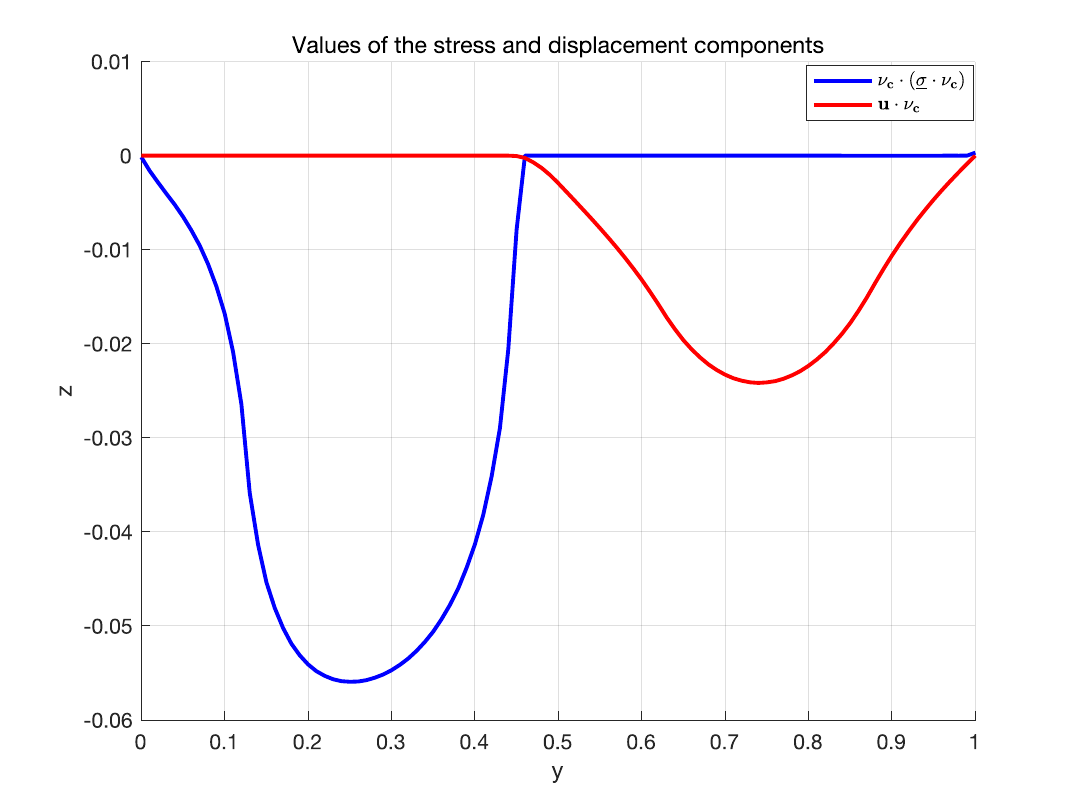} 
		\caption{Contact values of final stress and displacement iteration solutions for  Algorithm \ref{an iterative dd algorithm for continuous setting in mixed form} in Test model 1}
		\label{contact values_model_1}
\end{figure}

\begin{figure}[tbph] 
		\centering 
		\includegraphics[height=6.5cm,width=10cm]{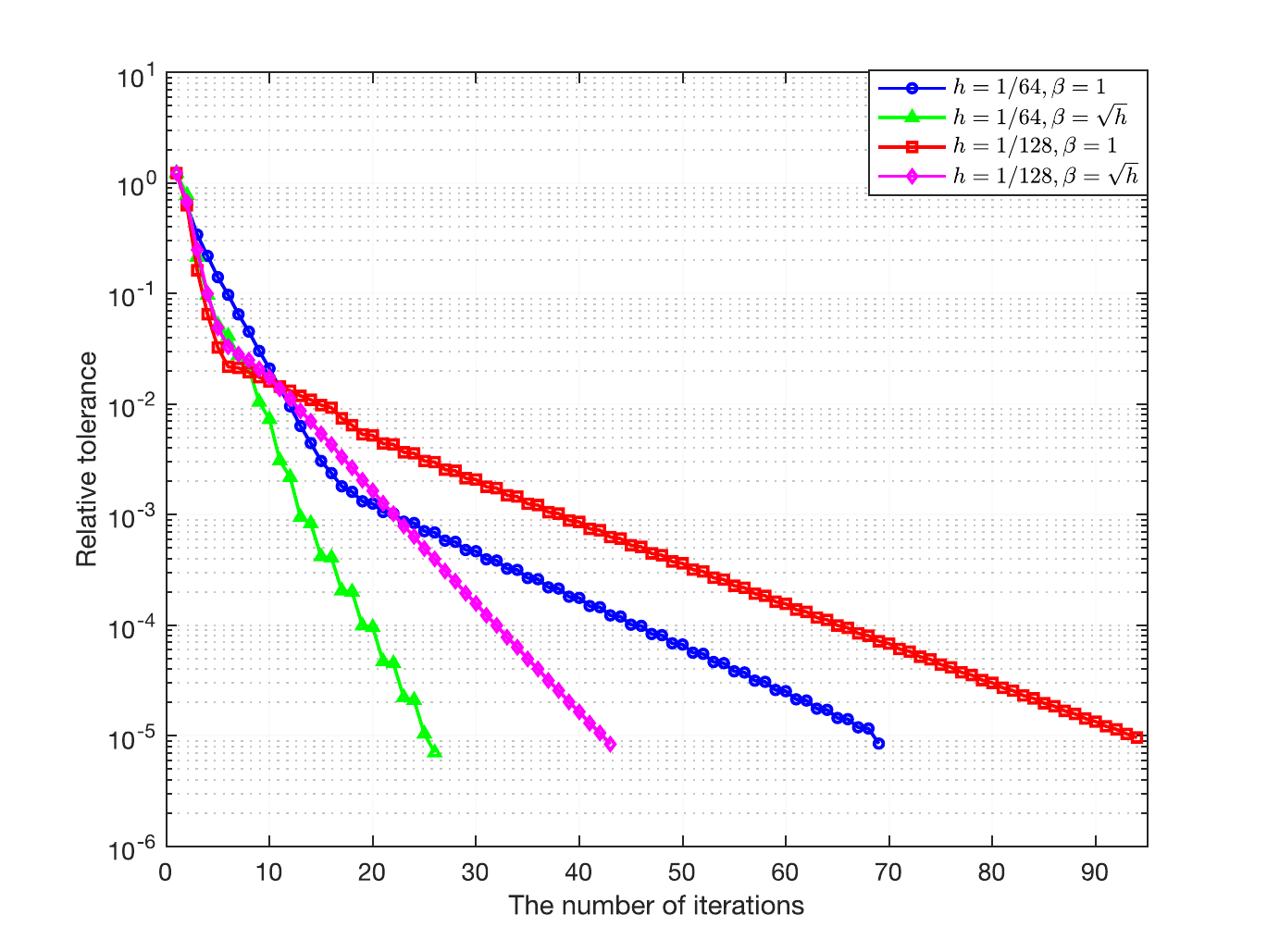} 
		\caption{Iterations with different values of $\beta$ for  Algorithm \ref{an iterative dd algorithm for continuous setting in mixed form} in Test model 1}
		\label{iterations for model 1}
\end{figure}

\subsection{Test model 2}\label{test 2}
The second numerical experiment employs Test Model 2 (Figure \ref{fig:heterogeneity_pattern}) to evaluate Algorithm \ref{an iterative dd algorithm for continuous setting in mixed form}. The material parameters are specified with $\nu_1 = \nu_2 = 0.35$ and $E_1=10^3, E_2 = 1$. We take $\mathbf{f}=(f_1,f_2)$ with $f_2(x,y)=0$ and $f_1(x,y)$ expressed as
\[ f(x,y) = 
\begin{cases} 
1/2, & \text{if } \frac{7}{8} \leq x \leq 1 \text{ and } \frac{1}{8} \leq y \leq \frac{1}{4} \\
-1/4, & \text{if } \frac{7}{8} \leq x \leq 1 \text{ and } \frac{3}{8} \leq y \leq \frac{5}{8} \\
1/2, & \text{if } \frac{7}{8} \leq x \leq 1 \text{ and } \frac{7}{8} \leq y \leq 1 \\
0, & \text{othewise}.
\end{cases}
\]

The number of iterations, relative tolerance, and relative errors in the stress and displacement fields with different $\beta$ (under fixed $h=1/64$ or $h=1/128$) are presented in Table \ref{Relative errors of Test model 2}. 
We observe that the relative errors for both the stress and displacement fields attain approximately the $10^{-4}$ level across different $\beta$ under fixed $h$. For instance, at $h=1/64$, the errors are 3.65e-04 (stress) and 2.98e-04 (displacement) when $\beta=\sqrt{h}$, and 3.64e-04 (stress) and 2.98e-04 (displacement) when $\beta=1$. Starting from the same initial guess, convergence under $\beta = \sqrt{h}$ requires fewer iterations than under $\beta = 1$ with fixed $h$. This confirms that the choice of $\beta$ substantially affects convergence speed. The number of iterations reported in Table~\ref{Relative errors of Test model 2} are visualized in Figure~\ref{iterations_2}: when \(h = 1/64\), varying $\beta$ from \(\sqrt{h}\) (green curve) to 1 (blue curve) slows down convergence, requiring about 50 iterations to satisfy the stopping criterion. Similarly, fixing \(h = 1/128\): for \(\beta=\sqrt{h}\) (purple curve), only about 40--50 iterations are needed to reach the stopping criterion, while for \(\beta=1\) (red curve), convergence slows down, taking more than 60 iterations. As clearly shown in Figure~\ref{iterations_2}, adopting \(\beta = \sqrt{h}\) leads to faster convergence compared to \(\beta = 1\) under a fixed $h$.

Visual comparison of solution components provides additional validation of the method's accuracy. Figure~\ref{discrete_stress_2} (a)-(c) display the reference stress solution components, while (d)-(f) present the iteration approximation. The agreement between these solutions is visually consistent. Similar agreement is observed for the displacement field in Figure~\ref{discrete_displacement_2}. 
Moreover, our observations in Figure \ref{contact values_2} confirm the contact boundary condition (\ref{model_problem_c}): when $\mathbf{n}_c \cdot (\underline{\bm{\sigma}}^n \cdot \mathbf{n}_c) < 0$, we have $\bm{u}^n \cdot \mathbf{n}_c = 0$; conversely, $\mathbf{n}_c \cdot (\underline{\bm{\sigma}}^n \cdot \mathbf{n}_c) = 0$ when $\bm{u}^n \cdot \mathbf{n}_c < 0$; The active contact set spans about $y\in[0.4,0.6]$ on $\Gamma_C$.

\begin{table}
	\footnotesize
	\centering
	\caption{Iterations and errors with different $\beta$ for Algorithm \ref{an iterative dd algorithm for continuous setting in mixed form} in Test model 2}
	\label{Relative errors of Test model 2}
	\begin{tabular}{cccccc}
		\hline
		$h$ & $\beta$ & Iterations & Relative tolerance & $e_{\sigma}^r$ & $e_u^r$\\
		\hline
\multirow{2}*{1/64} &1	& 48	&9.34e-06	& 3.64e-04 &2.98e-04\\
 &$\sqrt{h}$	& 27	&7.20e-06	&3.65e-04 &2.98e-04\\ \hline
\multirow{2}*{1/128} &1 & 66 &9.54e-06	&4.46e-04 & 3.26e-04
		\\
 &$\sqrt{h}$ & 42 &9.15e-06	&4.47e-04 &3.14e-04
		\\
		\hline
	\end{tabular}
\end{table}
\begin{figure}[tbph] 
		\centering 
		\includegraphics[height=8cm,width=15cm]{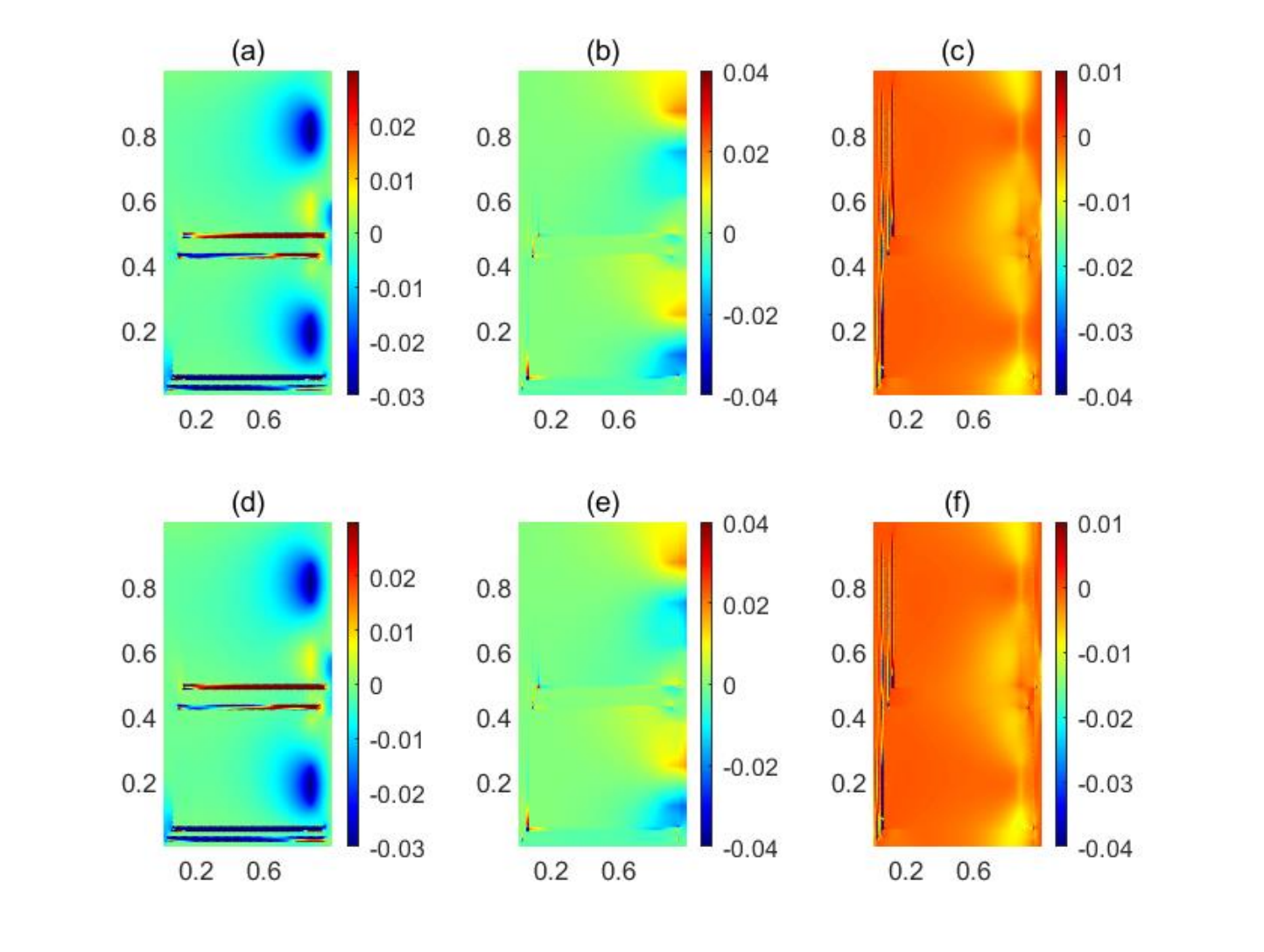} 
		\caption{(a)-(c): components of stress reference solution in Test model 2 (i.e. $(\underline{\bm{\sigma}}_h)_{11},(\underline{\bm{\sigma}}_h)_{12},(\underline{\bm{\sigma}}_h)_{22}$); (d)-(f): components of final stress iteration solution for Algorithm \ref{an iterative dd algorithm for continuous setting in mixed form} in Test model 2 (i.e. $(\underline{\bm{\sigma}}^n)_{11},(\underline{\bm{\sigma}}^n)_{12},(\underline{\bm{\sigma}}^n)_{22}$) with $h=1/64$, $\beta=1$.}
		\label{discrete_stress_2}
\end{figure}
\begin{figure}[tbph] 
		\centering 
		\includegraphics[height=8cm,width=10cm]{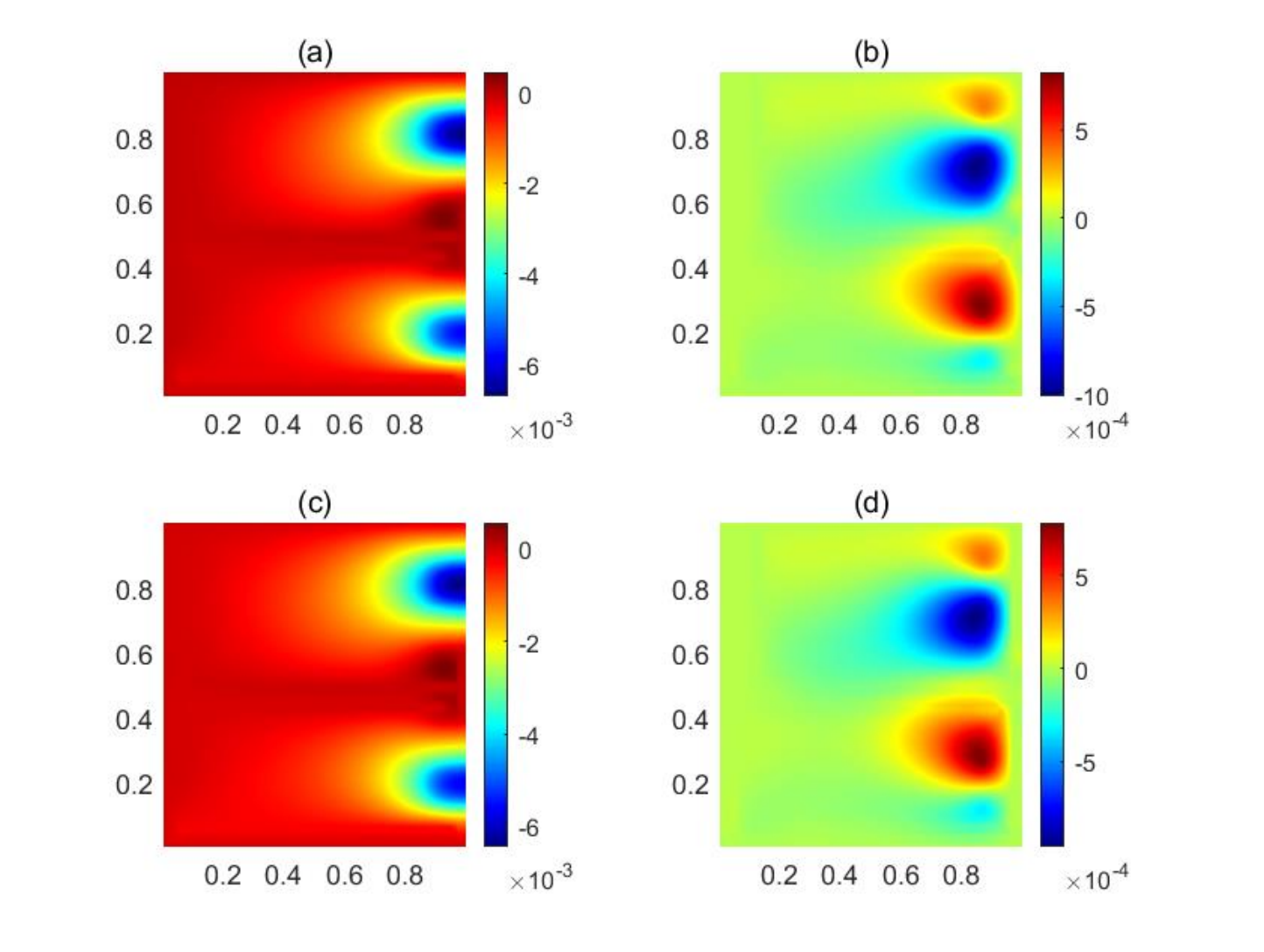} 
		\caption{(a)-(b): components of displacement reference solution in Test model 2 (i.e. $(\mathbf{u}_h)_1,(\mathbf{u}_h)_2$); (c)-(d): components of final displacement iteration solution for Algorithm \ref{an iterative dd algorithm for continuous setting in mixed form} in Test model 2 (i.e. $(\mathbf{u}^n)_1,(\mathbf{u}^n)_2$) with $h=1/64$, $\beta=1$}
		\label{discrete_displacement_2}
\end{figure}

\begin{figure}[tbph] 
		\centering 
		\includegraphics[height=6cm,width=9cm]{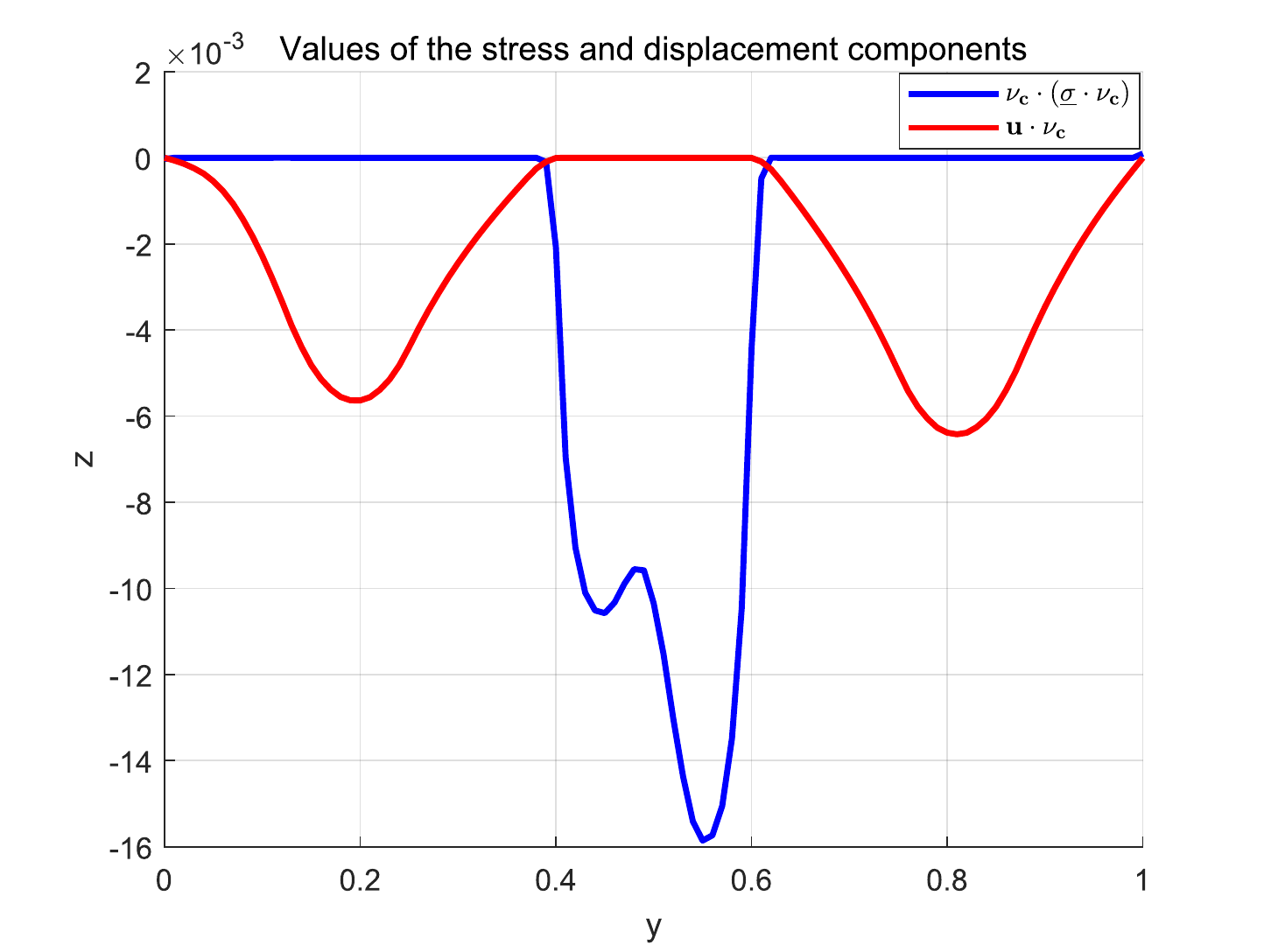} 
		\caption{Contact values of final stress and displacement iteration solutions for Algorithm \ref{an iterative dd algorithm for continuous setting in mixed form} in Test model 2}
		\label{contact values_2}
\end{figure}

\begin{figure}[tbph] 
		\centering 
		\includegraphics[height=6.5cm,width=10cm]{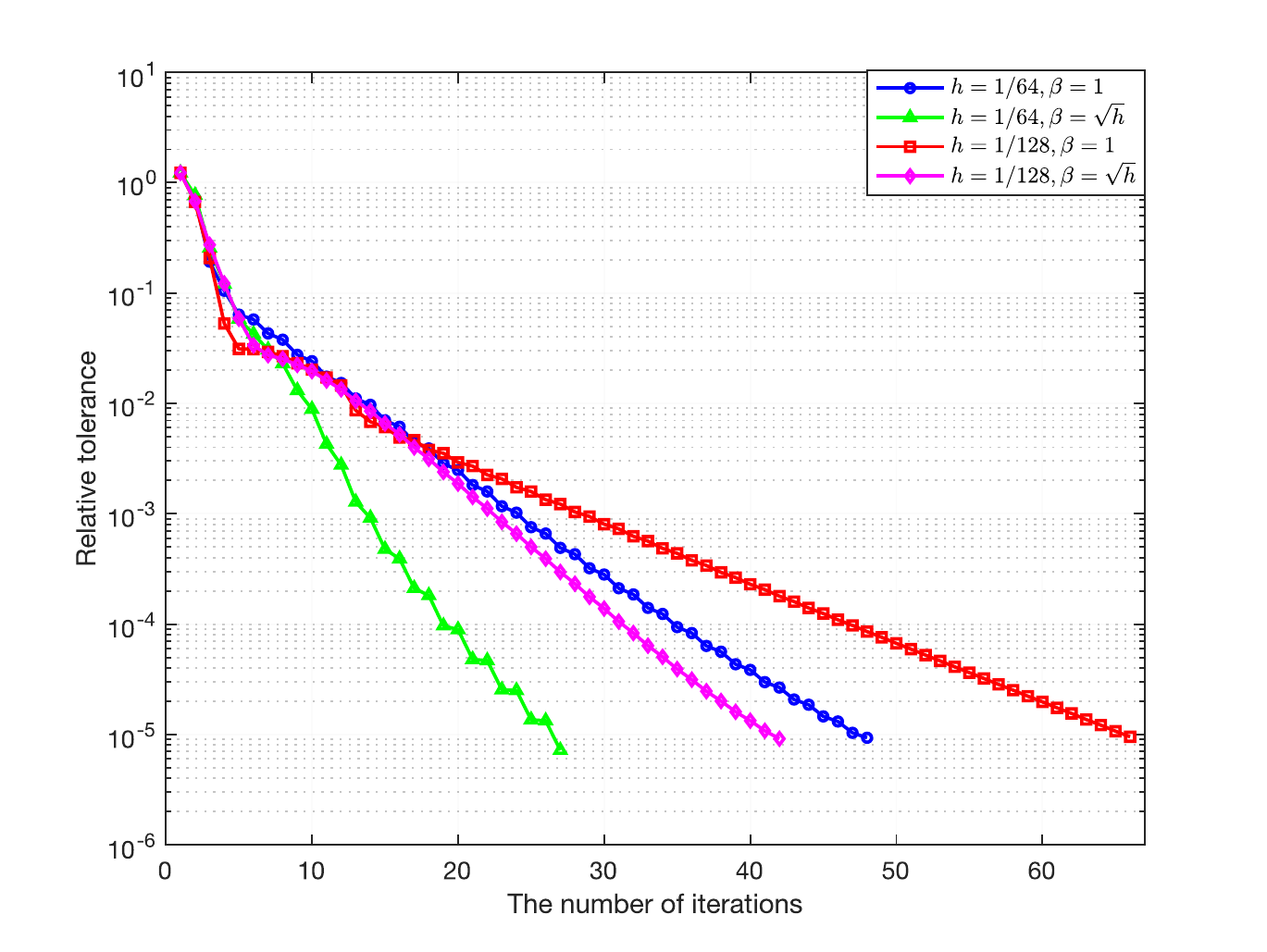} 
		\caption{Iterations with different values of $\beta$ for Algorithm \ref{an iterative dd algorithm for continuous setting in mixed form} in Test model 2}
		\label{iterations_2}
\end{figure}

\subsection{Test for the iterative framework associated with CEM-GMsFEM} \label{numerical results for DDM-CEM}
In this section, we test the iterative contact-resolving hybrid framework associated with mixed CEM-GMsFEM (Algorithm \ref{alg:BB}); the simulation uses Test Model 2 with all material parameters ($\nu_1$, $\nu_2$, $E_1$, $E_2$) and the source term $f$ consistent with Section \ref{test 2}. We fix $h=1/64$ and let $\beta=\sqrt{h}$, $H=1/16$. We assign three multiscale basis functions per coarse element.
Recall the definition in Section \ref{standard cem}: $K_{i,m}$ is defined as the oversampling region by extending $K_i$ by $m$ coarse mesh layers. Then we denote ``the number of layer extension" as ``$osly$", which equals $m$ for the region $K_{i,m}$.

Table~\ref{Relative errors for the stress of Test model 2 within CEM-GMsFEM} presents the number of iterations, relative tolerance, and relative errors for the stress and displacement fields with different $osly$. The key findings are: (i) The relative errors decrease as $osly$ increases. (ii) For $osly = 3$, the relative errors for stress and displacement reach the \(10^{-2}\) level (i.e., 5.14e-02, 3.09e-02 respectively). Increasing $osly$ to $5$ further reduces the errors to the $10^{-4}$ level for both stress and displacement (i.e. 8.94e-04, 6.86e-04 respectively). (iii) As $osly$ increases from 1 to 5, the number of iterations decreases from 96 to 43, demonstrating that larger oversampling regions substantially improve the convergence speed of the iterative framework. These results confirm the accuracy and efficiency of the proposed method associated with CEM-GMsFEM.

\begin{table}[htbp]
	\footnotesize
	\centering
	\caption{Iterations and errors with different $osly$ for Algorithm \ref{alg:BB} in Test model 2}
	\label{Relative errors for the stress of Test model 2 within CEM-GMsFEM}
	\begin{tabular}{cccccc}
		\hline
		$osly$ & Iterations & Relative tolerance & $e_{\sigma}^r$ & $e_u^r$\\
		\hline
1 & 96 & 8.64e-06 & 2.16e-01 & 1.81e-01\\
2 & 71 & 9.75e-06 & 1.23e-01 & 6.37e-02\\
3 & 58 & 7.39e-06 & 5.14e-02 & 3.09e-02\\
4 & 46 & 6.93e-06 & 7.26e-03 & 6.29e-03\\
5 & 43 & 8.24e-06 & 8.94e-04 & 6.86e-04\\
		\hline
	\end{tabular}
\end{table}

\subsection{Test for nearly incompressible material}\label{numerical results for nearly incompressible materials}
In this section, we examine the performance of the proposed methods in mixed formulations (Algorithms \ref{an iterative dd algorithm for continuous setting in mixed form} and \ref{alg:BB}) applied to nearly incompressible materials via Test Model~1. We fix $h=1/64$ and let $\beta = \sqrt{h}$, $H = 1/16$. The following three parameter cases are studied:
(i) $E_1/E_2=10000$, $\nu_1=0.49/0.499/0.4999$, $\nu_2=0.35$; (ii) $E_1/E_2=0.0001$, $\nu_1=0.35$, $\nu_2=0.49/0.499/0.4999$;
(iii) $E_1/E_2=1$, $\nu_1=\nu_2=0.49/0.499/0.4999$.

Table~\ref{Relative errors for nearly incompressible material} reveals that the relative errors for both stress and displacement fields uniformly converge to approximately the $10^{-4}$ level (specifically, between 0.0001 and 0.0009) after a number of iterations, thereby confirming the locking robustness of the mixed FEM approach. The robustness of the CEM–GMsFEM‐combined method is further confirmed by Tables~\ref{Relative errors for nearly incompressible material $E_1/E_2=0.0001$}–\ref{Relative errors for nearly incompressible material $E_1/E_2=1$}, where both stress and displacement relative errors reach the $10^{-4}$ level if $osly$ increases to 5.

\begin{table}[htbp]
	\footnotesize
	\centering
	\caption{Iterations and errors with different material parameters for Algorithm \ref{an iterative dd algorithm for continuous setting in mixed form} in Test model 1}
	\label{Relative errors for nearly incompressible material}
	\begin{tabular}{cccccc}
		\hline
		$E_1/E_2$ &$\nu_1,\nu_2$ & Iterations & Relative tolerance & $e_{\sigma}^r$ & $e_u^r$\\
		\hline
       10000 & 0.49,0.35	& 32	& 9.29e-06	& 3.15e-04	& 1.72e-04 \\
       10000 & 0.499,0.35	& 56	& 8.81e-06	& 3.78e-04	& 2.71e-04 \\
       10000 & 0.4999,0.35	& 64	& 6.59e-06	& 5.85e-04	& 3.94e-04 \\
       0.0001	& 0.35,0.49 & 29	& 6.41e-06	& 4.75e-04	& 1.13e-04 \\
0.0001	& 0.35,0.499 & 38	& 6.78e-06	& 6.09e-04	& 1.31e-04 \\
0.0001	& 0.35,0.4999 & 58	& 7.48e-06	& 6.38e-04	& 2.69e-04 \\
1	&0.49,0.49 & 24	& 6.25e-06	& 1.89e-04	& 1.49e-04 \\
1	&0.499,0.499 & 38	& 5.90e-06	& 2.76e-04	& 2.28e-04 \\
1	&0.4999,0.4999 & 57	& 8.27e-06	& 5.52e-04	& 3.74e-04 \\
	\hline
	\end{tabular}
\end{table}

\begin{table}[htbp]
	\footnotesize
	\centering
	\caption{Iterations and errors under different $osly$ for Algorithm \ref{alg:BB} with $E_1/E_2=0.0001,\nu_1=0.35,\nu_2=0.499$  in Test model 1}
	\label{Relative errors for nearly incompressible material $E_1/E_2=0.0001$}
	\begin{tabular}{cccccc}
		\hline
		$osly$ & Iterations & Relative tolerance & $e_{\sigma}^r$ & $e_u^r$\\
		\hline
        1	& 68	& 8.61e-04	& 2.39e-01	& 1.26e-01 \\
2	& 53	& 7.29e-04	& 1.38e-01	& 8.12e-02 \\
3	& 44	& 5.68e-04	& 3.79e-02	& 2.31e-02 \\
4	& 37	& 6.90e-04	& 4.35e-03	& 3.15e-03 \\
5	& 32	& 6.33e-04	& 8.21e-04	& 5.96e-04 \\
	\hline
	\end{tabular}
\end{table}
\begin{table}[htbp]
	\footnotesize
	\centering
	\caption{Iterations and errors under different $osly$ for Algorithm \ref{alg:BB} with $E_1/E_2=10000,\nu_2=0.35,\nu_1=0.499$ in Test model 1}
	\label{Relative errors for nearly incompressible material $E_1/E_2=10000$}
	\begin{tabular}{cccccc}
		\hline
		$osly$ & Iterations & Relative tolerance & $e_{\sigma}^r$ & $e_u^r$\\
		\hline
        1	& 79	& 6.51e-06	& 5.21e-01	& 3.35e-01 \\
2	& 62	& 8.32e-06	& 1.96e-01	& 2.81e-02 \\
3	& 54	& 6.58e-06	& 3.81e-02	& 1.58e-02 \\
4	& 51	& 9.92e-06	& 4.91e-03	& 3.79e-03 \\
5	& 46	& 8.63e-06	& 8.84e-04	& 6.02e-04 \\
	\hline
	\end{tabular}
\end{table}
\begin{table}[htbp]
	\footnotesize
	\centering
	\caption{Iterations and errors under different $osly$ for Algorithm \ref{alg:BB} with $E_1/E_2=1,\nu_1=\nu_2=0.499$ in Test model 1}
	\label{Relative errors for nearly incompressible material $E_1/E_2=1$}
	\begin{tabular}{cccccc}
		\hline
		$osly$ & Iterations & Relative tolerance & $e_{\sigma}^r$ & $e_u^r$\\
		\hline
       1	& 63	& 7.43e-06	& 1.54e-01	& 9.36e-02 \\
2	& 50	& 9.51e-06	& 8.79e-02	& 5.85e-02 \\
3	& 44	& 6.74e-06	& 4.12e-02	& 1.06e-02 \\
4	& 36	& 9.46e-06	& 5.31e-03	& 3.22e-03 \\
5	& 24	& 8.27e-06	& 7.28e-04	& 6.59e-04 \\
	\hline
	\end{tabular}
\end{table}

\subsection{Comparison of computational costs}\label{times}
This section presents the computational performance of three approaches: (i) the reference solution obtained by solving (\ref{penalty_mixed_discrete_formula}) with a semismooth Newton method \cite{hintermuller2002primal}, (ii) the iterative solution using Algorithm~\ref{an iterative dd algorithm for continuous setting in mixed form}, and (iii) the iterative solutions using Algorithm~\ref{alg:BB} with coarse mesh sizes $H = 1/32,1/16,1/8$. We fix $h = 1/64$ and let $\beta = \sqrt{h}$ for all tests.

Recall that all simulations were performed in MATLAB 2021a on a Lenovo ThinkCentre M80q Gen 4 desktop equipped with an Intel Core i9-13900T processor and 32 GB RAM. We employed several computational optimizations, including sparse matrix storage, vectorized operations, and parallel loops (\texttt{parfor}). The reported times cover the entire iterative procedure (to reach the stopping criterion $10^{-6}$), not including some linear solves via the backslash operator (prepared before the loop). 
For Algorithm~\ref{alg:BB}, we only record the online stage (multiscale basis construction is considered offline) and use 3 basis functions per coarse element.
As shown in Table~\ref{times_compare}, Algorithm~\ref{an iterative dd algorithm for continuous setting in mixed form} reduces the computation time from $15.2881\mathrm{s}$ to $6.4762\mathrm{s}$. Further acceleration is achieved by the multiscale technique in Algorithm~\ref{alg:BB}: the solution time drops to $4.9401\mathrm{s}$ for $H=1/32$ and to only $1.6873\mathrm{s}$ for $H=1/8$, while preserving the desired accuracy.
\begin{table}[ht]
    \footnotesize
    \centering
    \caption{Computational costs for different iterative methodologies (stopping criterion: $10^{-6}$).}
    \label{times_compare}
    \begin{tabular}{|c|c|c|c|c|c|}
        \hline
        & \makecell{Reference solution by solving (\ref{penalty_mixed_discrete_formula})\\(using semismooth Newton \cite{hintermuller2002primal})} 
        & \makecell{Iterative solution\\ by Algorithm~\ref{an iterative dd algorithm for continuous setting in mixed form}} 
        & \multicolumn{3}{c|}{\makecell{Iterative solution by Algorithm~\ref{alg:BB}}} \\
        \cline{4-6}
        & & & $H=1/32$ & $H=1/16$ & $H=1/8$ \\
        \hline
        DOFs & 516864 & 516864 & 14892 & 12636 & 12084 \\ \hline
        Times (s) & 15.2881 & 6.4762 & 4.9401 & 2.1341 & 1.6873 \\
        \hline
    \end{tabular}
\end{table}

\section{Conclusions}
\label{conclusions}
In this work, we present an efficient iterative contact-resolving hybrid method tailored for multiscale contact mechanics involving high-contrast material coefficients. The proposed approach localizes the nonlinear contact constraints within a smaller subdomain, while the larger subdomain is described by a linear system. Within this framework, we introduce four distinct discretization strategies. These combine standard (mixed) finite element methods applied over the entire domain, or standard (mixed) multiscale methods in the larger subdomain coupled with standard (mixed) finite element methods in the smaller one. The use of the standard finite element method offers simplicity and ease of implementation. In contrast, the multiscale reduction technique employed in the larger subdomain effectively circumvents the excessive degrees of freedom typically associated with conventional approaches. Furthermore, the incorporation of mixed formulations across the framework enhances robustness against locking effects, even in nearly incompressible material regimes.
Convergence analysis and corresponding algorithms are provided for all proposed schemes. Finally, a comprehensive set of numerical experiments demonstrates the accuracy and robustness of the presented methodology.

\section*{CRediT authorship contribution statement}

\textbf{Eric T. Chung:} Writing – review \& editing, Supervision, Resources, Methodology, Funding acquisition, Conceptualization. \textbf{Hyea Hyun Kim:} Writing – review \& editing, Methodology, Software, Conceptualization. \textbf{Xiang Zhong:} Writing – review \& editing, Writing – original draft, Visualization, Validation, Software, Resources, Methodology, Formal analysis, Data curation, Conceptualization.

\section*{Declaration of competing interest}

The authors declare that they have no known competing financial interests or personal relationships that could have appeared
to influence the work reported in this paper.

\section*{Declaration of Generative AI and AI-assisted technologies in the writing process}

During the preparation of this work the authors used ChatGPT in order to improve readability and language. After
using this tool, the authors reviewed and edited the content as needed and take full responsibility for the content of the
publication.

\section*{Acknowledgments}
Eric T. Chung's work is partially supported by the Hong Kong RGC General Research Fund (Project numbers:
14305222). Hyea Hyun Kim's work is supported by the National Research Foundation of Korea (NRF) grant RS-2025-00516964.

\bibliographystyle{siamplain}

\end{document}